\documentclass{elsart}
\usepackage{etex}
\usepackage{amssymb,amsmath,amscd, mathrsfs, msc}
 \usepackage{epsfig}
 \usepackage[colorlinks=true,backref=page]{hyperref}
 \usepackage{pst-grad} 
 \usepackage{pst-plot} 
 \usepackage{pgf}
\usepackage{float}
\usepackage{tikz}
\usepackage{pgfbaseplot}
\usepackage{booktabs}
\usepackage{subcaption}
\newcommand*\diff{\mathop{}\!\mathrm{d}}
\usepackage{graphicx}
\usepackage{color}

\usepackage{pstricks, pst-node, pst-plot, pst-circ}
\usepackage{moredefs}
\usepackage{array} 
\usepackage{extarrows}
\usepackage[fleqn,tbtags]{mathtools}
\usepackage{mhsetup}
\usepackage{arydshln}
\usepackage{cite}

\usepackage{cases}
\usepackage{verbatim}
\usepackage[a]{esvect}
\usepackage{ntheorem}
\newcommand{\nocontentsline}[3]{}
\newcommand{\tocless}[2]{\bgroup\let\addcontentsline=\nocontentsline#1{#2}\egroup}
\def\d{\mathrm{d}} 

\theoremstyle{plain}

\newtheorem{theorem}{Theorem}[section]

\newtheorem{lemma}[theorem]{Lemma}
\theoremstyle{definition}

\theoremstyle{remark}
\newtheorem{remark}[theorem]{Remark}
\numberwithin{equation}{section} \theoremstyle{corollary}

\newenvironment{proof}[1][Proof]{\begin{trivlist}
\item[\hskip \labelsep {\bfseries #1}]}{\end{trivlist}}

\allowdisplaybreaks
 
\begin{document}
\begin{frontmatter}

\title{Stabilization of turning processes using spindle feedback with state-dependent delay\thanksref{GRN}}
\thanks[GRN]{}

 \author[UTD]{Qingwen Hu}  \ead{qingwen@utdallas.edu} 
 \and\author[UTD]{Huan Zhang}\ead{hxz143430@utdallas.edu} 
 \address[UTD]{Department of Mathematical Sciences,  The University of Texas at Dallas, Richardson, TX, 75080, USA}

\date{}

\maketitle

\begin{abstract}
We develop a stabilization strategy of turning processes by means of delayed spindle control. We show that turning processes which contain intrinsic state-dependent delays can be stabilized by a spindle control with state-dependent delay, and develop analytical  description of the stability region in the parameter space.
Numerical simulations stability region are also given to illustrate the general results.
\end{abstract}

\begin{keyword} Turning processes \sep state-dependent delay\sep    stability\sep spindle speed 
 
\end{keyword}

 
 \end{frontmatter}
\thispagestyle{plain}
 
\section{Introduction} Study of machine-tool chatter generated between the workpiece and the cutting tool in turning processes 
can be traced  back to the work of Taylor  \cite{Taylor}. Since violent chatter severely limits the productivity and product quality of the manufacturing processes and  significantly reduces the tool life, extensive efforts \cite{Tobias-0,Tobias,KT,Altintas,  BST,Balach,Gilsinn,HKT,HKT-1,Insperger-2,Ismail,Long-1,Stone-Sue,Smith-Tlusky,Stepan} till recent years have been contributed to understanding the mechanisms of chatter and to find efficient strategies for suppression of the vibrations. However, the long standing research has proved that the investigation of machine-tool chatter is a highly nontrivial and delicate problem. This is partially due to the complexity interplayed by the configurations of  the  machine which  feeds  and rotates the workpiece, the tool which cuts the workpiece with certain depth and width, and the material properties of workpiece and machining tool including shape and stiffness.  

Delay differential equations were introduced to model machine-tool vibrations as early as the work of  Koenigsberger and  Tlusty~\cite{KT}. It has been observed in the  work \cite{Turi-1, BST} that the time  delay  between two succeeding  cuts depends  on the speed of the workpiece rotation and on the  workpiece  surface generated by the earlier cut. Namely, the time delay is generically state-dependent. This observation has led to modeling the turning processes with state-dependent delay differential equations. We refer to \cite{HKWW} for a review of this type of delay differential equations and refer to \cite{Turi-1} for a state-dependent model of turning processes.

Spindle speed control in the turning process is an important method for suppression of machine-tool chatter and has been extensively investigated in the literature \cite{Pak,Sexton,Insperger-2}. In the work of \cite{HKT,HKT-1} we proposed  a spindle speed control law  for the state-dependent model and determined the stability region where stabilization can be achieved in high speed turning processes. However, such a spindle speed is a real-time control law which depends on the instant status of the system and may be difficult to accomplish for practical realization since usually there must be certain response time to activate the state-dependent control. In this paper, we propose a delayed spindle speed control law based on a state-dependent delay differential equation. Such a delayed spindle control will be dependent on historical state of the system. We are interested whether or not the delayed spindle control is feasible to achieve stability, and if feasible, how to analytically describe the stability region in the parameter space. 

The significance of  analytical descriptions of the stability regions includes that it indicates where in the parameter space the working machine-tool  may achieve stability during the process, and that, if   stability is not achievable, namely, persistent vibrations are unavoidable,  we can predict the vibrations through bifurcation analysis by observing the crossing of the stability boundaries.

We organize the remaining part of the paper as following: In Section~\ref{model}
 we introduce a model for turning processes with state-dependent delay and propose a stabilization of delayed spindle velocity. We show that the resulting nonlinear systems with state-dependent delay can be transformed into a system with constant and distributed delays for which we develop linear stability analysis and analytical description of  the stability region in the parameter space.  In Section~\ref{sec-instant-spindle}, we analyze a model of turning  processes with instantaneous spindle control strategy, for comparison  of the stability region of the corresponding delayed spindle control.  Numerical simulations are provided in Section~\ref{Numerical-simulations} to illustrate the results obtained in Sections~\ref{model} and \ref{sec-instant-spindle}. We   conclude the paper and discuss open problems in Section~\ref{Concluding-remarks}.

\section{Turning processes with delayed state-dependent spindle control}\label{model}

 Our starting point is the turning process as shown in Figure~\ref{Figure1}. The tool is assumed to be compliant and has bending oscillations in directions $x$ and $y$. The governing equations read
\begin{align}
 m\ddot{x}(t)+c_x\dot{x}(t)+k_xx(t)&=F_x,\label{milling-eqn-1-1}\\
  m\ddot{y}(t)+c_y\dot{y}(t)+k_yy(t)&=-F_y.\label{milling-eqn-1-2}
\end{align}
The $x$ and $y$ components of the cutting process force can be written as
\begin{align}
F_x&=K_x\omega h^q,\label{milling-eqn-2-1}\\
F_y&=K_y\omega h^q,\label{milling-eqn-2-2}
\end{align}
where   $m$ is the mass of the tool; $K_x$ and $K_y$ are the cutting coefficients in the $x$ and $y$ directions; 
$k_x$,   $k_y$  are stiffness; 
$c_x$,   $c_y$   are damping coefficients;
$\omega$ is the depth of cut; $h$ is the chip thickness and $q$ is an exponent with empirical value 0.75. The chip thickness $h$ is determined by the feed motion, the current tool position, and the earlier position of the tool and is given as follows
\begin{align}
h(t)=\nu\tau(t)+y(t)-y(t-\tau(t)),\label{milling-eqn-9}
\end{align}where $\nu$ is the speed of the feed.
\begin{figure}[H]
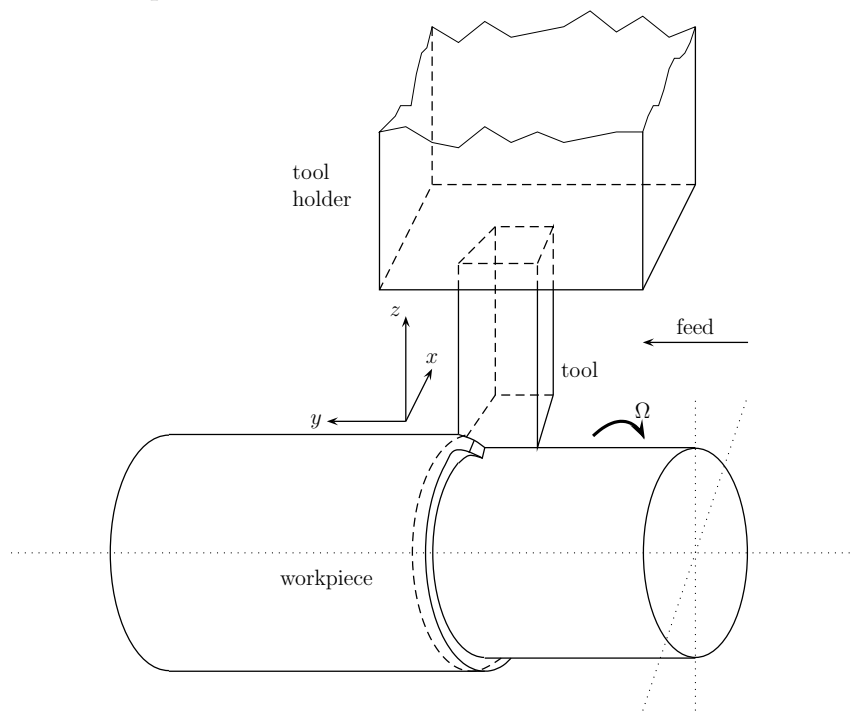

\vskip7cm\hskip10cm
\scalebox{.7}{
\psellipse(0,0)(1,2)
\psellipse(-4,0)(1,2)
\psline[linecolor=white,fillcolor=white,fillstyle=solid](-4,-2)(-4,2)(-3,2)(-3,-2)
\psline[linecolor=white,fillcolor=white,fillstyle=solid](-4.5,2)(-4,2)(-4,1.7)(-4.5,1.7)
\psline(-4,2)(0,2)
\psline(-4,-2)(0,-2)
\pscurve(-4.5, 2.25)
                       (-4.2, 2.139877162)
                       (-4,2)(-4,2)
                         (-4.047644118, 1.8) (-4.047644118, 1.8)
                         (-4.562500000, 1.948557159)
                         (-4.752771932, 1.672075858)
                         (-4.910144119, 1.322516818)
                        (-5.027738640, 0.9151574468)
                        (-5.100416051, 0.4678013045)
                                (-5.125, 0.)
                        (-5.100416051, -0.4678013045)
                        (-5.027738640, -0.9151574468)
                        (-4.910144119, -1.322516818)
                        (-4.752771932, -1.672075858)
                        (-4.562500000, -1.948557159)
                        (-4.347644118, -2.139877162)
                        (-4.117594521, -2.237674265)
                        (-3.882405479, -2.237674265)
                        (-3.652355882, -2.139877162)
                        (-3.5, -2)
\pscurve[linestyle=dashed] (-4.367594521, 2.237674265)
                         (-4.597644118, 2.139877162)
                         (-4.812500000, 1.948557159)
                         (-5.002771932, 1.672075858)
                         (-5.160144119, 1.322516818)
                        (-5.277738640, 0.9151574468)
                        (-5.350416051, 0.4678013045)
                                (-5.375, 0.)
                        (-5.350416051, -0.4678013045)
                        (-5.277738640, -0.9151574468)
                        (-5.160144119, -1.322516818)
                        (-5.002771932, -1.672075858)
                        (-4.812500000, -1.948557159)
                        (-4.597644118, -2.139877162)
                        (-4.367594521, -2.237674265)
                        (-4.132405479, -2.237674265)
                        (-3.902355882, -2.139877162)
                        (-3.687500000, -2)
\psline(-4.2, 2.139877162)(-4.3,1.9)
\pscurve(-4.5,1.7)(-4.3,1.85)(-4.047644118, 1.8)
\psellipse(-10,0)(1.125,2.25)
\psline[linecolor=white,fillcolor=white,fillstyle=solid](-10,2.25)(-8,2.25)(-8,-2.25)(-10,-2.25)
\psline(-10,-2.25)(-4,-2.25)\psline(-10,2.25)(-4.5,2.25)
\psline(-4.5,2.25)(-4.5,5)
\psline(-2.7,3)(-3.,2.)(-3.,5)
\psline[linestyle=dashed](-2.7,3)(-3.8,3)(-4.35,2.2)
\psline(-6,8)(-6,5)(-1,5)(-1,8)
\psline(-1,5)(0,7)(0,10)
\psline(-2.7,3)(-2.7,5)
\psline[linestyle=dashed](0,7)(-5,7)(-6,5)
\psline[linestyle=dashed](-5,10)(-5,7)
\psline[linestyle=dashed](-4.5,5)(-4.5,5.5)(-3.,5.5)(-3.,5)
\psline[linestyle=dashed](-3.,5.5)(-2.7,6.2)(-2.7,5)
\psline[linestyle=dashed](-2.7,6.2)(-3.8,6.2)(-4.5,5.5)
\psline[linestyle=dashed](-3.8,6.2)(-3.8,3)
\psline[linewidth=.5pt](0,10)(-.5,9.8)(-1,10.2)(-1.5,9.9)(-2,10.3)(-2.5,10)(-3,9.9)(-3.5,9.8)(-4,10.1)(-4.5,9.7)(-5,10)(-5.1,9.6)(-5.2,9.5)(-5.3,9.1)(-5.4,8.5)(-5.6,8.5)(-5.7,8.3)(-6,8)(-5.5,8.1)(-5,7.8)(-4.5,7.7)(-4,8.1)(-3.5,7.8)(-3,8)(-2.5,7.8)(-2,7.9)(-1.5,8)(-1,8)(-0.9,8.3)(-.8,8.5)(-.7,8.5)(-.6,8.8)(-.5,9.2)(-.4,9.4)(-.3,9.4)(-.2,9.5)(-.1,9.7)(0,10)
\pscurve[linewidth=1.7pt]{<-}(-0.960769515, 2.165063510)
                         (-1.025861389, 2.279038980)
                         (-1.100686534, 2.370243048)
                         (-1.184497322, 2.437764430)
                         (-1.276456342, 2.480928476)
                         (-1.375644772, 2.499303904)
                         (-1.481071553, 2.492707115)
                         (-1.591683296, 2.461204019)
                         (-1.706374804, 2.405109387)
                         (-1.824000117, 2.324983697)
                         (-1.943383964, 2.221627538)
\rput(-1,2.7){$\Omega$}
\psline{->}(-5.5,2.5)(-7,2.5)
\psline{->}(-5.5,2.5)(-5.5,4.5)
\psline{->}(-5.5,2.5)(-5,3.5)
\psline[linestyle=dotted](-13,0)(3,0)
\psline[linestyle=dotted](0,-3)(0,3)
\psline[linestyle=dotted](-1,-3)(1,3)
\rput(-7.2,2.5){$y$}
\rput(-5.7,4.6){$z$}
\rput(-5,3.7){$x$}
\rput(-7.4,7){\vbox{\hsize.5cm{\noindent{tool holder}}}}
\psline{->}(1,4)(-1,4)
\rput(0,4.3){feed}
\rput(-2.2,3.5){tool}
\rput(-7,-.5){workpiece}
}
\vskip2cm
\caption{Turning model}\label{Figure1}
\end{figure}
 
Assuming that the spindle velocity $\Omega$ (expressed in [rad/s]) is time-varying, the time delay $\tau$ between the present and the previous cut is determined by the equation,
\begin{align}\label{milling-eqn-3} 
R \int_{t-\tau(t)}^{t} \Omega (s) \diff s = 2 R \pi +x(t) -x(t- \tau (t)),
\end{align}
where $R$ is the radius of the workpiece. Then the time delay $\tau$ is implicitly determined by the oscillation in the $x$ direction. Equations~(\ref{milling-eqn-1-1}--\ref{milling-eqn-3}) form a model of turning processes with state-dependent delay.

We rewrite the governing equation (\ref{milling-eqn-3})  for the delay as
\begin{align}
\int_{t-\tau(t)}^t\frac{R\Omega-\dot{x}(s)}{2R\pi}\mathrm{d}s=1.\label{milling-eqn-5}
\end{align}
Let $\dot{x}(t)=u(t),\,\dot{y}(t)=v(t)$.  System (\ref{milling-eqn-1-1}--\ref{milling-eqn-3}) can be re-written as
\begin{align}\label{SDDEs-system}\left\{\begin{aligned}
 \frac{\mathrm{d}}{\mathrm{d}t}\begin{bmatrix}x\\ y\\  u \\ v\end{bmatrix}&=
\begin{bmatrix}
u\\
v\\
-\frac{c_x}{m}u-\frac{k_x}{m}x+\frac{K_x\omega}{m}(\nu\tau+y(t)-y(t-\tau(t)))^q\\
-\frac{c_y}{m}v-\frac{k_y}{m}y-\frac{K_y\omega}{m}(\nu\tau+y(t)-y(t-\tau(t)))^q
\end{bmatrix},\\
1&=\int_{t-\tau(t)}^t \frac{R\Omega-u(s)}{2\pi R}\mathrm{d}s, 
\end{aligned}\right.
\end{align} 
which is a system of differential  equations with threshold type state-dependent delay \cite{HU-JDE-1}. 
We are concerned about stabilization of turning processes with a delayed spindle speed control.  We assume that
\begin{align}\label{delayed-spindle}
\Omega(t) = \frac{c}{R} x(t- \tau (t)),                                                                                          \end{align} where $c \in \mathbb{R}$ is a parameter.
Now we show that system~(\ref{SDDEs-system}) coupled with the spindle speed control (\ref{delayed-spindle}) can be transformed into a system of integral-differential equations.
Assuming that $R\Omega >u(t)$ for all  $t\in\mathbb{R}$, we consider the following process of change of variables for system (\ref{SDDEs-system}):
\begin{align}
\eta & =\int_{0}^t\frac{R\Omega-u(s)}{2\pi R}\mathrm{d}s\notag \\ 
 & =\int_{0}^t\frac{cx(s-\tau(s))-u(s)}{2\pi R}\mathrm{d}s,\label{c-v-1}\\
 r(\eta)& =x(t),\label{c-v-2}\\
 \rho(\eta)&= y(t),\label{c-v-3}\\
 j(\eta) &=u(t),\label{c-v-4}\\
 l(\eta) &= v(t),\label{c-v-5}\\
 k(\eta) &= \tau(t).\label{c-v-6}
\end{align} Then by (\ref{milling-eqn-5}) and (\ref{c-v-1}) we have
\[
\eta-1= \int_{0}^{t-\tau(t)}\frac{R\Omega-u(s)}{2\pi R}\mathrm{d}s,                                                                                                              \] and
\begin{align*} 
 r(\eta-1)& =x(t-\tau(t)),\\
 \rho(\eta-1)&= y(t-\tau(t)).
\end{align*}
The second equation of (\ref{SDDEs-system})  for $\tau$ can be rewritten as
\begin{align*}
 \tau(t)& =t-(t-\tau(t))\\
        & = \int_{\eta-1}^\eta \frac{\mathrm{d}t}{\mathrm{d}\bar{\eta}}d\bar{\eta}\\
        & = \int_{\eta-1}^\eta \frac{2\pi R}{c\, r(\bar{\eta}-1)-j(\bar{\eta})}\mathrm{d}\bar{\eta}\\
        & =  \int_{-1}^0 \frac{2\pi R}{c\, r_\eta(s-1)-j_{\eta}(s)}\mathrm{d}s.
\end{align*}
It follows that
\begin{align}\label{eqn-z-eta}
 k(\eta) =  \int_{-1}^0 \frac{2\pi R}{c\, r_\eta(s-1)-j_{\eta}(s)}\mathrm{d}s.
\end{align} Furthermore, taking derivative with respect to $t$ on both sides of (\ref{c-v-2}) we have
\begin{align*}
 \frac{\mathrm{d}r}{\mathrm{d}\eta} \frac{\mathrm{d}\eta}{\mathrm{d}t}=\dot{x}(t)
\end{align*}which leads to
\begin{align*}
 \frac{\mathrm{d}r}{\mathrm{d}\eta} =\dot{x}(t)\frac{\mathrm{d}t}{\mathrm{d}\eta}=j(\eta)\frac{2\pi R}{c\,r(\eta-1)-j(\eta)}.
\end{align*}
Similarly we have
\begin{align*}
 \frac{\mathrm{d}\rho}{\mathrm{d}\eta} =\dot{y}(t)\frac{\mathrm{d}t}{\mathrm{d}\eta}=l(\eta)\frac{2\pi R}{c\,r(\eta-1)-j(\eta)}.
\end{align*}Therefore, system (\ref{SDDEs-system}) can be rewritten as
\begin{align}\label{SDDEs-system-eta}\left\{\begin{aligned}
 \frac{\mathrm{d}}{\mathrm{d}\eta}\begin{bmatrix}r\\ \rho\\  j \\ l\end{bmatrix}&=
\begin{bmatrix}
j\\
l\\
-\frac{c_x}{m}j-\frac{k_x}{m}r+\frac{K_x\omega}{m}(\nu\,k+\rho-\rho(\eta-1))^q\\
-\frac{c_y}{m}l-\frac{k_y}{m}\rho-\frac{K_y\omega}{m}(\nu\,k+\rho-\rho(\eta-1))^q
\end{bmatrix}\frac{2\pi R}{c\,r(\eta-1)-j(\eta)},\\
 k(\eta) & =  \int_{-1}^0 \frac{2\pi R}{c\,r_\eta(s-1)-j_{\eta}(s)}\mathrm{d}s, 
\end{aligned}\right.
\end{align}
which is an integral-differential  equation with both discrete and distributed delays. 

\subsection{Characteristic equation of  system~(\ref{SDDEs-system-eta})}
Let $(r^*, \rho ^*, k^*, j^*, l^*)$ be a stationary state of  system~(\ref{SDDEs-system-eta}). The the right hand sides of the differential equations lead to $j^*=l^*=0$ and
\begin{align*}
 r^*=\frac{K_xw(\nu k^*)^q}{k_x},\,\,\rho^*=-\frac{K_yw}{k_y}(\nu k^*)^q,\,\,k^*=\frac{2\pi R}{cr^*}=\left(\frac{2\pi R k_x}{cK_xw\nu^q}\right)^\frac{1}{q+1}.
\end{align*}
The unique stationary point of  system~(\ref{SDDEs-system-eta}) is
\begin{align}
(r^*, \rho ^*, k^*, j^*, l^*)=\left(\frac{K_xw(\nu k^*)^q}{k_x}, -\frac{K_{y} \omega (\nu k^*)^q}{k_{y}}, k^*, 0 , 0 \right),
\end{align} where $k^*=\left(\frac{2 \pi R }{c}\cdot\frac{k_{x}}{K_{x}\omega \nu^q}\right)^{\frac{1}{q+1}}$. For system~(\ref{SDDEs-system-eta}) setting $\textbf{x} = (x_1,\,x_2,\,x_3,\,x_4) = (r, \rho, j, l) - (r^*, \rho ^*, j^*, l^*)$,  we obtain the linearization of  system~(\ref{SDDEs-system-eta})  near the stationary state:
\begin{align}\label{linearization}
\frac{\diff \textbf{x}}{\diff \eta} = k^* (M \textbf{x} + N \textbf{x}(\eta -1))- \int_{-1}^{0} \left( \frac{{k^*}^3 \nu}{2 \pi R} P \textbf{x}_\eta (s) - \frac{{c k^*}^3 \nu}{2 \pi R} Q \textbf{x}_\eta (s-1) \right) \diff s,
\end{align}
where $M, N, P$ and $Q$ are $4 \times 4$ matrices given by
\begin{align*}
M=
\begin{bmatrix}
0 &0 &1 &0 \\ 
0 &0 &0 &1 \\ 
-\frac{k_{x}}{m} &\frac{q K_{x} \omega (\nu k^*)^{q-1}}{m} &-\frac{c_{x}}{m} &0 \\ 
0 &-\frac{k_{y}}{m}-\frac{q K_{y} \omega (\nu k^*)^{q-1}}{m} &0 &-\frac{c_{y}}{m}\\
\end{bmatrix},\quad N=
\begin{bmatrix}
0 &0 &0 &0 \\ 
0 &0 &0 &0 \\ 
0 &-\frac{q K_{x} \omega (\nu k^*)^{q-1}}{m} &0 &0\\
0 &\frac{q K_{y} \omega (\nu k^*)^{q-1}}{m} &0 &0\\
\end{bmatrix},
\end{align*}
and
\begin{equation*}
P=
\begin{bmatrix}
0 &0 &0 &0 \\ 
0 &0 &0 &0 \\
0 &0 &-\frac{q K_{x} \omega (\nu k^*)^{q-1}}{m} &0\\
0 &0 &\frac{q K_{y} \omega (\nu k^*)^{q-1}}{m} &0
\end{bmatrix},\quad
Q=
\begin{bmatrix}
0 &0 &0 &0 \\ 
0 &0 &0 &0 \\
-\frac{q K_{x} \omega (\nu k^*)^{q-1}}{m} &0 &0 &0\\
\frac{q K_{y} \omega (\nu k^*)^{q-1}}{m} &0 &0 &0 
\end{bmatrix}.
\end{equation*}
To have the stationary state of system~(\ref{SDDEs-system-eta}) the same as the corresponding system with constant spindle velocity $\Omega_0$, we require $k^*$ equal to the stationary state $\tau^*$ of $\tau$. Then by (\ref{milling-eqn-3}) and (\ref{c-v-6}) we need  the spindle control parameter $c$ to satisfy 
\begin{align}\label{virtual-spindle-speed}\frac{2 \pi}{\Omega_{0}}=\left(\frac{2 \pi R k_{x}}{c K_{x} \omega \nu^q}\right)^{\frac{1}{q+1}},
\end{align}
which leads to 
\begin{align}\label{c-value-1}
c=\frac{R\Omega_{0}^{q+1} k_{x}}{(2 \pi \nu)^q K_{x} \omega}.  \end{align}
 In the following we treat $\Omega_0>0$ satisfying  (\ref{virtual-spindle-speed}) as a virtual constant spindle speed.

We denote by $k_{r}$ the cutting force ratio $\frac{K_{y}}{K_{x}}$, by $K_{1}$ the dimensionless depth of cut $\frac{q K_{y} \omega (2 \pi R)^{q-1}}{k_{x}}$, and by $p$ the dimensionless feed per revolution $\frac{\nu}{R \Omega_{0}}$, then we have:
\begin{align}\label{parameter-simplification}
\left\{
\begin{aligned}
\frac{q K_{x} \omega (\nu k^*)^{q-1}}{m} & = \frac{k_{x}}{m} \frac{K_{1}}{k_{r}} p^{q-1},\\
 \frac{q K_{y} \omega (\nu k^*)^{q-1}}{m} & = \frac{k_{x}}{m} K_{1} p^{q-1},\\
 \frac{{k^*}^3 \nu}{2 \pi R} & = p {k^*}^2,\\
\frac{{c k^*}^3 \nu}{2 \pi R} & = \frac{k^* q k_{r}}{K_{1}} p^{1-q}.
\end{aligned}\right.
\end{align}
Thus system (\ref{linearization}) can be rewritten explicitly as follows:
\begin{align}\label{4-eqns}
\left\{
\begin{aligned}
 \dot{x}_1 = & \, k^* x_3,\\
  \dot{x}_2 = &\,  k^* x_4,\\
 \dot{x}_3 = &\,   k^* \left(-\frac{k_x}{m}x_1 -\frac{c_x}{m}x_3+ \frac{k_{x}}{m} \frac{K_{1}p^{q-1}}{k_{r}}  (x_2- x_2(\eta-1))\right)\\
& + \int_{-1}^{0} \frac{k_{x}}{m} \frac{K_{1}p^{q-1}}{k_{r}}  \left( p {k^*}^2 x_{3\eta}(s)- \frac{qk^*  k_{r}p^{1-q}}{K_{1}} x_{1\eta} (s-1) \right) \mathrm{d}s,\\
 \dot{x}_4  = &\, k^* \left(-\frac{k_y}{m}x_2 - \frac{c_y}{m}x_4 -\frac{k_{x}K_{1} p^{q-1}}{m}  (x_2- x_2(\eta-1))\right)\\
 & -\int_{-1}^{0} \frac{k_{x}K_{1}p^{q-1}}{m}  \left( p {k^*}^2 x_{3\eta}(s)- \frac{qk^* k_{r}p^{1-q}}{K_{1}} x_{1\eta} (s-1) \right) \mathrm{d}s.
\end{aligned}\right.
\end{align}
For the convenience of computing the characteristic equation, we transform system (\ref{4-eqns}) into a set of second order scalar equations of $(x_{1}, x_{2})$. Notice that $\int_{-1}^{0} x_{3 \eta}(s) \d s=\int_{-1}^{0} \frac{\dot{x}_{1 \eta}(s)}{k^*} \d s= \frac{1}{k^*} (x_1(\eta)- x_1(\eta-1))$. We differentiate the first two equations of (\ref{4-eqns}) and substitute $\dot{x}_3, \ \dot{x}_4$  to obtain:
\begin{align}\label{second-order-eqns}
\left\{
\begin{aligned}
&\ddot{x}_{1} + \frac{k^* c_{x}}{m} \dot{x}_{1} +\frac{{k^*}^2 k_{x}}{m} x_{1} -\frac{k_{x} K_{1} {k^*}^2p^q}{m k_{r}}  (x_{1}(\eta) - x_{1}(\eta -1)) +  \frac{k_{x}{k^*}^2 q}{m} \int_{-1}^{0} x_{1\eta}(s-1) \mathrm{d}s \\
 &\quad = \frac{k_{x} K_{1} {k^*}^2p^{q-1}}{m k_{r}}  (x_{2}(\eta) - x_{2}(\eta -1)), \\
&\ddot{x}_{2} + \frac{k^* c_{y}}{m} \dot{x}_{2} + \frac{{k^*}^2 k_{y}}{m} x_{2} + \frac{k_{x} K_{1} {k^*}^2p^{q-1}}{m}  (x_{2}(\eta) - x_{2}(\eta -1)) \\
& \quad =-\frac{k_{x} K_{1} {k^*}^2p^q}{m}  (x_{1}(\eta) - x_{1}(\eta -1))+ \frac{k_{x}k_r{k^*}^2 q}{m}  \int_{-1}^{0} x_{1\eta}(s-1) \mathrm{d}s.
\end{aligned}\right.
\end{align}
Assume that the tool is symmetric, namely $c_x= c_y, \ k_x= k_y$. Inspecting the equations in system (\ref{second-order-eqns}) we know that  system (\ref{second-order-eqns}) has no nonconstant solution of the form $(x_1, x_2) =(c_1, c_2) e^{\lambda \eta}$, $(c_1, \ c_2) \in\mathbb{R}^2$ with a zero component. Otherwise, the second order scalar equation $\ddot{\bf y}(\eta) + \frac{k^* c_{x}}{m} \dot{\bf y}(\eta) +\frac{{k^*}^2 k_{x}}{m} {\bf y}(\eta)=0$ has nonconstant 1-periodic solution, which is impossible because $\frac{{k^*}^2 k_{x}}{m}>0$.

By setting $(x_1, x_2) =(c_1, c_2) e^{\lambda \eta}$ with $c_1\neq 0$ and $c_2 \neq 0$ and multiplying the corresponding left and right hand sides  of the two equations at (\ref{second-order-eqns}), we obtain the characteristic equation of the linearized system
\begin{align}\label{characteristic-eqn}
\mathcal{P}(\lambda) \left( \mathcal{P}(\lambda) + \left( \frac{k_x K_1 {k^*}^2 p^{q-1}} {m}  \left( 1-\frac{p}{k_r} \right) + \frac{k_x {k^*}^2q}{m} \frac{1}{\lambda e^\lambda} \right) \left( 1-e^{-\lambda} \right) \right) = 0,
\end{align}
where $P(\lambda)=\lambda^2 +\frac{k^* c_x}{m} \lambda +\frac{{k^*}^2 k_x}{m}$.  Denote by 
\begin{align}\label{xi-delta-def}
\xi=\frac{c_xk^*}{m},\,\delta=\frac{k_x{k^*}^2}{m},\,h_1=K_1 p^{q-1}  \left( 1-\frac{p}{k_r} \right).
\end{align} 
Then the characteristic equation (\ref{characteristic-eqn}) can be rewritten as:
\begin{align*}
\mathcal{P}(\lambda) \left( \mathcal{P}(\lambda) + \delta \left(h_1 +  \frac{q}{\lambda e^\lambda} \right) \left( 1-e^{-\lambda} \right) \right) = 0,
\end{align*} where $P(\lambda)=\lambda^2+\xi\lambda+\delta$.
We have
\begin{align}\label{charac-eqn-1}
\mathcal{P}(\lambda) \left( \mathcal{P}(\lambda) + \delta\left(h_1 +  \frac{q}{\lambda e^\lambda} \right) \left( 1-e^{-\lambda} \right) \right) = 0,
\end{align} 

Notice that  the roots of the quadratic polynomial $\mathcal{P}(\lambda)$ always have negative real parts since $\xi=\frac{k^* c_x}{m}>0$ and $\delta=\frac{{k^*}^2 k_x}{m}>0$. In the following we consider the roots of 
\begin{align}\label{charac-1}
 \mathcal{P}(\lambda) + \delta \left(h_1 +  \frac{q}{\lambda e^\lambda} \right) \left( 1-e^{-\lambda} \right) =0.
\end{align}
We first notice that $\lambda=0$ is not a removable singularity of $\mathcal{P}(\lambda) + \delta\left(h_1+\frac{q}{\lambda e^\lambda`}\right) (1- e^{-\lambda})$ such that the limit there is zero since we have
\[
 \frac{\d}{\d\lambda}\lambda \left(\mathcal{P}(\lambda) + \delta\left(h_1+\frac{q}{\lambda e^\lambda}\right) (1- e^{-\lambda})\right)\,\vline_{\lambda=0}=\delta(q+1)\neq 0,
\]if $\delta\neq 0$.
Therefore, if $\lambda=i\beta$, $\beta\in\mathbb{R}$ is an eigenvalue then $\beta\neq 0$ and we have,
\begin{align}\label{real-imaginary-02}
\left\{
\begin{aligned}
 -\beta^2+\delta+\delta h_1(1-\cos\beta)-\frac{\delta q}{\beta}\sin\beta(1-2\cos\beta)  & =  0,\\
 \xi\beta+\delta h_1\sin\beta-\frac{\delta q}{\beta}(1-\cos\beta)(1+2\cos\beta)  &  =  0.
 \end{aligned}
 \right.
\end{align}

We have
\begin{lemma}\label{lemma-2-3}Suppose that $\xi$, $q$,  $h_1$ and $\delta$ are positive with $\xi\neq 2q$. If $\lambda=i\beta$, $\beta\in\mathbb{R}$  is a zero of $\mathcal{P}(\lambda) + \delta\left(h_1+\frac{q}{\lambda e^\lambda}\right) (1- e^{-\lambda})$, then $\beta\neq 0$ and the following statements are true:
\begin{enumerate}\item[i$\,)$] If $h_1=0$, then $0<\beta^2<\frac{9\delta q}{8\xi}$; If $h_1\neq 0$, then
\[
0<\beta^2\leq \frac{\delta |h_1|+2\delta h_1^2+2q\xi+\sqrt{(\delta |h_1|+2\delta h_1^2+2q\xi)^2+8 q^2\delta |h_1|}}{2 |h_1|},
\] If, in addition, $(\delta |h_1|-2\delta h_1^2-{2q\xi})^2-4q^2\delta |h_1|\geq 0$, then
\[
 \beta^2\geq\max\left\{0,\, \frac{\delta |h_1|-2\delta h_1^2-{2q\xi}+\sqrt{(\delta |h_1|-2\delta h_1^2-{2q\xi})^2-4q^2\delta |h_1|}}{2 |h_1|}\right\}.
\]
\item[ ii$\,)$]We have
\begin{align}\label{lemma-2-3-1}\left\{
\begin{aligned}
h_1 = &\frac{q\beta(1-\cos\beta)(1+2\cos\beta)-\xi\beta+q\xi\sin\beta(1-2\cos\beta)}{\xi\beta(1-\cos\beta)+\beta^2\sin\beta},\\[0.25em]
\delta = & \frac{\xi\beta^2(1-\cos\beta)+\beta^3\sin\beta}{2q\cos\beta(1-\cos\beta)+\beta\sin\beta}.
\end{aligned}\right.
 \end{align}  
 
 \item[iii$\,)$] $\beta\neq 2n\pi$ for every $n\in\mathbb{Z}$. 
\end{enumerate}
\end{lemma}
\begin{proof} i) If $h_1=0$, by  the second equation of (\ref{real-imaginary-02}) we have
\[
 \beta^2=\frac{\delta q}{\xi}(-2\cos^2 \beta+\cos\beta+1).
\]Evaluating the  quadratic polynomial of $\cos\beta$ at the right hand side in $[-1,\,1]$ we obtain that
\[
 0<\beta^2<\frac{9\delta q}{8\xi}.
\]
If $h_1\neq 0$ we multiply  both sides of the second equation of (\ref{real-imaginary-02}) by $\frac{q(1-\cos\beta)}{\beta h_1}$ and add the resulted sides to the first equation, and obtain,
 \begin{align*}
  -\beta^2+\delta+\delta h_1(1-\cos\beta)+\frac{q\xi}{h_1}(1-\cos\beta)-\frac{q^2\delta}{\beta^2h_1}(1-\cos\beta)^2(3-2(1-\cos\beta))=0.
 \end{align*} Notice that the term $(1-\cos\beta)^2(3-2(1-\cos\beta))$ assumes maximum $1$ when $1-\cos\beta=1$ and minimum $-2$ at $1-\cos\beta=2$.
We have the inequalities
\[
 \delta-2\delta |h_1|-\frac{2q\xi}{|h_1|}-\frac{q^2\delta}{\beta^2|h_1|}\leq \beta^2\leq \delta+2\delta |h_1|+\frac{2q\xi}{|h_1|}+\frac{2q^2\delta}{\beta^2 |h_1|},
\]
the latter of which lead to
\begin{align*}
0<\beta^2\leq \frac{\delta |h_1|+2\delta h_1^2+{2q\xi}+\sqrt{(\delta |h_1|+2\delta h_1^2+{2q\xi})^2+8 q^2\delta |h_1|}}{2 |h_1|}.
\end{align*}Moreover, if $(\delta |h_1|-2\delta h_1^2-{2q\xi})^2-4q^2\delta |h_1|\geq 0$, we have
\[
 \beta^2\geq\max\left\{0,\, \frac{\delta |h_1|-2\delta h_1^2-{2q\xi}+\sqrt{(\delta |h_1|-2\delta h_1^2-{2q\xi})^2-4q^2\delta |h_1|}}{2 |h_1|}\right\}.
\]

ii) By the first equation of (\ref{real-imaginary-02}), we have
\begin{align}\label{lemma2-3-1-h2}
 \delta h_1=\frac{\beta^2-\delta+\frac{\delta q}{\beta}\sin\beta(1-2\cos\beta)}{1-\cos\beta},
\end{align} with which the second  equation of (\ref{real-imaginary-02}) becomes
\begin{align}\label{real-imaginary-03}
 \xi\beta+\sin\beta\left(\frac{\beta^2-\delta+\frac{\delta q}{\beta}\sin\beta(1-2\cos\beta)}{1-\cos\beta}\right)-\frac{\delta q}{\beta}(1-\cos\beta)(1+2\cos\beta)   =  0.
\end{align}
From (\ref{real-imaginary-03}) we obtain
\begin{align}\label{real-imaginary-04}
 \delta q=\frac{1}{2}\left(\xi\beta^2+\frac{\beta(\beta^2-\delta)\sin\beta}{1-\cos\beta}\right).
\end{align}Solving for $\delta$ from (\ref{real-imaginary-04}) we have
\begin{align}\label{real-imaginary-05}
 \left({2q\cos\beta(1-\cos\beta)+\beta\sin\beta}\right)\delta=\xi\beta^2(1-\cos\beta)+\beta^3\sin\beta.
\end{align}
Notice that if ${2q\cos\beta(1-\cos\beta)+\beta\sin\beta}=0$, (\ref{real-imaginary-05}) becomes $0=\frac{1}{2}\xi\beta^2-q\beta^2\cos\beta$ and hence $\xi=2q\cos\beta$. (\ref{real-imaginary-04}) is then reduced to
\[
 \delta q=\frac{1}{2}\left(\xi\beta^2-(\beta^2-\delta)(-2q\cos\beta)\right)=\frac{1}{2}\left(\xi\beta^2-(\beta^2-\delta)\xi\right)=\frac{1}{2}\xi\delta.
\]
That is, $\xi=2q$, which is a contradiction to the assumption. Therefore, we obtain ${2q\cos\beta(1-\cos\beta)+\beta\sin\beta}\neq 0$ and (\ref{real-imaginary-05}) implies that
\begin{align}\label{real-imaginary-06}
 \delta= \frac{\xi\beta^2(1-\cos\beta)+\beta^3\sin\beta}{2q\cos\beta(1-\cos\beta)+\beta\sin\beta}.
\end{align}
 Next we solve for $h_1$. Note that $\delta>0$ and hence $\xi(1-\cos\beta)+\beta\sin\beta\neq 0$. We bring $\delta$ at (\ref{real-imaginary-06}) into (\ref{lemma2-3-1-h2}) to obtain
\begin{align}\label{real-imaginary-07}
\begin{aligned}
 h_1=&\frac{q\beta(1-\cos\beta)(1+2\cos\beta)-\xi\beta+q\xi\sin\beta(1-2\cos\beta)}{\xi\beta(1-\cos\beta)+\beta^2\sin\beta}.
  \end{aligned}
\end{align}

iii)  Notice that by the second equation of (\ref{real-imaginary-02}), $1-\cos\beta\neq 0$, otherwise, $\sin\beta=0$ and hence $\xi\beta=0$ which is impossible. Therefore, we have $\beta\neq 2n\pi$ for every $n\in\mathbb{Z}$.
\qed
\end{proof}
\begin{remark}
With the estimates of $\beta$ at i) of Lemma~\ref{lemma-2-3} we can provide an approximation  of the frequency of vibrations by $\frac{\beta}{2\pi}$. 
\end{remark}
\begin{lemma}\label{lemma-2-4}Assume that $\xi,\,h_1$ and $\delta$ are positive constants. If $\beta\in\mathbb{R}$ is such that ii) and iii) of Lemma~\ref{lemma-2-3} are satisfied, then $i\beta$ is a purely imaginary zero of
of $\mathcal{P}(\lambda) + \delta\left(h_1+\frac{q}{\lambda e^\lambda}\right) (1- e^{-\lambda})$.
 \end{lemma}
\begin{proof}
 We show that (\ref{real-imaginary-02}) is true. Working backward from $h_1$ at (\ref{real-imaginary-07}), we have (\ref{lemma2-3-1-h2}) which leads to the first equation of (\ref{real-imaginary-02}). Similarly we working backward from $\delta$ at (\ref{real-imaginary-06}) we obtain (\ref{real-imaginary-04}) and (\ref{real-imaginary-03}) which combined with 
(\ref{lemma2-3-1-h2}) give the second equation of (\ref{real-imaginary-02}). \qed \end{proof}

In the following, we are interested  to obtain a parameterization of all possible  values of $(h_1,\,\delta)$ for which  $\mathcal{P}(\lambda) + \delta\left(h_1+\frac{q}{\lambda e^\lambda}\right) (1- e^{-\lambda})$ has purely imaginary zeros.
By the definitions of $h_1$ and $\delta$ and by Lemma~\ref{lemma-2-4},  we need to parameterize positive $\delta$ and $h_1\in\mathbb{R}$ satisfying
ii) and iii) of Lemma~\ref{lemma-2-3}. Also, we note that from ii) of Lemma~\ref{lemma-2-3} that $(h_1,\,\delta)$ is an even function of $\beta$ on its domain. Hence we assume $\beta>0$ in the following. For simplicity, we denote $I_n=(2(n-1),\,2n\pi),\,n\in\mathbb{N}.$
\begin{lemma}\label{lemma-2-5}
Let $\xi$ and $q$ be positive with $\xi\neq 2q$.  For every $n\in\mathbb{N}$, the equation $\xi(1-\cos\beta)+\beta\sin\beta=0$ has a unique zero $\beta_n^*\in ((2n-1)\pi,\,2n\pi)\subset I_n$. 
\end{lemma}
\begin{proof}
 Note that
  \begin{align}\label{zero-01}
   \frac{\d}{\d \beta}\left(\frac{\xi}{\beta}+\frac{\sin\beta}{1-\cos\beta}\right)= -\frac{\xi}{\beta^2}-\frac{1}{(1-\cos\beta)}<0,
  \end{align}
and that $\lim_{\beta\rightarrow (2n\pi)^-} \frac{\xi}{\beta}+\frac{\sin\beta}{1-\cos\beta}=+\infty$, $\lim_{\beta\rightarrow (2(n-1)\pi)^+} \frac{\xi}{\beta}+\frac{\sin\beta}{1-\cos\beta}=-\infty$.
It follows that the equation  $\xi(1-\cos\beta)+\beta\sin\beta=0$ has a unique zero $\beta_n^*$  in $(2(n-1)\pi,\,2n\pi)$. Notice that by the equation we have $\frac{\sin\beta_n^*}{1-\cos\beta_n^*}=-\frac{\xi}{\beta_n^*}<0$. Therefore,  we have
that $\beta_n^*\in((2n-1)\pi,\,(2n-1)\pi)$. \qed
\end{proof}
\begin{lemma}\label{lemma-2-6}
Let $\xi$ and $q$ be positive with $\xi\neq 2q$, $\beta_1=2n\pi-\arccos\frac{-1+\sqrt{5}}{2}$. For every $n\in\mathbb{N}$, consider the zeros of the equation $2q\cos\beta(1-\cos\beta)+\beta\sin\beta=0$ in the interval $(2(n-1),\,2n\pi)$. The following are true:
\begin{enumerate}
\item[$i)$]If $q\geq\beta_1\frac{\sqrt{2\sqrt{5}-2}}{4\sqrt{5}-8}$, there are exactly three zeros $\tilde{\beta}_{n,\,1}<\tilde{\beta}_{n,\,2}\leq \tilde{\beta}_{n,\,3}$ and $\tilde{\beta}_{n,\,1}\in \left((2n-\frac{3}{2})\pi,\,(2n-1)\pi\right)$,\,$\tilde{\beta}_{n,\,2}\in \left((2n-\frac{1}{2})\pi,\,2n\pi\right)$ and $\tilde{\beta}_{n,\,3}\in \left((2n-\frac{1}{2})\pi,\,2n\pi\right)$.
\item[$ii)$] If $0<q<\beta_1\frac{\sqrt{2\sqrt{5}-2}}{4\sqrt{5}-8}$, there is a unique zero $\tilde{\beta}_{n,\,1}\in \left((2n-\frac{3}{2})\pi,\,(2n-1)\pi\right)$.
\end{enumerate}
\end{lemma}
\begin{proof}
 i) Note that $\beta\neq 0$ and $\cos\beta\neq 0$ and $\cos\beta\neq 1$ for every $\beta\in I_n=(2(n-1)\pi,\,2n\pi)$. Consider functions $f: I_n\setminus\{(2n-\frac{3}{2})\pi,\,(2n-\frac{1}{2})\pi\} \rightarrow\mathbb{R}$ and $g: (0,\,+\infty)\rightarrow\mathbb{R}$ defined by
\begin{align*}
 y=f(\beta)=\frac{\sin\beta}{\cos\beta(1-\cos\beta)},\,
 y=g(\beta)=-\frac{2q}{\beta}.
\end{align*}Note that $f$ and $g$ have different signs in $(2(n-1)\pi,\,(2n-\frac{3}{2})\pi)\cup((2n-1)\pi,\,2n-\frac{1}{2})\pi)$. Hence the zeros of $f-g$, if any, must be in the intervals $(2n-\frac{3}{2})\pi,\,(2n-1)\pi)$ and $((2n-\frac{1}{2})\pi,\,2n\pi)$. See Figure~\ref{lemma25-fig}.
\begin{figure}
\begin{center}
\psscalebox{1.0 1.0} 
{
\begin{pspicture}(0,-4.753479)(11.07,4.753479)
\psline[linecolor=black, linewidth=0.04, arrowsize=0.05291667cm 2.0,arrowlength=1.4,arrowinset=0.0]{->}(0.8,-4.673479)(0.8,4.846521)
\psline[linecolor=black, linewidth=0.04, arrowsize=0.05291667cm 2.0,arrowlength=1.4,arrowinset=0.0]{->}(0.0,-0.313479)(10.4,-0.313479)
\psline[linecolor=black, linewidth=0.04, linestyle=dashed, dash=0.17638889cm 0.10583334cm](9.6,-4.233479)(9.6,4.446521)
\psline[linecolor=black, linewidth=0.04, linestyle=dashed, dash=0.17638889cm 0.10583334cm](7.2788224,-4.273438)(7.2411776,4.32648)
\psline[linecolor=black, linewidth=0.04, linestyle=dashed, dash=0.17638889cm 0.10583334cm](2.96,-4.373479)(2.96,4.306521)
\psbezier[linecolor=black, linewidth=0.04](0.9,3.966521)(0.95727026,3.0646448)(1.3226267,0.126671)(1.8484331,0.10652099609375)(2.3742397,0.08637099)(2.7445014,3.0625038)(2.8,4.0062337)
\psbezier[linecolor=black, linewidth=0.04](9.53326,-3.935074)(9.461714,-3.208931)(8.877472,-0.6970694)(8.441683,-0.7012324050321468)(8.005894,-0.7053954)(7.4932985,-3.0721219)(7.4130306,-3.9692216)
\psbezier[linecolor=black, linewidth=0.04](7.08,3.986521)(7.111503,0.72830147)(6.166998,-0.32068786)(5.1177144,-0.3227607718620493)(4.0684304,-0.3248337)(3.1592846,-2.3616765)(3.1,-3.953479)
\psbezier[linecolor=black, linewidth=0.07](1.02,-4.573479)(0.77768487,-0.680843)(9.38,-1.1795042)(10.3,-1.07347900390625)
\rput[bl](2.28,-4.753479){$(2n-\frac{3}{2})\pi$}
\rput[bl](6.56,-4.733479){$(2n-\frac{1}{2})\pi$}
\rput[bl](8.34,-0.173479){$\beta_1$}
\pscircle[linecolor=black, linewidth=0.04, fillstyle=solid,fillcolor=black, dimen=outer](8.783044,-1.1004355){0.08304354}
\rput{-1.3019527}(0.02784795,0.18397798){\pscircle[linecolor=black, linewidth=0.04, fillstyle=solid,fillcolor=black, dimen=outer](8.11,-1.133479){0.09}}
\pscircle[linecolor=black, linewidth=0.04, fillstyle=solid,fillcolor=black, dimen=outer](8.46,-0.693479){0.08}
\rput[bl](7.78,0.506521){$\tilde{\beta}_{n,\,2}$}
\rput[bl](8.48,-2.753479){$\tilde{\beta}_{n,3}$}
\psline[linecolor=black, linewidth=0.04, linestyle=dashed, dash=0.17638889cm 0.10583334cm](8.08,-1.073479)(8.08,0.38652098)
\psline[linecolor=black, linewidth=0.04, linestyle=dashed, dash=0.17638889cm 0.10583334cm](8.78,-2.013479)(8.78,-0.313479)
\psline[linecolor=black, linewidth=0.04, linestyle=dashed, dash=0.17638889cm 0.10583334cm](8.4774275,-0.29348108)(8.491268,-0.84217215)
\psline[linecolor=black, linewidth=0.04, linestyle=dashed, dash=0.17638889cm 0.10583334cm](3.6020103,-1.7735183)(3.6379898,-0.31343964)
\pscircle[linecolor=black, linewidth=0.04, fillstyle=solid,fillcolor=black, dimen=outer](3.62,-1.773479){0.08}
\rput[bl](3.46,-0.23347901){$\tilde{\beta}_{n,1}$}
\rput[bl](4.28,2.366521){$y=\frac{\sin \beta}{\cos \beta(1-\cos \beta)}$}
\rput[bl](9.84,-1.653479){$y=-\frac{2q}{\beta}$}
\rput[bl](10.08,-0.873479){$\beta$}
\rput[bl](0.4,4.466521){$y$}
\rput[bl](9.36,-4.593479){$2n\pi$}
\end{pspicture}
}
\end{center}\caption{The graphs of $y=f(\beta)=\frac{\sin\beta}{\cos\beta(1-\cos\beta)}$ and $y=g(\beta)=-\frac{2q}{\beta}$.}
\label{lemma25-fig}
\end{figure}
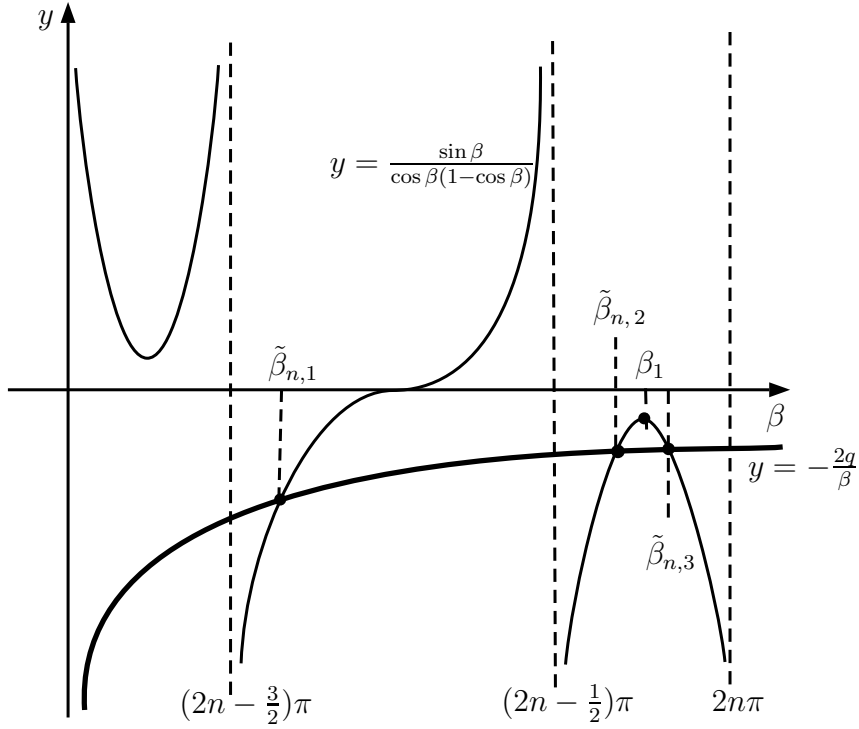

In the interval  $(2n-\frac{3}{2})\pi,\,(2n-1)\pi)$, we have
\begin{align*}
 \lim_{\beta\rightarrow ((2n-\frac{3}{2})\pi)^+}f(\beta)-g(\beta)=& -\infty,\, \lim_{\beta\rightarrow ((2n-1)\pi)^+}f(\beta)-g(\beta)=(2n-1)\pi.
\end{align*}By continuity of $f-g$ in the interval  $((2n-\frac{3}{2})\pi,\,(2n-1)\pi)$ and by the intermediate value theorem, $f-g$ has at least one zero. To show the uniqueness of the zero in the interval $((2n-\frac{3}{2})\pi,\,(2n-1)\pi)$, we consider the inverses of $f$ and $g$ restricted to  $((2n-\frac{3}{2})\pi,\,(2n-1)\pi)$, where $f$ and $g$ are increasing. We have
\begin{align*}
 \frac{\d}{\d \beta}(f^{-1}(y)-g^{-1}(y))= & \frac{(1-\cos\beta)\cos^2\beta}{1-\cos\beta-\cos^2\beta}-\frac{\beta^2}{2q}.
 \end{align*}
If there are multiple zeros, then by continuity the derivatives of $f^{-1}-g^{-1}$ evaluated at the zeros cannot be all negative or all positive.
Let $\beta_0\in((2n-\frac{3}{2})\pi,\,(2n-1)\pi)$ be a zero of $f-g$. Then we have
\begin{align}\label{2q-sine}
 2q\cos\beta_0(1-\cos\beta_0)=-\beta_0\sin\beta_0,
\end{align}and
\begin{align*}
 \frac{\d}{\d \beta}(f^{-1}(y)-g^{-1}(y))\vline_{\,\beta=\beta_0}
 = &\frac{(1-\cos\beta_0)\cos^2\beta_0}{1-\cos\beta_0-\cos^2\beta_0}-\frac{\beta_0^2}{2q}\\
 = &\frac{\cos\beta_0\left(\frac{-\beta_0\sin\beta_0}{2q}\right)}{1-\cos\beta_0-\cos^2\beta_0}-\frac{\beta_0^2}{2q}\\
 = & -\frac{\beta_0}{2q}\left(\frac{\cos\beta_0\sin\beta_0}{1-\cos\beta_0-\cos^2\beta_0}+\beta_0\right)\\
 = & -\frac{\beta_0}{2q}\left(\frac{\beta_0(1-\cos^2\beta_0)-\cos\beta_0(\beta_0-\sin\beta_0)}{1-\cos\beta_0-\cos^2\beta_0}\right)\\
 < &\,\, 0.
\end{align*}This is a contradiction. Therefore,  $f-g$  has a unique zero $\tilde{\beta}_{n,\,1}$ in $((2n-\frac{3}{2})\pi,\,(2n-1)\pi)$. 

Next we consider $f,\,g$ in the interval $\left((2n-\frac{1}{2})\pi,\,2n\pi\right)$ with the assumption $q\geq \beta_1\frac{\sqrt{2\sqrt{5}-2}}{4\sqrt{5}-8}$. To find the local maximum of $f$, we let $f'(\beta)=0$ and obtain
\[
 \frac{1-\cos\beta-\cos^2\beta}{(1-\cos\beta)\cos^2\beta}=0,
\]which leads to $\cos\beta=\frac{-1+\sqrt{5}}{2}$ and $\beta=\beta_1$ where $\beta_1=2n\pi-\arccos \frac{-1+\sqrt{5}}{2}$. Then $f$ is increasing in $(2(n-\frac{1}{2})\pi,\,\beta_1)$ and increasing in $(\beta_1,\,2n\pi)$. Moreover,  $q\geq \beta_1\frac{\sqrt{2\sqrt{5}-2}}{4\sqrt{5}-8}$ is equivalent to 
 \begin{align}\label{fg-beta1}
 f(\beta_1)\geq g(\beta_1).
 \end{align}

 Note that $\lim_{\beta\rightarrow ((2n-\frac{1}{2})\pi)^+}f(\beta)=-\infty< g((2n-\frac{1}{2})\pi)$. Then by the intermediate value theorem, $f-g$ has at least one zero in the interval $((2n-\frac{1}{2})\pi,\,\beta_1)$.
 
 Note that $1-\cos\beta-\cos^2\beta>0$ for $\beta\in ((2n-\frac{1}{2})\pi,\,\beta_1)$ with $0<\cos\beta<\frac{-1+\sqrt{5}}{2}$ and that $f$ and $g$ are increasing in $((2n-\frac{1}{2})\pi,\,\beta_1)$. Let $\beta_0\in ((2n-\frac{1}{2})\pi,\,\beta_1)$ be such that $f(\beta_0)=g(\beta_0)$. We have
\begin{align}
 \frac{\d}{\d \beta}(f^{-1}(y)-g^{-1}(y))\vline_{\,\beta=\beta_0}
 = &\frac{(1-\cos\beta_0)\cos^2\beta_0}{1-\cos\beta_0-\cos^2\beta_0}-\frac{\beta_0^2}{2q}\notag\\
 = &\frac{\cos\beta_0\left(\frac{-\beta_0\sin\beta_0}{2q}\right)}{1-\cos\beta_0-\cos^2\beta_0}-\frac{\beta_0^2}{2q}\notag\\
 = & -\frac{\beta_0}{2q}\left(\frac{\cos\beta_0\sin\beta_0}{1-\cos\beta_0-\cos^2\beta_0}+\beta_0\right)\notag\\
 = & -\frac{\beta_0}{2q}\left(\frac{\beta_0(\cos\beta_0-\sin^2\beta_0)+\cos\beta_0\sin\beta_0)}{1-\cos\beta_0-\cos^2\beta_0}\right)\notag\\
 = & -\frac{\beta_0}{2q}\left(\frac{\beta_0(\cos^2\beta_0+\cos\beta_0-1)+\cos\beta_0\sin\beta_0)}{1-\cos\beta_0-\cos^2\beta_0}\right)\notag\\
 > &\,\, 0,\label{fg-beta1-3}
\end{align} which combined with continuity of $f-g$, implies that  $f-g$  has a unique zero $\tilde{\beta}_{n,\,2}$ in $((2n-\frac{1}{2})\pi,\,\beta_1)$.

Next we consider $f,\,g$ in the interval $\left(\beta_1,\,\,2n\pi\right)$. Note that $\lim_{\beta\rightarrow (2n\pi)^+}f(\beta)=-\infty< g(2n\pi)$ which combined with (\ref{fg-beta1}) implies that $f-g$ has at least one zero in the interval $(\beta_1,\,2n\pi)$. Moreover, the zero is unique since $f-g$ is decreasing in the interval $(\beta_1,\,2n\pi)$. That is, there exists a unique zero $\tilde{\beta}_{n,\,2}\in (\beta_1,\,2n\pi)$.

ii)   We note from the first part of the proof of i) that the existence and uniqueness of $\tilde{\beta}_{n,\,1}\in ((2n-\frac{3}{2})\pi,\,(2n-1)\pi)$ is independent of the value of $q>0$. 
To complete the proof of i), we show that $f-g$ has no zero in 
$((2n-\frac{1}{2})\pi,\,2n\pi)$ if $0<q< \beta_1\frac{\sqrt{2\sqrt{5}-2}}{4\sqrt{5}-8}$ which is equivalent to 
 \begin{align}\label{fg-beta1-0}
 f(\beta_1)< g(\beta_1).
 \end{align}
 Notice that $\lim_{\beta\rightarrow ((2n-\frac{1}{2})\pi)^+}f(\beta)=-\infty< g((2n-\frac{1}{2})\pi)$. If there is a zero of $f-g$ in $((2n-\frac{1}{2})\pi,\,\beta_1)$, there must be at least two and the derivatives of $f^{-1}-g^{-1}$ there are not all positive or all negative. But this is contradicting (\ref{fg-beta1-3}). Therefore, there is no zero in  $((2n-\frac{1}{2})\pi,\,\beta_1)$.
 
 Next we consider $f-g$ in $(\beta_1,\,2n\pi)$. Since we have (\ref{fg-beta1-0}) and $\lim_{\beta\rightarrow (2n\pi)^+}f(\beta)=-\infty< g(2n\pi)$, if there is a zero of $f-g$ in $(\beta_1,\,2n\pi)$, there must be at least two and the derivatives of $f-g$ there are not all positive or all negative. But this is contradicting the fact that $f-g$ is decreasing in $(\beta_1,\,2n\pi)$. Therefore, $f-g$ has no zero in $((2n-\frac{1}{2})\pi,\,2n\pi)$ if $0<q< \beta_1\frac{\sqrt{2\sqrt{5}-2}}{4\sqrt{5}-8}$ and the unique zero is $\tilde{\beta}_{n,\,1}\in \left((2n-\frac{3}{2})\pi,\,(2n-1)\pi\right)$. \qed
\end{proof}
Now we are in the position to give a parameterization of $\delta>0$. 
\begin{theorem}\label{para-delta-1}
 Let $\xi$ and $q$ be positive constants with $\xi\neq 2p$. For every $n\in\mathbb{N}$, let $\beta_1=2n\pi-\arccos\frac{-1+\sqrt{5}}{2}$ and $\beta_n^*$ be the unique zero of the equation
 $\xi(1-\cos\beta)+\beta\sin\beta=0$. Let $\tilde{\beta}_{n,\,1}$, $\tilde{\beta}_{n,\,2}$ and $\tilde{\beta}_{n,\,3}$ be all possible zeros of $2q\cos\beta(1-\cos\beta)+\beta\sin\beta=0$. Let $\delta: (2(n-1)\pi,\,2n\pi)\rightarrow \mathbb{R}$ be defined by
 \[
  \delta=\frac{\xi\beta^2(1-\cos\beta)+\beta^3\sin\beta}{2q\cos\beta(1-\cos\beta)+\beta\sin\beta}.
 \]
 The following are true. 
 \begin{enumerate}
 \item[$i)$]If $q\geq\beta_1\frac{\sqrt{2\sqrt{5}-2}}{4\sqrt{5}-8}$, and $ \beta_n^*\leq \tilde{\beta}_{n,\,2}$, then $\delta(\beta)>0$ if and only if 
\[
\beta\in (2(n-1)\pi,\,\tilde{\beta}_{n,\,1})\cup (\beta_n^*,\,\tilde{\beta}_{n,\,2})\cup (\tilde{\beta}_{n,\,3},\,2n\pi).
\]
\item[$ii)$]If $q\geq\beta_1\frac{\sqrt{2\sqrt{5}-2}}{4\sqrt{5}-8}$, and $\tilde{\beta}_{n,\,2}\leq \beta_n^*\leq \tilde{\beta}_{n,\,3}$, then $\delta(\beta)>0$ if and only if 
\[
\beta\in (2(n-1)\pi,\,\tilde{\beta}_{n,\,1})\cup (\tilde{\beta}_{n,\,2},\,\beta_n^*)\cup (\tilde{\beta}_{n,\,3},\,2n\pi).
\]
\item[$iii)$]If $q\geq\beta_1\frac{\sqrt{2\sqrt{5}-2}}{4\sqrt{5}-8}$, and $\beta_n^*\geq \tilde{\beta}_{n,\,3}$, then $\delta(\beta)>0$ if and only if 
\[
\beta\in (2(n-1)\pi,\,\tilde{\beta}_{n,\,1})\cup (\tilde{\beta}_{n,\,2},\,\tilde{\beta}_{n,\,3})\cup (\beta_n^*,\,2n\pi).
\]
\item[$iv)$]If $0<q<\beta_1\frac{\sqrt{2\sqrt{5}-2}}{4\sqrt{5}-8}$, then $\delta(\beta)>0$ if and only if 
\[
\beta\in (2(n-1)\pi,\,\tilde{\beta}_{n,\,1})\cup (\beta_n^*,\,2n\pi).
\]
\end{enumerate}
\end{theorem}
\begin{proof}By Lemma~\ref{lemma-2-6} we know that if $q\geq\beta_1\frac{\sqrt{2\sqrt{5}-2}}{4\sqrt{5}-8}$, then
\[
 2q\cos\beta(1-\cos\beta)+\beta\sin\beta>0, 
\]is equivalent to
\begin{align}\label{delta-num-positive}
 \beta\in (2(n-1)\pi,\,\tilde{\beta}_{n,\,1})\cup (\tilde{\beta}_{n,\,2},\,\tilde{\beta}_{n,\,3}).
\end{align}
By Lemma~\ref{lemma-2-5}, we know that $\xi(1-\cos\beta)+\beta\sin\beta>0$ if and only if
\begin{align}\label{delta-den-positive}
 \beta\in (2(n-1)\pi,\,\beta_n^*)\cup (\beta_n^*,\,2n\pi).
\end{align}
Then by (\ref{delta-num-positive}) and (\ref{delta-den-positive}) and by choosing the intervals where $\xi(1-\cos\beta)+\beta\sin\beta$ and $2q\cos\beta(1-\cos\beta)+\beta\sin\beta$ are both positive or both negative, we obtain the conclusions of i), ii) and iii).

Similarly, by Lemma~\ref{lemma-2-6} we know that if $0<q<\beta_1\frac{\sqrt{2\sqrt{5}-2}}{4\sqrt{5}-8}$, then
\[
 2q\cos\beta(1-\cos\beta)+\beta\sin\beta>0, 
\]is equivalent to
\begin{align}\label{delta-num-positive-2}
 \beta\in (2(n-1)\pi,\,\tilde{\beta}_{n,\,1})\cup (\tilde{\beta}_{n,\,1},\,2n\pi).
\end{align}Then by (\ref{delta-num-positive-2}) and (\ref{delta-den-positive}) and by choosing the intervals where $\xi(1-\cos\beta)+\beta\sin\beta$ and $2q\cos\beta(1-\cos\beta)+\beta\sin\beta$ are both positive or both negative, we obtain the conclusions of iv).\qed
\end{proof}

\section{System with instantaneous spindle speed}\label{sec-instant-spindle}

Even though we knew that delayed spindle velocity control is more practically feasible than an instantaneous one, we are  interested what the trade-off between these two approaches could be.    In this section, we consider   model (\ref{SDDEs-system}) of turning processes with the following instantaneous spindle velocity control 
\begin{align}\label{instant-spindle}
 \Omega(t)= \frac{1}{R}(\dot{x}(t)+ cx(t)),
\end{align}
where $c\in\mathbb{R}$ is a parameter, which is the control strategy  considered in \cite{HKT-1}. We reconsider it here for convenience of comparison. Then equation (\ref{milling-eqn-3}) governing the state-dependent delay becomes
\begin{align}
\int_{t-\tau(t)}^t\frac{c}{2\pi R}\cdot x(s) \mathrm{d} s=1. \label{milling-eqn-8a}
\end{align}
System (\ref{milling-eqn-1-1}--\ref{milling-eqn-9})   with the spindle speed control strategy (\ref{milling-eqn-8a})  can be rewritten as
\begin{align}\label{SDDEs-system-control}
\left\{\begin{aligned}
 \frac{\mathrm{d}}{\mathrm{d} t}\begin{bmatrix}x\\ y\\  u \\ v\end{bmatrix}&=
\begin{bmatrix}
u\\
v\\
-\frac{c_x}{m}u-\frac{k_x}{m}x+\frac{K_x\omega}{m}(\nu\tau+y(t)-y(t-\tau(t)))^q\\
-\frac{c_y}{m}v-\frac{k_y}{m}y-\frac{K_y\omega}{m}(\nu\tau+y(t)-y(t-\tau(t)))^q
\end{bmatrix},\\
1&=\int_{t-\tau(t)}^t\frac{c}{2\pi R}\cdot x(s) \d s, 
\end{aligned}\right.
\end{align}where $u(t)=\dot{x}(t),\,v(t)=\dot{y}(t)$ for $t>0$. Assuming that $x(t)>0$ for all $t>0$ we put $\eta  =\int_{0}^t\frac{c}{2\pi R}\cdot x(s)\d s$ and consider the  change of variables  
$
 r(\eta)  =x(t),  
 \rho(\eta) = y(t),  
 j(\eta)  =u(t),  
 l(\eta)  = v(t),  
 k(\eta)  = \tau(t). 
$  
Then by (\ref{milling-eqn-8a})   we have
$
\eta-1= \int_{0}^{t-\tau(t)}\frac{c}{2\pi R}\cdot x(s)\d s,$                                                                                                               
  $r(\eta-1)  =x(t-\tau(t))$ and $
 \rho(\eta-1) = y(t-\tau(t)).$  
The second equation of (\ref{SDDEs-system})  for $\tau$ can be rewritten as
\begin{align*}
 \tau(t)  =t-(t-\tau(t)) 
          = \int_{\eta-1}^\eta \frac{\mathrm{d}t}{\mathrm{d}\bar{\eta}}\mathrm{d}\bar{\eta} 
          = \int_{\eta-1}^\eta \frac{2\pi R}{c}\frac{1}{r(\bar{\eta})}\mathrm{d}\bar{\eta} 
          =  \int_{-1}^0 \frac{2\pi R}{c}\frac{1}{r_{\eta}(s)}\mathrm{d} s.
\end{align*}
It follows that
$
 k(\eta) = \int_{-1}^0 \frac{2\pi R}{c}\frac{1}{r_{\eta}(s)}\mathrm{d} s.  
$   
   With the same process of the last section, we rewrite system (\ref{SDDEs-system})  as
\begin{align}\label{SDDEs-system-eta-2}\left\{\begin{aligned}
 \frac{\mathrm{d}}{\mathrm{d}\eta}\begin{bmatrix}r\\ \rho\\  j \\ l\end{bmatrix}&=
\begin{bmatrix}
j\\
l\\
-\frac{c_x}{m}j-\frac{k_x}{m}r+\frac{K_x\omega}{m}(\nu\,k+\rho-\rho(\eta-1))^q\\
-\frac{c_y}{m}l-\frac{k_y}{m}\rho-\frac{K_y\omega}{m}(\nu\,k+\rho-\rho(\eta-1))^q
\end{bmatrix}\frac{2\pi R}{c}\frac{1}{r(\eta)},\\
 k(\eta) & = \int_{-1}^0 \frac{2\pi R}{c}\frac{1}{r_{\eta}(s)}\d s.
\end{aligned}\right.
\end{align}
The unique stationary point of system (\ref{SDDEs-system-eta-2})   is the same as that of (\ref{SDDEs-system-eta}) if the parameter $c\in\mathbb{R}$ for the spindle velocities assumes the same value. Namely, we have the stationary state of system (\ref{SDDEs-system-eta-2}): 
 \begin{align}\label{equilibrium-control}
 (\bar{r},\,\bar{\rho},\,\bar{k},\,\bar{j},\,\bar{l})=\left(\frac{K_x\omega\nu^q}{k_x}{k^*}^q, \,-\frac{K_y\omega\nu^q}{k_y}{k^*}^q,\,{k^*},\,0,\,0\right),
 \end{align} where ${k^*}=\left(\frac{2\pi R}{ c}\cdot\frac{k_x}{K_x\omega\nu^q}\right)^{\frac{1}{q+1}}$.  Setting $\textbf{x} = (x_1,\,x_2,\,x_3,\,x_4) = (r, \rho, j, l) - (\bar{r}, \bar{\rho}, \bar{j}, \bar{l})$,
we obtain the linearized system of system (\ref{SDDEs-system-eta-2})  near its stationary state and obtain that 
\begin{align}\label{linearized-2}
\frac{\d \textbf{x}}{\d \eta} = k^*(M \textbf{x} + N \textbf{x}(\eta-1))+\frac{c {k^*}^3\nu}{2\pi R} \int_{-1}^{0} Q \textbf{x}_\eta (s) \d s.
\end{align}

By the same token leading to (\ref{virtual-spindle-speed}) we  require $k^*=\frac{2\pi}{\Omega_0}$ which is equivalent to 
 \begin{align}\label{c-value-2}
  c=\frac{R\Omega_{0}^{q+1} k_{x}}{(2 \pi \nu)^q K_{x} \omega},
 \end{align}which is the same as  (\ref{c-value-1}),
 then systems (\ref{SDDEs-system-eta-2}) and (\ref{SDDEs-system-eta}) both have the same stationary states as that of the corresponding system with constant spindle velocity $\Omega_0$. In the following we always assume $c$ satisfies (\ref{c-value-2}) such that the parameter substitution at (\ref{parameter-simplification}) is still valid.  Then   system~(\ref{linearized-2}) can be rewritten as:
\begin{align}\label{4-eqn-2}
\left\{
\begin{aligned}
\dot{x}_1= &\, k^* x_3,\\
 \dot{x}_2= &\, k^* x_4,\\
 \dot{x}_3 = &\, k^* \left(-\frac{k_x}{m}x_1 -\frac{c_x}{m}x_3+ \frac{k_{x}}{m} \frac{K_{1}p^{q-1}}{k_{r}}  (x_2- x_2(\eta-1))\right)\\
& - \frac{q k^* k_x}{m} \int_{-1}^{0} x_{1\eta}(s) \d s,\\
 \dot{x}_4  = &\, k^* \left(-\frac{k_y}{m}x_2 - \frac{c_y}{m}x_4 -\frac{k_{x}K_{1} p^{q-1} }{m} (x_2- x_2(\eta-1))\right)\\
& +  \frac{q k_xk^* k_r}{m} \int_{-1}^{0}  x_{1\eta}(s) \d s.
\end{aligned}\right.
\end{align}
Writing (\ref{4-eqn-2}) into second order scalar equations of  $(x_1,x_2)$, we have
\begin{align}\label{second-order-instant-model}
\left\{
\begin{aligned}
&\ddot{x}_{1} + \frac{k^* c_{x}}{m} \dot{x}_{1} + \frac{{k^*}^2 k_{x}}{m} x_{1} +  \frac{q {k^*}^2 k_x}{m} \int_{-1}^{0} x_{1\eta}(s) \d  s  = \frac{k_{x} K_{1} {k^*}^2p^{q-1}}{m k_{r}}  (x_{2}(\eta) - x_{2}(\eta -1)), \\
&  \ddot{x}_{2} + \frac{k^* c_{y}}{m} \dot{x}_{2} + \frac{{k^*}^2 k_{y}}{m} x_{2} + \frac{k_{x} K_{1} {k^*}^2p^{q-1}}{m}  (x_{2}(\eta) - x_{2}(\eta -1))  \\
 & \quad  =   \frac{q k_xk_r{k^*}^2 }{m} \int_{-1}^{0} x_{1\eta}(s) \d s.
\end{aligned}\right.
\end{align}
Assuming $c_x=c_y$ and $k_x=k_y$ and
bringing the ansatz $(x_1, x_2) =(c_1, c_2) e^{\lambda \eta}$ with $c_1\neq 0$ and $c_2 \neq 0$ into (\ref{second-order-instant-model}), we obtain the characteristic equation of system (\ref{SDDEs-system-eta-2}):
\begin{align}
\mathcal{P}(\lambda) \left(\mathcal{P}(\lambda) + \frac{k_x {k^*}^2}{m} \left(K_1 p^{q-1}+\frac{q}{\lambda}\right) (1- e^{-\lambda})\right)=0.
\end{align} Let $\delta$  be defined at (\ref{xi-delta-def}) and let
\begin{align}\label{h-2-gamma-def}
h_2=K_1 p^{q-1}.
\end{align} The characteristic equation of system (\ref{SDDEs-system-eta-2}) can be rewritten as:
\begin{align*}
\mathcal{P}(\lambda) \left(P(\lambda) +\delta \left(h_2+\frac{q}{\lambda}\right) (1- e^{-\lambda})\right)=0.
\end{align*}
Since the roots of the quadratic polynomial $\mathcal{P}(\lambda)$ always have negative real parts, we consider the roots of 
\begin{align}\label{charac-2}
 \mathcal{P}(\lambda) + \delta\left(h_2+\frac{q}{\lambda}\right) (1- e^{-\lambda})=0.
\end{align}

We first notice that $\lambda=0$ is not a removable singularity of $\mathcal{P}(\lambda) + \delta\left(h_2+\frac{q}{\lambda}\right) (1- e^{-\lambda})$ such that the limit there is zero since we have
\[
 \frac{\d}{\d\lambda}\lambda \left(\mathcal{P}(\lambda) + \delta\left(h_2+\frac{q}{\lambda}\right) (1- e^{-\lambda})\right)\,\vline_{\lambda=0}=\delta(q+1)\neq 0,
\]if $\delta\neq 0$.
Therefore, if $\lambda=i\beta$, $\beta\in\mathbb{R}$ is an eigenvalue then $\beta\neq 0$ and we have,
\begin{align}\label{real-imaginary-2}
\left\{
\begin{aligned}
 -\beta^2+\delta+\delta h_2(1-\cos\beta)+\frac{\delta q}{\beta}\sin\beta  & =  0,\\
 \xi\beta+\delta h_2\sin\beta-\frac{\delta q}{\beta}(1-\cos\beta)  &  =  0.
 \end{aligned}
 \right.
\end{align}

We have
\begin{lemma}\label{lemma-3-3}Suppose that $\xi$, $q$,  $h_2$ and $\delta$ are positive with $\xi\neq 2q$. If $\lambda=i\beta$, $\beta\in\mathbb{R}$  is a zero of $\mathcal{P}(\lambda) + \delta\left(h_2+\frac{q}{\lambda}\right) (1- e^{-\lambda})$, then $\beta\neq 0$ and the following statements are true:
\begin{enumerate}\item [ i$\,)$]  
\[
\max\left\{0,\,\delta-\frac{q\xi}{h_2}\right\}<\beta^2\leq \frac{\delta h_2+2\delta h_2^2-\xi q+\sqrt{(\delta h_2+2\delta h_2^2-\xi q)^2+8\delta h_2 q^2}}{2 h_2},
\]  and $\beta\neq 2n\pi$ for every $n\in\mathbb{Z}$.
\item[ ii$\,)$]We have
\begin{align}\label{lemma-3-3-1}\left\{
\begin{aligned}
h_2 = &\frac{q\beta(1-\cos\beta)-\xi\beta-q\xi\sin\beta}{\xi\beta(1-\cos\beta)+\beta^2\sin\beta},\\[0.25em]
\delta = & \frac{\xi\beta^2(1-\cos\beta)+\beta^3\sin\beta}{2q(1-\cos\beta)+\beta\sin\beta}.
\end{aligned}\right.
 \end{align} 
\item[iii$)$]   $\beta\in (2(n-1)\pi,\,2n\pi)$ or $\beta\in (-2n\pi,\,-2(n-1)\pi)$ for some $n\in\mathbb{N}$. 
\end{enumerate}
\end{lemma}
\begin{proof} i) Notice that by the second equation of (\ref{real-imaginary-2}), $1-\cos\beta\neq 0$, otherwise, $\sin\beta=0$ and hence $\xi\beta=0$ which is impossible. Therefore, we have $\beta\neq 2n\pi$ for every $n\in\mathbb{Z}$.

Solving (\ref{real-imaginary-2}) for $(\cos\beta,\,\sin\beta)$ we obtain that
\begin{align}\label{cos-sin-beta}\left\{
\begin{aligned}
\cos\beta & =\frac{h_2^2+\frac{q^2}{\beta^2}- h_2(\frac{\beta^2}{\delta}-1)-\frac{q\xi}{\delta}}{ h_2^2+\frac{q^2}{\beta^2}},\\
\sin\beta & =\frac{h_2^2+\frac{q}{\beta^2}-h_2\beta\frac{\xi}{\delta}}{h_2^2+\frac{q^2}{\beta^2}}.
\end{aligned}\right.
 \end{align} 
 Since 
 $-1\leq \cos\beta<1$, by the first equation of (\ref{cos-sin-beta}) we have,
 \[
  -1\leq \frac{h_2^2+\frac{q^2}{\beta^2}- h_2(\frac{\beta^2}{\delta}-1)-\frac{q\xi}{\delta}}{ h_2^2+\frac{q^2}{\beta^2}}<1,
 \]
which combined with $\beta\neq 0$ lead to
\begin{align*}
\max\left\{0,\,\delta-\frac{q\xi}{h_2}\right\}<\beta^2\leq \frac{\delta h_2+2\delta h_2^2-\xi q+\sqrt{(\delta h_2+2\delta h_2^2-\xi q)^2+8\delta h_2 q^2}}{2 h_2}.
\end{align*}

ii) By the first equation of (\ref{real-imaginary-2}), we have
\begin{align}\label{lemma3-3-1-h2}
 \delta h_2=\frac{\beta^2-\delta-\frac{\delta q}{\beta}\sin\beta}{1-\cos\beta},
\end{align} with which the second  equation of (\ref{real-imaginary-2}) becomes
\begin{align}\label{real-imaginary-3}
 \xi\beta+\sin\beta\left(\frac{\beta^2-\delta-\frac{\delta q}{\beta}\sin\beta}{1-\cos\beta}\right)-\frac{\delta q}{\beta}(1-\cos\beta)   =  0.
\end{align}
From (\ref{real-imaginary-3}) we obtain
\begin{align}\label{real-imaginary-4}
 \delta q=\frac{1}{2}\left(\xi\beta^2+\frac{\beta(\beta^2-\delta)\sin\beta}{1-\cos\beta}\right).
\end{align}Solving for $\delta$ from (\ref{real-imaginary-4}) we have
\begin{align}\label{real-imaginary-5}
 \left(q+\frac{\beta\sin\beta}{2(1-\cos\beta)}\right)\delta= \frac{1}{2}\xi\beta^2+\frac{\beta^3\sin\beta}{2(1-\cos\beta)}.
\end{align}
Notice that if $q+\frac{\beta\sin\beta}{2(1-\cos\beta)}=0$, (\ref{real-imaginary-4}) becomes $0=\frac{1}{2}\xi\beta^2-q\beta^2$ and hence $\xi=2q$, which is a contradiction to the assumption. Therefore, $q+\frac{\beta\sin\beta}{2(1-\cos\beta)}\neq 0$ and (\ref{real-imaginary-4}) implies that
\begin{align}\label{real-imaginary-6}
 \delta= \frac{\xi\beta^2(1-\cos\beta)+\beta^3\sin\beta}{2q(1-\cos\beta)+\beta\sin\beta}.
\end{align}
 Next we solve for $h_2$. Note that $\delta>0$ and hence $\xi(1-\cos\beta)+\beta\sin\beta\neq 0$. We bring $\delta$ at (\ref{real-imaginary-6}) into (\ref{lemma3-3-1-h2}) to obtain
\begin{align}\label{real-imaginary-7}
\begin{aligned}
 h_2=& \frac{\beta^2}{\delta(1-\cos\beta)}-\frac{1+\frac{q}{\beta}\sin\beta}{1-\cos\beta}\\
 = &\frac{2q(1-\cos\beta)+\beta\sin\beta}{(\xi(1-\cos\beta)+\beta\sin\beta)(1-\cos\beta)}-\frac{1+\frac{q}{\beta}\sin\beta}{1-\cos\beta}\\[0.25em]
 = & \frac{q\beta(1-\cos\beta)-\xi\beta-q\xi\sin\beta}{\xi\beta(1-\cos\beta)+\beta^2\sin\beta}.
  \end{aligned}
\end{align}
\qed
\end{proof}

\begin{lemma}\label{lemma-3-4}Assume  $\xi$,   $h_2$ and $\delta$ are positive.    If $\beta\in\mathbb{R}$ is such that  ii) and iii) of Lemma~\ref{lemma-3-3} are satisfied, then $\lambda=i\beta$ is a purely imaginary  zero of $\mathcal{P}(\lambda)+ \delta \left(h_2+\frac{q}{\lambda}\right)(1-e^{-\lambda})=0$.
\end{lemma}
 \begin{proof} We show that  (\ref{real-imaginary-2}) is true. Working backward at the derivation of $h_2$ at (\ref{real-imaginary-7}), we have
 \begin{align}\label{h2-backward}
  h_2=\frac{\beta^2}{\delta (1-\cos\beta)}-\frac{1+\frac{q}{\beta}\sin\beta}{1-\cos\beta},
 \end{align}which leads to the first equation of  (\ref{real-imaginary-2}).

 Similarly working backward at the derivation of $\delta$ at (\ref{real-imaginary-6}), we have (\ref{real-imaginary-4}) and 
 (\ref{real-imaginary-3}) which combined with (\ref{h2-backward}) we have the second equation of  (\ref{real-imaginary-2}). \hfill{\hspace*{1em}}\qed
\end{proof}
In the following we are interested to obtain a parameterization of all possible positive values of $h_2$ and $\delta$ for which $\mathcal{P}(\lambda)+ \delta \left(h_2+\frac{q}{\lambda}\right)(1-e^{-\lambda})=0$ has purely imaginary  zeros. Namely, we find out all possible values of $\beta$ such that   $h_2$ and $\delta$ parameterized at (\ref{lemma-3-3-1}) are both positive.

 We notice  from (\ref{lemma-3-3-1}) that $(h_2,\,\delta)$ is an even function of $\beta$ on its domain. In the following we assume that $\beta>0$.
 \begin{lemma}\label{lemma-3-3-2}
  Let $\xi$ and $q$ be positive with $\xi\neq 2q$. For every $n\in\mathbb{N}$,  the equation $2q(1-\cos\beta)+\beta\sin\beta=0$ has a unique zero  $\tilde{\beta}_n$ in $(2(n-1)\pi,\,2n\pi)$. Moreover, let $\beta_n^*\in (2(n-1)\pi,\,2n\pi)$ be the unique zero of $\xi(1-\cos\beta)+\beta\sin\beta=0$ obtained at Lemma~\ref{lemma-2-5}. The following are true:
  \begin{enumerate}
    \item[i)] if $\xi<2q$, then  $(2n-1)\pi<\beta_n^*<\tilde{\beta}_n<2n\pi$;
   \item[ii)] if $\xi>2q$, then  $(2n-1)\pi<\tilde{\beta}_n<\beta_n^*<2n\pi$.
  \end{enumerate}
 \end{lemma}
 \begin{proof}
  Note that $\cos\beta\neq 1$ in $(2(n-1)\pi,\,2n\pi)$. Consider
  \begin{align}\label{zero-1}
   \frac{2q}{\beta}+\frac{\sin\beta}{1-\cos\beta}=0.
  \end{align}Note that
  \begin{align}
   \frac{\d}{\d \beta}\left(\frac{2q}{\beta}+\frac{\sin\beta}{1-\cos\beta}\right)= -\frac{2q}{\beta^2}-\frac{1}{(1-\cos\beta)}<0,
  \end{align}
and that $\lim_{\beta\rightarrow (2n\pi)^-} \frac{2q}{\beta}+\frac{\sin\beta}{1-\cos\beta}=+\infty$, $\lim_{\beta\rightarrow (2(n-1)\pi)^+} \frac{2q}{\beta}+\frac{\sin\beta}{1-\cos\beta}=-\infty$.
It follows that the equation  $2q(1-\cos\beta)+\beta\sin\beta=0$ has a unique zero  $\tilde{\beta}_n$  in $(2(n-1)\pi,\,2n\pi)$.
Notice that by (\ref{zero-1}) we have $\frac{\sin\beta}{1-\cos\beta}<0$. Therefore,  we have
that $\tilde{\beta}_n$ is in the interval $((2n-1)\pi,\,(2n-1)\pi)$. Moreover, if $\xi>2q$, we have $\beta_n^*\cot\frac{\beta_n^*}{2}=-\xi<-2q=\tilde{\beta}_n\cot\frac{\tilde{\beta}_n}{2}$. Since $\beta\rightarrow\beta\cot\frac{\beta}{2}$ is a decreasing function in $(2(n-1)\pi,\,(2n-1)\pi)$ with derivative $\frac{\sin\beta-\beta}{1-\cos\beta}<0$, we obtain that $\beta_n^*>\tilde{\beta}_n$. Similarly, if $\xi<2q$, we have $\beta_n^*<\tilde{\beta}_n$.\qed
 \end{proof}

Now we obtain a parameterization for $\delta>0$.
\begin{theorem}\label{delta-parameterization}
 Let $\xi$ and $q$ be positive numbers with $\xi\neq 2q$, $\beta_n^*$ and $\tilde{\beta}_n$ be the unique zeros of the equations $\xi(1-\cos\beta)+\beta\sin\beta=0$ and $2q(1-\cos\beta)+\beta\sin\beta=0$  in $(2(n-1)\pi,\,2n\pi)$, respectively. Define the mapping $\delta: (2(n-1)\pi,\,2n\pi)\rightarrow\mathbb{R}$ by
 \begin{align}\label{delta-def}
  \delta(\beta)=\frac{\xi\beta^2(1-\cos\beta)+\beta^3\sin\beta}{2q(1-\cos\beta)+\beta\sin\beta}.
 \end{align}
Then the following are true:
 \begin{enumerate}
  \item[i)] If $\xi<2q$,  $\delta(\beta)>0$ if and only if 
  $
   \beta\in (2(n-1)\pi,\,\beta_n^*)\cup (\tilde{\beta}_n,\,2n\pi).
 $
\item[ii)] If $\xi>2q$,  $\delta(\beta)>0$ if and only if
 $
   \beta\in (2(n-1)\pi,\,\tilde{\beta}_n)\cup ({\beta}_n^*,\,2n\pi).
$

 \end{enumerate}

\end{theorem}
\begin{proof}
 Note that $\delta>0$ is equivalent to  
 \[
 (\xi(1-\cos\beta)+\beta\sin\beta)\cdot(2q(1-\cos\beta)+\beta\sin\beta)>0,
 \]
 and that $\xi(1-\cos\beta)+\beta\sin\beta$ and $2q(1-\cos\beta)+\beta\sin\beta$ changes from positive to negative 
when $\beta$ passes through the zeros $\beta_n^*$ and $\tilde{\beta}_n$ from $-\infty$, respectively. Then by Lemma~\ref{lemma-3-3-2}, the conclusions are true.\qed 
\end{proof}

Next we obtain a parameterization for $h_2>0$. For this purpose, we determine in the following lemma the one-sided limits of $h_2$ at its vertical asymptote  $\beta=\beta_n^*$ where $\beta_n^*$ is the zero  of $\delta$ in $(2(n-1)\pi,\,2n\pi)$.
\begin{lemma}\label{h_2-parameterization}
 Let $\xi$ and $q$ be positive numbers with $\xi\neq 2q$, $\beta_n^*$ be the unique zero of the equations $\xi(1-\cos\beta)+\beta\sin\beta=0$ with $\beta\in(2(n-1)\pi,\,2n\pi)$. Define the mapping $h_2: (2(n-1)\pi,\,2n\pi)\rightarrow\mathbb{R}$ 
 \begin{align}\label{delta-function}
   h_2(\beta)=\frac{q\beta(1-\cos\beta)-\xi\beta-q\xi\sin\beta}{\xi\beta(1-\cos\beta)+\beta^2\sin\beta}.
 \end{align}
Then we have $\lim_{\beta\rightarrow(2n\pi)^{-}}h_2(\beta)=+\infty$ $\lim_{\beta\rightarrow(2(n-1)\pi)^{+}}h_2(\beta)=-\infty$. Moreover, the following are true:
 \begin{enumerate}
  \item[i)] If $\xi<2q$, we have \begin{align*}
  \begin{aligned}
\lim_{\beta\rightarrow(\beta_n^*)^+}h_2(\beta)= & -\infty,\, \lim_{\beta\rightarrow(\beta_n^*)^-}h_2(\beta)=  +\infty,\, \lim_{\beta\rightarrow\tilde{\beta}_n}h_2(\beta) < 0,
  \end{aligned}\end{align*}
 \item[ii)] If $\xi>2q$,  we have \begin{align*}
  \begin{aligned}
\lim_{\beta\rightarrow(\beta_n^*)^+}h_2(\beta)= & +\infty,\, \lim_{\beta\rightarrow(\beta_n^*)^-}h_2(\beta)=  -\infty,\, \lim_{\beta\rightarrow\tilde{\beta}_n}h_2(\beta) >0. 
  \end{aligned}\end{align*}\end{enumerate}
\end{lemma}
\begin{proof}
 Notice that by Lemma~\ref{lemma-3-3-2}, we have
 \begin{align}\label{denom-h2}
  \xi\beta(1-\cos\beta)+\beta^2\sin\beta<0,
 \end{align}
for $\beta\in (\max\{\beta_n^*,\,\tilde{\beta}_n\},\,2n\pi)$ and that \[\lim_{\beta\rightarrow(2n\pi)^{-}}q\beta(1-\cos\beta)-\xi\beta-q\xi\sin\beta=-2n\pi\xi<0.
\]
It follows that $\lim_{\beta\rightarrow(2n\pi)^{-}}h_2(\beta)=+\infty$. Similarly, we have
 \[
  \xi\beta(1-\cos\beta)+\beta^2\sin\beta>0,
 \]
for $\beta\in (2(n-1)\pi,\,(2n-1)\pi)$ and  \[\lim_{\beta\rightarrow(2(n-1)\pi)^{+}}q\beta(1-\cos\beta)-\xi\beta-q\xi\sin\beta=-2(n-1)\pi\xi<0.
\]
It follows that $\lim_{\beta\rightarrow(2(n-1)\pi)^{+}}h_2(\beta)=-\infty$.
 
 Next we compute the one-sided limits of $h_2$ at $\beta_n^*$. Note by Lemma~\ref{lemma-3-3-2}, $\sin\beta_n^*=-\sqrt{1-\cos^2\beta_n^*}<0$ and
 \[
  \xi(1-\cos\beta_n^*)+\beta_n^*\sin\beta_n^*=0.
 \]It follows that $\xi(1-\cos\beta_n^*)-\beta_n^*\sqrt{1-\cos^2\beta_n^*}=0$ which leads to
 \[
  (\xi^2+{\beta_n^*}^2)\cos^2\beta_n^*-2\xi^2\cos\beta_n^*+ (\xi^2-{\beta_n^*}^2)=0,
 \] and hence $\cos\beta_n^*=\frac{\xi^2-{\beta_n^*}^2}{\xi^2+{\beta_n^*}^2}$. Then we have $\sin\beta_n^*=-\sqrt{1-\cos^2\beta_n^*}=-\frac{2{\beta_n^*}\xi}{\xi^2+{\beta_n^*}^2}$. With the expressions of $\sin\beta_n^*$ and $\cos\beta_n^*$, we have
 \begin{align}\label{numer-h2}
  q\beta_n^*(1-\cos\beta_n^*)-\xi\beta_n^*-q\xi\sin\beta_n^*=\frac{(2q-\xi)(\xi^2+{\beta_n^*}^2)\beta_n^*}{\xi^2+{\beta_n^*}^2},
 \end{align}which is positive if $\xi<2q$ and negative if $\xi>2q$.

  If $\xi<2q$, then by Lemma~\ref{lemma-3-3-2}, we have
 \begin{align}\label{demo-h2}
  \xi\beta(1-\cos\beta)+\beta^2\sin\beta=\begin{cases}
                                          >0, & \mbox{if $\beta\in ((2n-1)\pi,\,\,\beta_n^*)$},\\
                                          <0, & \mbox{if $\beta\in (\beta_n^*,\,2n\pi)$}.
                                         \end{cases}
 \end{align}
By (\ref{numer-h2})  and (\ref{demo-h2}), we obtain that
if $\xi<2q$, then we have $\lim_{\beta\rightarrow(\beta_n^*)^+}h_2(\beta)=-\infty$ and $\lim_{\beta\rightarrow(\beta_n^*)^-}h_2(\beta)=+\infty$.

If $\xi>2q$, then by Lemma~\ref{lemma-3-3-2}, we have
 \begin{align}\label{demo-h2-2}
  \xi\beta(1-\cos\beta)+\beta^2\sin\beta=\begin{cases}
                                          >0, & \mbox{if $\beta\in ((2n-1)\pi,\,\,\beta_n^*)$},\\
                                          <0, & \mbox{if $\beta\in (\beta_n^*,\,2n\pi)$}.
                                         \end{cases}
 \end{align}
By (\ref{numer-h2})  and (\ref{demo-h2-2}), we obtain that
if $\xi>2q$, then we have $\lim_{\beta\rightarrow(\beta_n^*)^+}h_2(\beta)=+\infty$ and $\lim_{\beta\rightarrow(\beta_n^*)^-}h_2(\beta)=-\infty$.

The last step is to compute the limit of $h_2$ at $\tilde{\beta}_n$. Note by Lemma~\ref{lemma-3-3-2}, $\sin\tilde{\beta}_n=-\sqrt{1-\cos^2\tilde{\beta}_n}<0$ and
 \[
  2q(1-\cos\tilde{\beta}_n)+\tilde{\beta}_n\sin\tilde{\beta}_n=0.
 \]It follows that $2q(1-\cos\tilde{\beta}_n)-\tilde{\beta}_n\sqrt{1-\cos^2\tilde{\beta}_n}=0$ which leads to
 \[
  (4q^2+{\tilde{\beta}_n}^2)\cos^2\tilde{\beta}_n-8q^2\cos\tilde{\beta}_n+ (4q^2-{\tilde{\beta}_n}^2)=0,
 \] and hence $\cos\tilde{\beta}_n=\frac{4q^2-{\tilde{\beta}_n}^2}{4q^2+{\tilde{\beta}_n}^2}$. Then we have $\sin\tilde{\beta}_n=-\sqrt{1-\cos^2\tilde{\beta}_n}=-\frac{4{\tilde{\beta}_n}q}{4q^2+{\tilde{\beta}_n}^2}$. With the expressions of $\sin\tilde{\beta}_n$ and $\cos\tilde{\beta}_n$, we have
 \begin{align}\label{numer-h2}
  q\tilde{\beta}_n(1-\cos\tilde{\beta}_n)-\xi\tilde{\beta}_n-q\xi\sin\tilde{\beta}_n=\frac{(2q-\xi)\tilde{\beta}_n^3}{4q^2+{\tilde{\beta}_n}^2},
 \end{align}which is positive if $\xi<2q$ and negative if $\xi>2q$. Then by (\ref{denom-h2}), we have $\lim_{\beta\rightarrow \tilde{\beta}_n}h_2<0$
 if $\xi<2q$ and  $\lim_{\beta\rightarrow \tilde{\beta}_n}h_2>0$ if $\xi>2q$. 
\qed
\end{proof}
Now we show that there are exactly two zeros of $h_2$ in $(2(n-1)\pi,\,2n\pi)$, $n\in\mathbb{N}$ under the assumption that $\xi<2q$.
\begin{theorem}\label{h_2-parameterization-2}
 Let $\xi$ and $q$ be positive numbers with $\xi< 2q$, $\beta_n^*$ and $\tilde{\beta}_n$ be the unique zeros of the equations $\xi(1-\cos\beta)+\beta\sin\beta=0$ and $2q(1-\cos\beta)+\beta\sin\beta=0$  in $(2(n-1)\pi,\,2n\pi)$, respectively. Let the map $h_2: (2(n-1)\pi,\,2n\pi)\rightarrow\mathbb{R}$ be defined at (\ref{delta-function}). Then $h_2$ has exactly two zeros $\gamma_n^*\in (2(n-1)\pi,\,\,\beta_n^*)$ and $\tilde{\gamma}_n\in (\tilde{\beta}_n,\,2n\pi)$. 
 Moreover, $h_2>0$ if and only if  $\beta\in (\gamma_n^*,\,\beta_n^*)\cup(\tilde{\gamma}_n,\,2n\pi)$.
\end{theorem}
\begin{proof} We first show the existence of the zeros of $h_2$. By Lemma~\ref{h_2-parameterization},   if $\xi<2q$, we have  \begin{align*}
\lim_{\beta\rightarrow(\beta_n^*)^-}h_2(\beta)=+\infty.\end{align*}Moreover we have
 \begin{align*}
\lim_{\beta\rightarrow(2(n-1)\pi)^+}h_2(\beta)=-\infty,\, 
\end{align*}By the intermediate value theorem and the continuity of $h_2$ in the interval $(2(n-1)\pi,\,\,\beta_n^*)$, $h_2$ has at least one zero $\gamma_n^*\in (2(n-1)\pi,\,\,\beta_n^*)$. 

Similarly, by  Lemma~\ref{h_2-parameterization}, we have $\lim_{\beta\rightarrow(2n\pi)^+}h_2(\beta)=+\infty$, and   $\lim_{\beta\rightarrow\tilde{\beta}_n}h_2(\beta)>0$ if $\xi<2q$.
By the intermediate value theorem and the continuity of $h_2$ in the interval $(\tilde{\beta}_n,\,2n\pi)$, $h_2$ has at least one zero $\tilde{\gamma}_n\in (\tilde{\beta}_n,\,2n\pi)$.

Next we show that $h_2$ has exactly two zeros in $(2(n-1)\pi,\,2n\pi)$. Let $h_2(\beta)=0$. We have
\[
 q\beta(1-\cos\beta)-\xi\beta-q\xi\sin\beta=0,
\]and hence $\sin\beta=\frac{\beta}{\xi}(1-\cos\beta)-\frac{\beta}{q}$. Then we have
\[
 \left(\frac{\beta}{\xi}(1-\cos\beta)-\frac{\beta}{q}\right)^2+\cos^2\beta=1,
\]which leads to
\begin{align}
\cos\beta=\frac{\frac{1}{\xi^2}-\frac{1}{q\xi}\pm\frac{1}{\beta}\sqrt{\frac{1}{\beta^2}+\frac{2}{q\xi}-\frac{1}{q^2}}}{\frac{1}{\xi^2}+\frac{1}{\beta^2}},\label{cos-h2-1}\intertext{and}
\sin\beta=\frac{\frac{1}{\beta\xi}+\frac{\beta}{q\xi^2}\mp\frac{1}{\xi}\sqrt{\frac{1}{\beta^2}+\frac{2}{q\xi}-\frac{1}{q^2}}}{\frac{1}{\xi^2}+\frac{1}{\beta^2}}-\frac{\beta}{q},\label{cos-h2-2}
\end{align} where  $\pm$ and $\mp$ indicate the correspondent plus and minus signs for two sets of the expressions of $(\cos\beta,\,\sin\beta)$. Then we have
\begin{align}\label{h2-roots}
 \cot\frac{\beta}{2}=\frac{\sin\beta}{1-\cos\beta}=-\frac{q}{\beta}\mp q\sqrt{\frac{1}{\beta^2}+\frac{2}{q\xi}-\frac{1}{q^2}}.
 \end{align}
Notice that the function $y=\cot\frac{\beta}{2}$ is decreasing in $(2(n-1)\pi,\,2n\pi)$ from $+\infty$ to $-\infty$ and that if $\xi<2q$,
\[
 \frac{\d}{\d\beta}\left(-\frac{q}{\beta}\mp q\sqrt{\frac{1}{\beta^2}+\frac{2}{q\xi}-\frac{1}{q^2}}\right)=\frac{q}{\beta^2}\pm\frac{q^2}{\beta^3}\frac{1}{\sqrt{\frac{1}{\beta^2}+\frac{2}{q\xi}-\frac{1}{q^2}}}>0,
\]which implies that  the functions $y=-\frac{q}{\beta}\mp q\sqrt{\frac{1}{\beta^2}+\frac{2}{q\xi}-\frac{1}{q^2}}$ increases in $(0,\,+\infty)$ from negative to $0$. Therefore, each of the equations at (\ref{h2-roots}) has a unique solution in $(2(n-1)\pi,\,2n\pi)$.

Finally the assertion that $h_2>0$ holds if and only if  $\beta\in (2(n-1)\pi,\,\,\gamma_n^*)\cup(\tilde{\gamma}_n,\,2n\pi)$ follows from the limits of $h_2$ obtained at Lemma~\ref{h_2-parameterization} and the uniqueness of the zeros of $h_2$ in the intervals $ (2(n-1)\pi,\,\,\beta_n^*)$ and $(\beta_n^*,\,2n\pi)$.  \qed
\end{proof}
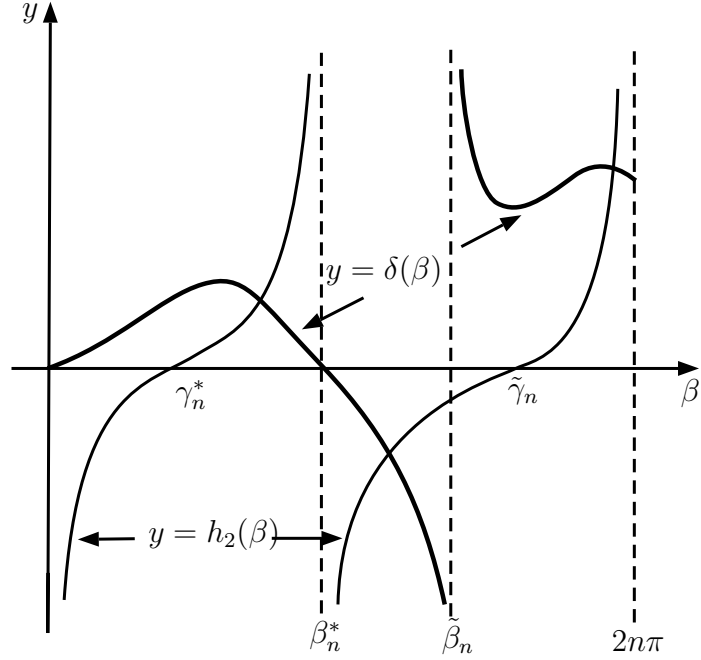
\begin{figure}
\begin{center}
\psscalebox{1.0 1.0} 
{
\begin{pspicture}(0,-4.305)(9.140031,4.305)
\psline[linecolor=black, linewidth=0.04, arrowsize=0.05291667cm 2.0,arrowlength=1.4,arrowinset=0.0]{<-}(0.517327,4.374964)(0.482673,-4.0249643)(0.48597336,-3.224971)
\psline[linecolor=black, linewidth=0.04, arrowsize=0.05291667cm 2.0,arrowlength=1.4,arrowinset=0.0]{->}(0.0,-0.505)(9.14,-0.505)
\psline[linecolor=black, linewidth=0.04, linestyle=dashed, dash=0.17638889cm 0.10583334cm](8.28,3.715)(8.28,-3.885)
\psline[linecolor=black, linewidth=0.04, linestyle=dashed, dash=0.17638889cm 0.10583334cm](4.12,3.695)(4.12,-3.765)
\psline[linecolor=black, linewidth=0.04, linestyle=dashed, dash=0.17638889cm 0.10583334cm](5.84,3.735)(5.84,-3.805)
\psbezier[linecolor=black, linewidth=0.04](3.96,3.415)(3.7304747,0.14498419)(3.0456886,0.03734841)(2.4271681,-0.33995037904359376)(1.8086478,-0.71724916)(0.88563037,-0.7652532)(0.7,-3.585)
\psbezier[linecolor=black, linewidth=0.04](8.06,3.215)(7.973325,-0.38533452)(6.980326,-0.36354187)(6.690818,-0.5106068289384712)(6.4013095,-0.6576718)(4.3951726,-1.2111785)(4.34,-3.645)
\psbezier[linecolor=black, linewidth=0.06](0.48,-0.505)(1.405049,-0.18011697)(1.9800149,0.4495356)(2.54,0.615)(3.0999851,0.7804644)(3.245049,0.45988303)(4.14,-0.485)(5.034951,-1.429883)(5.479985,-2.3395357)(5.76,-3.645)
\psbezier[linecolor=black, linewidth=0.06](5.98,3.475)(5.98211,3.0045302)(6.1648407,1.8500073)(6.458247,1.6927777777777875)(6.751653,1.5355482)(7.043943,1.7293786)(7.478008,2.054394)(7.912073,2.3794093)(8.299462,1.983225)(8.3,1.995)
\rput[bl](3.96,-4.265){$\beta_n^*$}
\rput[bl](5.72,-4.305){$\tilde{\beta}_n$}
\rput[bl](7.98,-4.265){$2n\pi$}
\rput[bl](2.18,-1.045){$\gamma_n^*$}
\rput[bl](6.6,-0.985){$\tilde{\gamma}_n$}
\rput[bl](8.9,-1.025){$\beta$}
\rput[bl](0.14,4.095){$y$}
\rput[bl](1.82,-2.925){$y=h_2(\beta)$}
\psline[linecolor=black, linewidth=0.04, arrowsize=0.05291667cm 2.0,arrowlength=1.4,arrowinset=0.0]{<-}(4.4,-2.765)(3.46,-2.765)
\psline[linecolor=black, linewidth=0.04, arrowsize=0.05291667cm 2.0,arrowlength=1.4,arrowinset=0.0]{<-}(0.88137656,-2.779742)(1.6386235,-2.770258)
\rput[bl](4.18,0.575){$y=\delta(\beta)$}
\psline[linecolor=black, linewidth=0.04, arrowsize=0.05291667cm 2.0,arrowlength=1.4,arrowinset=0.0]{->}(4.68,0.435)(3.86,0.015)
\psline[linecolor=black, linewidth=0.04, arrowsize=0.05291667cm 2.0,arrowlength=1.4,arrowinset=0.0]{->}(5.98,1.115)(6.78,1.535)
\end{pspicture}
}
\end{center}\caption{ Graphs of $y=\delta(\beta)=\frac{\xi\beta^2(1-\cos\beta)+\beta^3\sin\beta}{2q(1-\cos\beta)+\beta\sin\beta}$ and $y=h_2(\beta)=-\frac{q\beta(1-\cos\beta)-\xi\beta-q\xi\sin\beta}{\xi\beta(1-\cos\beta)+\beta^2\sin\beta}$ with $\xi<2q$.}
\label{delta-h2-positivity}
\end{figure}

 An immediate consequence of Theorems~\ref{delta-parameterization} and \ref{h_2-parameterization-2} is the following parameterization of $(\delta,\,h_2)$ where $\delta>0$ and $h_2>0$. See Figure~\ref{delta-h2-positivity}
for a demonstration.
 \begin{theorem}\label{delta-h2-parameterization}
  Let $\xi$ and $q$ be positive numbers with $\xi< 2q$, $\beta_n^*$ and $\tilde{\beta}_n$ be the unique zeros of the equations $\xi(1-\cos\beta)+\beta\sin\beta=0$ and $2q(1-\cos\beta)+\beta\sin\beta=0$  in $(2(n-1)\pi,\,2n\pi)$, respectively. Let the map $(\delta, h_2): (2(n-1)\pi,\,2n\pi)\rightarrow\mathbb{R}^2$ be defined at (\ref{delta-def}) and (\ref{delta-function}), respectively. Then $(\delta,\,h_2)(\beta)$ is positive if and only if 
  \[
  \beta\in (\gamma_n^*,\,\beta_n^*)\cup(\tilde{\gamma}_n,\,2n\pi),
  \] where  $\gamma_n^*\in (2(n-1)\pi,\,\,\beta_n^*)$ and $\tilde{\gamma}_n\in (\tilde{\beta}_n,\,2n\pi)$ are two zeros of $h_2$ in $I_n$.
 \end{theorem}

Now we consider the zeros of $h_2$ under the assumption that $\xi>2q$.
\begin{theorem}\label{h_2-parameterization-3}
 Let $\xi$ and $q$ be positive numbers with $\xi>2q$, and $n_0\in\mathbb{N}$ be such that
 \[
  2(n_0-1)\pi<q\sqrt{\frac{\xi}{\xi-2q}}\leq 2n_0\pi.
 \] Let 
$\beta_n^*$ and $\tilde{\beta}_n$ be the unique zeros of the equations $\xi(1-\cos\beta)+\beta\sin\beta=0$ and $2q(1-\cos\beta)+\beta\sin\beta=0$  in $(2(n-1)\pi,\,2n\pi)$, respectively. Let the map $h_2: (2(n-1)\pi,\,2n\pi)\rightarrow\mathbb{R}$ be defined at (\ref{delta-function}).  The following are true:
\begin{enumerate}
 \item[i)] For every $n>n_0$, $n\in\mathbb{N}$, 
 $
  h_2(\beta)>0
 $ if and only if $\beta\in (\beta_n^*,\,2n\pi)$.
 \item[ii)] For every $1\leq n\leq n_0$, $n\in\mathbb{N}$, with
  \[
  2(n-1)\pi<q\sqrt{\frac{\xi}{\xi-2q}}\leq (2n-1)\pi.
 \]
 $
  h_2(\beta)>0
 $ if and only if $\beta\in (\beta_n^*,\,2n\pi)$.
 \item[iii)]  For every $1\leq n\leq n_0$, $n\in\mathbb{N}$, with
  \[
  (2n-1)\pi<q\sqrt{\frac{\xi}{\xi-2q}},
 \] let $\bar{\beta}_n$ be the unique solution of $\beta\cot\beta=-q$ in $(2(n-1)\pi,\,2n\pi)$. 
 \begin{enumerate}
  \item[a)] If $\bar{\beta}_n>q\sqrt{\frac{\xi}{\xi-2q}}$, then
 $
  h_2(\beta)>0
 $ if and only if $\beta\in (\beta_n^*,\,2n\pi)$.
  \item[b)] If $\bar{\beta}_n\leq q\sqrt{\frac{\xi}{\xi-2q}}$, then $h_2$ has exactly two zeros $\gamma_n^*$ and $\tilde{\gamma}_n$ in $(2(n-1)\pi,\,\tilde{\beta}_n)$. Moreover,
 $
  h_2(\beta)>0
 $ if and only if $\beta\in (\gamma_n^*,\,\tilde{\gamma}_n)\cup(\beta_n^*,\,2n\pi)$.
 \end{enumerate}
 
\end{enumerate}

\end{theorem}
\begin{proof} i) By (\ref{cos-h2-1}) at the proof of Theorem~\ref{h_2-parameterization-2}, we know that if 
\begin{align}\label{discriminant}
 \frac{1}{\beta_0^2}+\frac{2}{q\xi}-\frac{1}{q^2}<0,
\end{align}
that is, $\beta_0>q\sqrt{\frac{\xi}{\xi-2q}}$, then $\beta_0$ is not a zero of $h_2$. Therefore, for every $n>n_0$, where $n_0\in\mathbb{N}$ is such that $2n_0\pi\geq q\sqrt{\frac{\xi}{\xi-2q}}$, 
 $
  h_2(\beta)>0
 $ if and only if $\beta\in (\beta_n^*,\,2n\pi)$.
 
 ii) For every $1\leq n\leq n_0$, $n\in\mathbb{N}$, with
  \[
  2(n-1)\pi<q\sqrt{\frac{\xi}{\xi-2q}}\leq (2n-1)\pi,
 \]we know from (\ref{discriminant}) that all possible zeros of $h_2$ in $(2(n-1)\pi,\,2n\pi)$ are in the interval $\left(2(n-1)\pi,\,q\sqrt{\frac{\xi}{\xi-2q}}\,\right]$. But by (\ref{h2-roots}), we know that $y=\cot\frac{\beta}{2}$ has no intersection in this interval  with  the functions $y=-\frac{q}{\beta}\mp q\sqrt{\frac{1}{\beta^2}+\frac{2}{q\xi}-\frac{1}{q^2}}$ which increase in $(0,\,+\infty)$ from negative to $0$.
 Therefore, in this case $h_2$ has no zero in $(2(n-1)\pi,\,2n\pi)$ and $
  h_2(\beta)>0
 $ if and only if $\beta\in (\beta_n^*,\,2n\pi)$.
 
 iii) a) We know from the derivation for (\ref{h2-roots}) that $\beta\in (2(n-1)\pi,\,2n\pi)$ is a zero of $h_2$ if and only if
  (\ref{h2-roots}) is satisfied. Multiply both sides of (\ref{h2-roots}) by $\beta$ and define the functions, $F: (2(n-1)\pi,\,2n\pi)\rightarrow\mathbb{R}$,\, $G: \left((2n-1)\pi,\,q\sqrt{\frac{\xi}{\xi-2q}}\right)\rightarrow(-q,\,0)$ and  and $H: \left((2n-1)\pi,\,q\sqrt{\frac{\xi}{\xi-2q}}\right)\rightarrow(-2q,\,-q)$ by
  \begin{align}
   y= & F(\beta)=\beta\cot\frac{\beta}{2},\label{fgh-1}\\
   y= & G(\beta)= -q+ \sqrt{q^2-\left(\frac{\xi-2q}{\xi}\right)\beta^2},\label{fgh-2}\\
     y= & H(\beta)= -q- \sqrt{q^2-\left(\frac{\xi-2q}{\xi}\right)\beta^2}.\label{fgh-3}
  \end{align}
Then the zeros of $h_2$ are the zeros of $F-G$ and $F-H$ in the interval $(2(n-1)\pi,\,2n\pi)$.  See Figure~\ref{lemma352-fig} for the graphs of $F$, $G$ and $H$.
\begin{figure}[t]
\begin{center}
\psscalebox{1.0 1.0} 
{
\begin{pspicture}(0,-4.793479)(10.366958,4.793479)
\psline[linecolor=black, linewidth=0.04, arrowsize=0.05291667cm 2.0,arrowlength=1.4,arrowinset=0.0]{->}(0.86,-4.633479)(0.86,4.886521)
\psline[linecolor=black, linewidth=0.04, arrowsize=0.05291667cm 2.0,arrowlength=1.4,arrowinset=0.0]{->}(0.06,-0.273479)(10.46,-0.273479)
\psline[linecolor=black, linewidth=0.04, linestyle=dashed, dash=0.17638889cm 0.10583334cm](9.66,-4.193479)(9.66,4.486521)
\psline[linecolor=black, linewidth=0.04, linestyle=dashed, dash=0.17638889cm 0.10583334cm](5.32,-4.333479)(5.32,4.346521)
\psbezier[linecolor=black, linewidth=0.04](0.89527947,-4.2892413)(6.8703475,-4.6138926)(9.399724,-3.280002)(9.19623,-2.1748401552332144)(8.992737,-1.0696782)(6.486736,-0.40966958)(0.8807502,-0.31126592)
\psbezier[linecolor=black, linewidth=0.04](1.7603544,4.0220437)(1.7311239,0.060011756)(4.9849715,-0.27203125)(5.4139123,-0.2549059100196973)(5.8428535,-0.23778059)(8.692181,-0.8966086)(8.78168,-4.2331715)
\rput[bl](4.6,-4.793479){$(2n-1)\pi$}
\rput[bl](8.56,-0.053479005){$\tilde{\gamma}_n$}
\pscircle[linecolor=black, linewidth=0.04, fillstyle=solid,fillcolor=black, dimen=outer](8.6430435,-3.1404355){0.08304354}
\rput{-1.3019527}(0.020757874,0.15997073){\pscircle[linecolor=black, linewidth=0.04, fillstyle=solid,fillcolor=black, dimen=outer](7.05,-0.833479){0.09}}
\pscircle[linecolor=black, linewidth=0.04, fillstyle=solid,fillcolor=black, dimen=outer](9.18,-2.153479){0.08}
\psline[linecolor=black, linewidth=0.04, linestyle=dashed, dash=0.17638889cm 0.10583334cm](7.04,-0.893479)(7.04,-0.293479)
\psline[linecolor=black, linewidth=0.04, linestyle=dashed, dash=0.17638889cm 0.10583334cm](9.18,-2.233479)(9.18,0.726521)
\psline[linecolor=black, linewidth=0.04, linestyle=dashed, dash=0.17638889cm 0.10583334cm](8.67599,-0.3732503)(8.672706,-3.1624029)
\rput[bl](1.9,3.246521){$y=\beta\cot\frac{\beta}{2}$}
\rput[bl](10.14,-0.833479){$\beta$}
\rput[bl](0.46,4.506521){$y$}
\rput[bl](9.36,-4.653479){$2n\pi$}
\rput[bl](8.28,0.806521){$q\sqrt{\frac{\xi}{\xi-2q}}$}
\rput[bl](6.82,-0.053479005){$\gamma_n^*$}
\rput[bl](1.32,-4.273479){$\frac{(y+q)^2}{q^2}+\frac{\beta^2}{\frac{q^2}{1-\frac{2q}{\xi}}}=1$}
\psline[linecolor=black, linewidth=0.04, linestyle=dashed, dash=0.17638889cm 0.10583334cm](9.08716,-2.133779)(0.8928401,-2.173179)
\rput[bl](0.2,-2.273479){$-q$}
\rput[bl](0.0,-4.373479){$-2q$}
\psline[linecolor=black, linewidth=0.04, linestyle=dashed, dash=0.17638889cm 0.10583334cm](8.230249,-0.2932845)(8.258446,-2.1223688)
\pscircle[linecolor=black, linewidth=0.04, fillstyle=solid,fillcolor=black, dimen=outer](8.24,-2.113479){0.08}
\rput[bl](8.02,-0.113479){$\bar{\beta}_n$}
\end{pspicture}
}
\caption{The graphs of $y=F(\beta)=\beta\cot\frac{\beta}{2}$, and those of $y=G(\beta)$, $y=H(\beta)$ which are the upper and lower part of the curve $\frac{(y+q)^2}{q^2}+\frac{\beta^2}{\left(\frac{q^2}{1-\frac{2q}{\xi}}\right)}=1$, $\beta>0$, respectively.}
\label{lemma352-fig}
\end{center}
\end{figure}
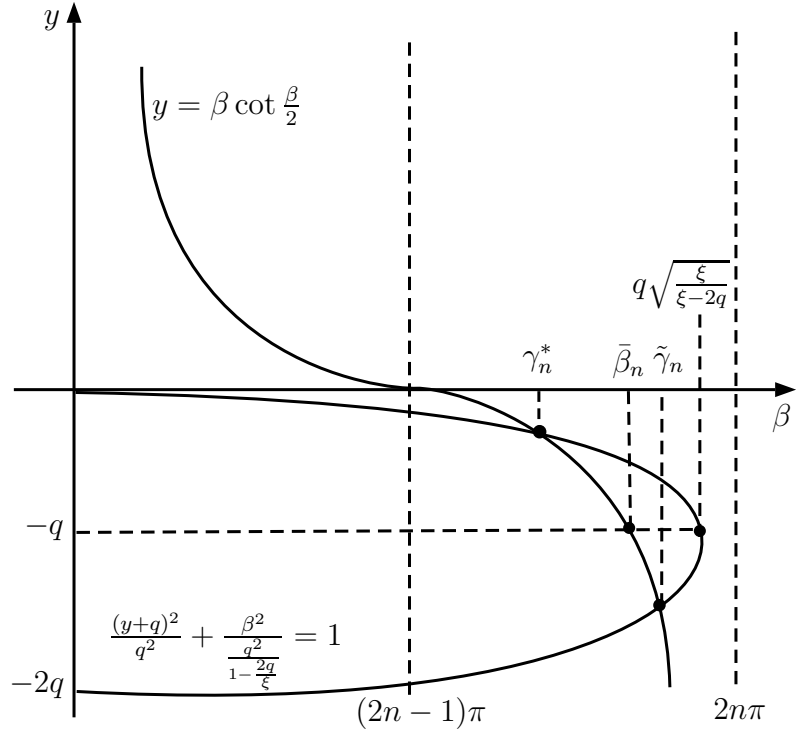

Notice that in the domains of $F,\,G$ and $H$, we have
\begin{align}
 \frac{\d}{\d\beta}F(\beta)= & \frac{\sin\beta-\beta}{1-\cos\beta}<0,\label{inv-fgh-1}\\
  \frac{\d}{\d\beta}G(\beta)= & -\frac{\left(\frac{\xi-2q}{\xi}\right)\beta}{\sqrt{q^2-\left(\frac{\xi-2q}{\xi}\right)\beta^2}}<0, \label{inv-fgh-2}\\
   \frac{\d}{\d\beta}H(\beta)= & \frac{\left(\frac{\xi-2q}{\xi}\right)\beta}{\sqrt{q^2-\left(\frac{\xi-2q}{\xi}\right)\beta^2}}>0. \label{inv-fgh-3}
\end{align}Then their inverses exist.  Notice that we have
\begin{align*}
 F^{-1}(0)= & (2n-1)\pi> 0=g^{-1}(0),\\
  F^{-1}(-q)= & \bar{\beta}_n>q\sqrt{\frac{\xi}{\xi-2q}}=G^{-1}(-q).
\end{align*}
By continuity of $F^{-1}$ and $G^{-1}$, if  $F^{-1}-G^{-1}$ has a zero, then there must be at least two and the derivative $\frac{\d}{\d\beta}(F^{-1}-G^{-1}) $ evaluated at the zeros cannot be all negative or all positive. To obtain a contradiction, we compute the derivative $\frac{\d}{\d\beta}(F^{-1}-G^{-1}) $ evaluated at the zeros and have
\begin{align}
 \frac{\d}{\d\beta}(F^{-1}-G^{-1})(y)& = \frac{1}{F'(F^{-1}(y))}-\frac{1}{G'(G^{-1}(y))}\notag\\
 & =\frac{1-\cos\beta}{\sin\beta-\beta}+\frac{\sqrt{q^2-\left(\frac{\xi-2q}{\xi}\right)\beta^2}}{\left(\frac{\xi-2q}{\xi}\right)\beta}.\label{h2-roots-3}
\end{align}
Note that for every zero $y_0$ of $F^{-1}-G^{-1}$, there exists $\beta_0$ such that $F^{-1}(y_0)=G^{-1}(y_0)=\beta_0.$ That is, $F(\beta)=G(\beta)$ which leads to
\begin{align}\label{h2-roots-4}
 \beta_0\cot\frac{\beta_0}{2}=-q+\sqrt{q^2-\left(\frac{\xi-2q}{\xi}\right)\beta_0^2}.
\end{align} By (\ref{h2-roots-3}) and (\ref{h2-roots-4}), we have
\begin{align}
 \frac{\d}{\d\beta}(F^{-1}-G^{-1})(y_0)& = \frac{1}{F'(F^{-1}(y_0))}-\frac{1}{G'(G^{-1}(y_0))}\notag\\
 & =\frac{1-\cos\beta_0}{\sin\beta_0-\beta_0}+\frac{\frac{\beta_0\sin\beta_0}{1-\cos\beta_0}+q}{\left(\frac{\xi-2q}{\xi}\right)\beta_0}\notag\\
 & =\frac{\beta_0(1-\cos\beta_0)\left(\frac{\xi-2q}{\xi}\right)+\sqrt{q^2-\left(\frac{\xi-2q}{\xi}\right)\beta_0^2}}{\left(\frac{\xi-2q}{\xi}\right)(\sin\beta_0-\beta_0)\beta_0}<0.\label{h2-roots-5}
\end{align} Namely, the derivatives of $F^{-1}-G^{-1}$ at the zeros are all negative. This is impossible and $F^{-1}-G^{-1}$ has no zero.

Next we turn to $F^{-1}-H^{-1}$. By (\ref{inv-fgh-1}) and (\ref{inv-fgh-3}) we know that $F^{-1}-H^{-1}$ is decreasing.
Note that 
\begin{align*}
  F^{-1}(-q)= & \bar{\beta}_n>q\sqrt{\frac{\xi}{\xi-2q}}=H^{-1}(-q),\\
 F^{-1}(-2q)= & \tilde{\beta}_n>0=H^{-1}(-2q),
\end{align*}where $\tilde{\beta}_n$ is the unique zero of $2q(1-\cos\beta)+\beta\sin\beta=0$ in $(2(n-1)\pi,\,2n\pi)$. By the continuity of $F^{-1}-H^{-1}$, $F^{-1}-H^{-1}$ has no zero in 
$(2(n-1)\pi,\,2n\pi)$ and $h_2>0$ if and only if $\beta\in (\beta_n^*,\,2n\pi)$. This completes the proof of 
iii) a).

iii) b) Let $F,\,G$ and $H$ be the same as are defined at (\ref{fgh-1}), (\ref{fgh-2}) and (\ref{fgh-3}). Since $\bar{\beta}_n\leq q\sqrt{\frac{\xi}{\xi-2q}}$, we have
\begin{align}
F^{-1}(0)= &\, 2(n-1)\pi>0=G^{-1}(0),\label{inv-end-points-1}\\
  F^{-1}(-q)= &\, \bar{\beta}_n\leq q\sqrt{\frac{\xi}{\xi-2q}}=G^{-1}(-q).\label{inv-end-points-2}
\end{align}By continuity of $F^{-1}-G^{-1}$ and by the intermediate value theorem, there exists at least one zero $\gamma_n^*\in (2(n-1)\pi,\,2n\pi)$. Next we show that $\gamma_n^*<\tilde{\beta}_n$.

We have
\begin{align*}
\tilde{\beta}_n\cot\frac{\tilde{\beta}_n}{2}=& -2q,\\
 \gamma_n^*\cot\frac{\gamma_n^*}{2}=&-q+\sqrt{q^2-\left(\frac{\xi-2q}{\xi}\right)(\gamma_n^*)^2}.
\end{align*}Then, we have
\[
 \tilde{\beta}_n\cot\frac{\tilde{\beta}_n}{2}=-2q<-q+\sqrt{q^2-\left(\frac{\xi-2q}{\xi}\right)(\gamma_n^*)^2}=\gamma_n^*\cot\frac{\gamma_n^*}{2}.
\]It follows that  that $\gamma_n^*<\tilde{\beta}_n<\beta_n^*$ since the map $f:\beta\rightarrow\beta\cot\frac{\beta}{2}$ is decreasing in $ (2(n-1)\pi,\,2n\pi)$.

Next we show the uniqueness of the zero of $F^{-1}-G^{-1}$. Suppose $F^{-1}-G^{-1}$ has more than two zeros. Then there must be at least three, counting multiplicity since $F^{-1}-G^{-1}$ is continuous in  $(-q,\,0)$ and (\ref{inv-end-points-1}) and (\ref{inv-end-points-2}) are satisfied. But by (\ref{h2-roots-5})
 the derivatives of $F^{-1}-G^{-1}$ at the zeros are all negative, which is impossible. Therefore, the zero of $F^{-1}-G^{-1}$ is unique.
 
 Next we turn to $F^{-1}-H^{-1}$. 
 Since $\bar{\beta}_n\leq q\sqrt{\frac{\xi}{\xi-2q}}$, we have
\begin{align} 
F^{-1}(-2q)> &\, 0=H^{-1}(-2q),\label{inv-end-points-3}\\
  F^{-1}(-q)= &\, \bar{\beta}_n\leq q\sqrt{\frac{\xi}{\xi-2q}}=H^{-1}(-q).\label{inv-end-points-4}
\end{align}By continuity of $F^{-1}-H^{-1}$ and by the intermediate value theorem, there exists at least one zero $\tilde{\gamma}_n\in (2(n-1)\pi,\,2n\pi)$. Next we show that $\beta_n^*>\tilde{\gamma}_n>\gamma_n^*$.

We have
\begin{align*}
 \tilde{\gamma}_n\cot\frac{\tilde{\gamma}_n}{2}=&-q-\sqrt{q^2-\left(\frac{\xi-2q}{\xi}\right)(\tilde{\gamma}_n)^2},\\
 \gamma_n^*\cot\frac{\gamma_n^*}{2}=&-q+\sqrt{q^2-\left(\frac{\xi-2q}{\xi}\right)(\gamma_n^*)^2}.
\end{align*}Then we have
\[
\tilde{\gamma}_n\cot\frac{\tilde{\gamma}_n}{2}<\gamma_n^*\cot\frac{\gamma_n^*}{2}.
\]It follows that  that $\gamma_n^*<\tilde{\gamma}_n$ since the map $F:\beta\rightarrow\beta\cot\frac{\beta}{2}$ is decreasing in $ (2(n-1)\pi,\,2n\pi)$.

Moreover,   we have
\[
 \tilde{\beta}_n\cot\frac{\tilde{\beta}_n}{2}=-2q<-q-\sqrt{q^2-\left(\frac{\xi-2q}{\xi}\right)(\tilde{\gamma}_n)^2}=\tilde{\gamma}_n\cot\frac{\tilde{\gamma}_n}{2}.
\]It follows that  that $\tilde{\gamma}_n<\tilde{\beta}_n$ since the map $f:\beta\rightarrow\beta\cot\frac{\beta}{2}$ is decreasing in $ (2(n-1)\pi,\,2n\pi)$. Therefore, we have $\beta_n^*>\tilde{\beta}_n>\tilde{\gamma}_n>\gamma_n^*$.

The uniqueness of the zero of $F^{-1}-H^{-1}$  follows from (\ref{inv-fgh-1}) and (\ref{inv-fgh-3}) that \[\frac{\d}{\d y}\left(F^{-1}(y)-H^{-1}(y)\right)<0,\] and $F^{-1}-H^{-1}$ is decreasing.

We have shown that $\gamma_n^*$ and $\tilde{\gamma}_n$ are the only zeros of $h_2$ and are in $(2(n-1)\pi,\,\tilde{\beta}_n)$. Then by its continuity and the limits obtained at Lemma~\ref{h_2-parameterization}, we 
have $h_2>0$ if and only if $\beta\in (\gamma_n^*,\,\tilde{\gamma}_n)\cup (\beta_n^*,\,2n\pi)$. This completes the proof of 
iii) b).\qed

\end{proof}
 By Theorems~\ref{delta-parameterization} and \ref{delta-h2-parameterization}, we can immediately obtain the following parameterization of all positive $(\delta,\,h_2)$ under the assumption $\xi>2p$, which is similar to Theorem~\ref{delta-h2-parameterization}.
 \begin{theorem}\label{para-delta-h2}
 Let $\xi$ and $q$ be positive numbers with $\xi>2q$, and $n_0\in\mathbb{N}$ be such that
 \[
  2(n_0-1)\pi<q\sqrt{\frac{\xi}{\xi-2q}}\leq 2n_0\pi.
 \] Let 
$\beta_n^*$, $n\in\mathbb{N}$  be the unique zero of the equation $\xi(1-\cos\beta)+\beta\sin\beta=0$  in $(2(n-1)\pi,\,2n\pi)$,  and $\bar{\beta}_n$ be the unique solution of $\beta\cot\beta=-q$ in $(2(n-1)\pi,\,2n\pi)$. Let the map $(\delta,\,h_2): (2(n-1)\pi,\,2n\pi)\rightarrow\mathbb{R}^2$ be defined at (\ref{delta-def}) and (\ref{delta-function}).  The following are true:
\begin{enumerate}
 \item[$i)$] For every $n>n_0$, $n\in\mathbb{N}$, 
 $
  (\delta,\,h_2)(\beta)>0
 $ if and only if $\beta\in  (\beta_n^*,\,2n\pi)$.
 \item[$ii)$] For every $1\leq n\leq n_0$, $n\in\mathbb{N}$, with  $\bar{\beta}_n>q\sqrt{\frac{\xi}{\xi-2q}}$, then
 $
  (\delta,\,h_2)(\beta)>0
 $ if and only if $\beta\in  (\beta_n^*,\,2n\pi)$.
  \item[$iii)$] For every $1\leq n\leq n_0$, $n\in\mathbb{N}$, with  $\bar{\beta}_n\leq q\sqrt{\frac{\xi}{\xi-2q}}$, $
  (\delta,\,h_2)(\beta)>0
 $ if and only if $\beta\in  (\gamma_n^*,\,\tilde{\gamma}_n)\cup(\beta_n^*,\,2n\pi)$, where $\gamma_n^*,\,\tilde{\gamma}_n$ are zeros of $h_2$ in $(2(n-1)\pi,\,2n\pi)$.
\end{enumerate}
  
 \end{theorem}

\section{Numerical simulations}\label{Numerical-simulations}
In this section, we numerically demonstrate the results of Theorem~\ref{para-delta-1},   Theorem~\ref{delta-h2-parameterization} and Theorem~\ref{para-delta-h2}.

Figure~\ref{fig-01} shows that case of  item (i) of Theorem~\ref{para-delta-1}  with $q>\beta_1\frac{\sqrt{2\sqrt{5}-2}}{4\sqrt{5}-8}$ and Figure~\ref{fig-02} the details of the stability region near the origin $(0,\,0)$.  Note that if $(\delta,\,h_1)=(0,\,0)$, the characteristic equation $\mathcal{P}(\lambda)+\delta(h_1+\frac{q}{\lambda e^{\lambda}})(1-e^{-\lambda})=0$ has only two roots with real parts negative. That is, the equilibrium of system~(\ref{SDDEs-system-eta}) is stable. It is known that the equilibrium changes stability as the parameter $(\delta,\,h_1)$ varies  only if it passes the boundary where  purely imaginary roots of  the characteristic equation occur. Figure~\ref{fig-02} shows that with delayed instead of instantaneous spindle speed control, it is still possible to stabilize the equilibrium, as a neighborhood without intersecting either branches of the parameterized curves of $(\delta,\,h_1)$ near the origin exists.

Figure~\ref{fig-4}  shows the case of item (iv) of Theorem~\ref{para-delta-1} for  $\mathcal{P}(\lambda)+\delta(h_2+\frac{q}{\lambda})(1-e^{-\lambda})=0$ with $q<\beta_1\frac{\sqrt{2\sqrt{5}-2}}{4\sqrt{5}-8}$.  Figure~\ref{fig-5} shows the stability region with $\xi<2q$ described at Theorem~\ref{delta-h2-parameterization}. Since the branch parameterized with $\beta\in (\gamma_n^*,\,\beta_n^*)$ has a vertical asymptote and has a zero on the horizontal line $h_2=0$, the connected   stability region for positive $(\delta,\,h_2)$ is enclosed by the first branch with $n=1$, $\delta=0$ and  $h_2=0$ and the other branches with $n\geq 2$ will not contribute an stability region.

Figure~\ref{fig-6} shows the stability region with $\xi>2q$ described at Theorem~\ref{para-delta-h2}. With $\xi=1.62$ and $q=0.8$, we have
\[
2\pi\leq q\sqrt{\frac{\xi}{\xi-2q}}\leq 4\pi,
\]which leads to $n_0=2$. According to  Theorem~\ref{para-delta-h2}, for $n\geq 3$,   $(\delta,\,h_2)$ is positive if and only if $\beta\in (\beta_n^*,\,2n\pi)$. For $n=n_0=2$, we have $\bar{\beta}_n=9.591212>q\sqrt{\frac{\xi}{\xi-2q}}=7.2>2\pi$. Then  $(\delta,\,h_2)$ is positive if and only if $\beta\in (\beta_n^*,\,2n\pi)$ with $\beta_n^*=9.75394647$.  For $n=1<n_0$, we have $\bar{\beta}_n=3.581158<q\sqrt{\frac{\xi}{\xi-2q}}=7.2$. Then  $(\delta,\,h_2)$ is positive if and only if $\beta\in  (\gamma_n^*,\,\tilde{\gamma}_n)\cup (\beta_n^*,\,2n\pi)$, where $\beta_n^*=3.92455245$, $\gamma_n^*=3.19356076$, and $\tilde{\gamma}_n=3.86974862$. Note that the extra interval $  (\gamma_n^*,\,\tilde{\gamma}_n)$ in addition to $ (\beta_n^*,\,2n\pi)$ corresponds to the small lobe near $(0,\,0)$ in  Figure~\ref{fig-6}.

It turns out that in this case with $\xi>2q$, the stability region is the connected region between  each of the stability lobes corresponding to $n\in\mathbb{N}$. This means that when the damping coefficient is large the stability region is unbounded in both of the directions of $\delta$ and $h_2$, in contrast to the scenarios shown in Figures~\ref{fig-01}--\ref{fig-5} which are unbounded in  the directions of $h_1$ or $h_2$ only.

\providelength{\AxesLineWidth}       \setlength{\AxesLineWidth}{0.5pt}%
\providelength{\plotwidth}           \setlength{\plotwidth}{10cm}
\providelength{\LineWidth}           \setlength{\LineWidth}{0.7pt}%
\providelength{\MarkerSize}          \setlength{\MarkerSize}{4pt}%
\newrgbcolor{GridColor}{0.8 0.8 0.8}%

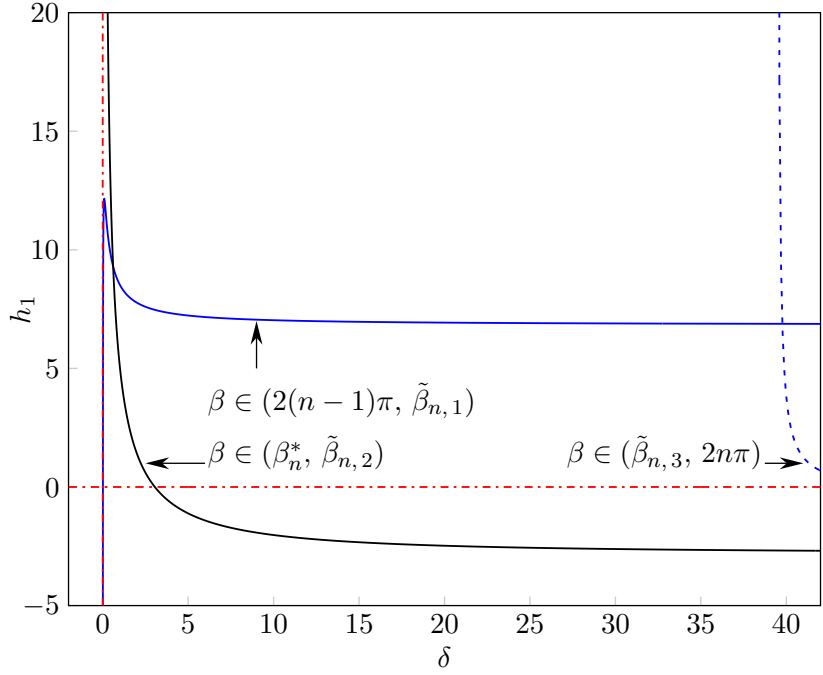
\begin{figure}[H]
\begin{center}
\psset{xunit=0.022727\plotwidth,yunit=0.031548\plotwidth}%
\begin{pspicture}(-6.460829,-7.777778)(42.202765,20.584795)%


\psline[linewidth=\AxesLineWidth,linecolor=GridColor](0.000000,-5.000000)(0.000000,-4.619632)
\psline[linewidth=\AxesLineWidth,linecolor=GridColor](5.000000,-5.000000)(5.000000,-4.619632)
\psline[linewidth=\AxesLineWidth,linecolor=GridColor](10.000000,-5.000000)(10.000000,-4.619632)
\psline[linewidth=\AxesLineWidth,linecolor=GridColor](15.000000,-5.000000)(15.000000,-4.619632)
\psline[linewidth=\AxesLineWidth,linecolor=GridColor](20.000000,-5.000000)(20.000000,-4.619632)
\psline[linewidth=\AxesLineWidth,linecolor=GridColor](25.000000,-5.000000)(25.000000,-4.619632)
\psline[linewidth=\AxesLineWidth,linecolor=GridColor](30.000000,-5.000000)(30.000000,-4.619632)
\psline[linewidth=\AxesLineWidth,linecolor=GridColor](35.000000,-5.000000)(35.000000,-4.619632)
\psline[linewidth=\AxesLineWidth,linecolor=GridColor](40.000000,-5.000000)(40.000000,-4.619632)
\psline[linewidth=\AxesLineWidth,linecolor=GridColor](-2.000000,-5.000000)(-1.472000,-5.000000)
\psline[linewidth=\AxesLineWidth,linecolor=GridColor](-2.000000,0.000000)(-1.472000,0.000000)
\psline[linewidth=\AxesLineWidth,linecolor=GridColor](-2.000000,5.000000)(-1.472000,5.000000)
\psline[linewidth=\AxesLineWidth,linecolor=GridColor](-2.000000,10.000000)(-1.472000,10.000000)
\psline[linewidth=\AxesLineWidth,linecolor=GridColor](-2.000000,15.000000)(-1.472000,15.000000)
\psline[linewidth=\AxesLineWidth,linecolor=GridColor](-2.000000,20.000000)(-1.472000,20.000000)

{ \footnotesize 
\rput[t](0.000000,-5.380368){$0$}
\rput[t](5.000000,-5.380368){$5$}
\rput[t](10.000000,-5.380368){$10$}
\rput[t](15.000000,-5.380368){$15$}
\rput[t](20.000000,-5.380368){$20$}
\rput[t](25.000000,-5.380368){$25$}
\rput[t](30.000000,-5.380368){$30$}
\rput[t](35.000000,-5.380368){$35$}
\rput[t](40.000000,-5.380368){$40$}
\rput[r](-2.528000,-5.000000){$-5$}
\rput[r](-2.528000,0.000000){$0$}
\rput[r](-2.528000,5.000000){$5$}
\rput[r](-2.528000,10.000000){$10$}
\rput[r](-2.528000,15.000000){$15$}
\rput[r](-2.528000,20.000000){$20$}
} 

\psframe[linewidth=\AxesLineWidth,dimen=middle](-2.000000,-5.000000)(42.000000,20.000000)

{ \small 
\rput[b](20.000000,-8.277778){
\begin{tabular}{c}
$\delta$\\
\end{tabular}
}

\rput[t]{90}(-6.460829,7.500000){
\begin{tabular}{c}
$h_1$\\
\end{tabular}
}
} 

\newrgbcolor{color162.0448}{0  0  1}
\psline[plotstyle=line,linejoin=1,linestyle=solid,linewidth=\LineWidth,linecolor=color162.0448]
(0.009026,-5.000000)(0.009085,-4.858943)
\psline[plotstyle=line,linejoin=1,linestyle=solid,linewidth=\LineWidth,linecolor=color162.0448]
(32.747417,6.889832)(42.000000,6.876533)
\psline[plotstyle=line,linejoin=1,linestyle=solid,linewidth=\LineWidth,linecolor=color162.0448]
(0.009085,-4.858943)(0.011051,-1.035312)(0.013130,1.735165)(0.015423,3.895502)(0.017808,5.526941)
(0.020406,6.844819)(0.023076,7.866122)(0.025961,8.710353)(0.028893,9.374698)(0.032041,9.931788)
(0.035213,10.373383)(0.038600,10.746437)(0.041983,11.042249)(0.045334,11.277041)(0.048881,11.476454)
(0.052363,11.633225)(0.055747,11.755973)(0.059000,11.851544)(0.062405,11.932161)(0.065640,11.993397)
(0.068671,12.039329)(0.071819,12.077077)(0.074721,12.104127)(0.077344,12.123016)(0.079657,12.135803)
(0.082030,12.145564)(0.084057,12.151461)(0.085710,12.154756)(0.087393,12.156817)(0.088675,12.157569)
(0.089975,12.157653)(0.091292,12.157081)(0.093076,12.155313)(0.094892,12.152421)(0.097210,12.147259)
(0.100062,12.138856)(0.103490,12.126098)(0.107548,12.107757)(0.112299,12.082512)(0.118390,12.045317)
(0.125417,11.997172)(0.134810,11.926495)(0.147070,11.827251)(0.171136,11.622480)(0.198773,11.387406)
(0.217894,11.230493)(0.235571,11.091253)(0.252365,10.964633)(0.269248,10.843047)(0.286039,10.727759)
(0.302528,10.619865)(0.318480,10.520314)(0.333637,10.429924)(0.349808,10.337784)(0.364871,10.255753)
(0.380866,10.172431)(0.397882,10.087839)(0.413355,10.014337)(0.429719,9.939930)(0.447054,9.864630)
(0.462305,9.801206)(0.478346,9.737177)(0.495240,9.672550)(0.513057,9.607330)(0.528026,9.554732)
(0.543678,9.501761)(0.560059,9.448422)(0.577223,9.394716)(0.595225,9.340647)(0.614128,9.286217)
(0.628937,9.245160)(0.644324,9.203903)(0.660322,9.162447)(0.676968,9.120793)(0.694302,9.078942)
(0.712368,9.036896)(0.731212,8.994655)(0.750886,8.952221)(0.771445,8.909593)(0.792950,8.866775)
(0.815469,8.823766)(0.839072,8.780567)(0.855451,8.751663)(0.872373,8.722675)(0.889866,8.693605)
(0.907959,8.664451)(0.926685,8.635214)(0.946076,8.605895)(0.966168,8.576494)(0.987000,8.547011)
(1.008614,8.517446)(1.031054,8.487800)(1.054369,8.458073)(1.078610,8.428265)(1.103834,8.398376)
(1.130102,8.368406)(1.157479,8.338357)(1.171606,8.323302)(1.186038,8.308228)(1.200784,8.293133)
(1.215856,8.278018)(1.231264,8.262884)(1.247019,8.247730)(1.263133,8.232556)(1.279618,8.217362)
(1.296489,8.202149)(1.313757,8.186915)(1.331438,8.171662)(1.349546,8.156390)(1.368097,8.141098)
(1.387108,8.125786)(1.406596,8.110454)(1.426578,8.095103)(1.447074,8.079733)(1.468105,8.064343)
(1.489690,8.048934)(1.511853,8.033505)(1.534617,8.018056)(1.558006,8.002589)(1.582047,7.987102)
(1.606768,7.971595)(1.632197,7.956069)(1.658365,7.940524)(1.685305,7.924960)(1.713053,7.909377)
(1.741643,7.893774)(1.771117,7.878152)(1.801514,7.862511)(1.832879,7.846851)(1.865260,7.831171)
(1.898706,7.815473)(1.933270,7.799755)(1.969010,7.784019)(2.005987,7.768263)(2.044266,7.752489)
(2.083917,7.736695)(2.125014,7.720883)(2.167639,7.705051)(2.211879,7.689201)(2.257827,7.673332)
(2.305583,7.657444)(2.355256,7.641537)(2.406965,7.625612)(2.460837,7.609667)(2.517011,7.593704)
(2.575637,7.577722)(2.636879,7.561722)(2.700917,7.545702)(2.767947,7.529664)(2.838182,7.513608)
(2.911859,7.497532)(2.989236,7.481438)(3.070601,7.465326)(3.156269,7.449195)(3.246590,7.433045)
(3.341954,7.416877)(3.442797,7.400691)(3.549603,7.384486)(3.662917,7.368262)(3.783355,7.352020)
(3.911608,7.335760)(4.048462,7.319481)(4.194814,7.303184)(4.351688,7.286868)(4.520259,7.270534)
(4.701888,7.254182)(4.898152,7.237811)(5.110894,7.221422)(5.342281,7.205015)(5.594875,7.188590)
(5.871733,7.172146)(6.176527,7.155684)(6.513709,7.139204)(6.888732,7.122706)(7.308344,7.106190)
(7.780999,7.089655)(8.317434,7.073103)(8.931494,7.056532)(9.641340,7.039943)(10.471280,7.023336)
(11.454607,7.006711)(12.638174,6.990068)(14.090078,6.973407)(15.913241,6.956728)(18.270936,6.940031)
(21.438622,6.923316)(25.920282,6.906583)(32.747417,6.889832)

\newrgbcolor{color163.0443}{0  0  0}
\psline[plotstyle=line,linejoin=1,linestyle=solid,linewidth=\LineWidth,linecolor=color163.0443]
(0.291376,20.000000)(0.292102,19.950816)
\psline[plotstyle=line,linejoin=1,linestyle=solid,linewidth=\LineWidth,linecolor=color163.0443]
(41.695061,-2.687391)(42.000000,-2.688650)
\psline[plotstyle=line,linejoin=1,linestyle=solid,linewidth=\LineWidth,linecolor=color163.0443]
(0.292102,19.950816)(0.302649,19.264294)(0.313385,18.610662)(0.324314,17.987523)(0.335440,17.392706)
(0.346767,16.824247)(0.358300,16.280359)(0.370043,15.759417)(0.382002,15.259937)(0.394181,14.780562)
(0.406584,14.320049)(0.419218,13.877257)(0.432087,13.451137)(0.445196,13.040722)(0.458551,12.645121)
(0.472158,12.263510)(0.486022,11.895126)(0.500149,11.539262)(0.514546,11.195261)(0.529218,10.862515)
(0.544172,10.540455)(0.559415,10.228552)(0.574953,9.926313)(0.590793,9.633277)(0.606942,9.349013)
(0.623408,9.073117)(0.640198,8.805211)(0.657321,8.544939)(0.671263,8.341995)(0.685426,8.143561)
(0.699816,7.949485)(0.714436,7.759620)(0.729291,7.573827)(0.744385,7.391972)(0.759723,7.213927)
(0.775310,7.039573)(0.791150,6.868792)(0.807250,6.701474)(0.823613,6.537512)(0.840245,6.376806)
(0.857151,6.219257)(0.874338,6.064773)(0.891809,5.913265)(0.909572,5.764646)(0.927633,5.618835)
(0.945996,5.475754)(0.964669,5.335326)(0.983657,5.197479)(1.002968,5.062143)(1.022607,4.929251)
(1.042582,4.798738)(1.062900,4.670543)(1.083567,4.544606)(1.104592,4.420870)(1.125981,4.299278)
(1.147742,4.179778)(1.169883,4.062319)(1.192412,3.946850)(1.215339,3.833324)(1.238670,3.721696)
(1.262414,3.611919)(1.280500,3.530777)(1.298827,3.450635)(1.317400,3.371476)(1.336222,3.293284)
(1.355298,3.216042)(1.374632,3.139734)(1.394228,3.064345)(1.414091,2.989860)(1.434224,2.916264)
(1.454633,2.843543)(1.475322,2.771683)(1.496295,2.700671)(1.517559,2.630492)(1.539116,2.561134)
(1.560973,2.492584)(1.583135,2.424830)(1.605606,2.357860)(1.628393,2.291662)(1.651499,2.226224)
(1.674932,2.161535)(1.698697,2.097584)(1.722799,2.034361)(1.747244,1.971854)(1.772039,1.910054)
(1.797189,1.848951)(1.822701,1.788534)(1.848581,1.728794)(1.874836,1.669721)(1.901473,1.611307)
(1.928497,1.553543)(1.955917,1.496419)(1.983740,1.439926)(2.011972,1.384058)(2.040621,1.328804)
(2.069695,1.274158)(2.099202,1.220111)(2.129150,1.166656)(2.159546,1.113784)(2.190399,1.061489)
(2.221718,1.009763)(2.253511,0.958599)(2.285787,0.907990)(2.318556,0.857929)(2.351826,0.808410)
(2.385608,0.759426)(2.419911,0.710970)(2.454744,0.663036)(2.490119,0.615618)(2.526046,0.568710)
(2.562536,0.522305)(2.599599,0.476399)(2.637246,0.430985)(2.675490,0.386057)(2.714343,0.341611)
(2.753815,0.297640)(2.793920,0.254139)(2.834670,0.211104)(2.876079,0.168529)(2.918159,0.126408)
(2.960925,0.084738)(3.004390,0.043512)(3.048569,0.002727)(3.093476,-0.037622)(3.139127,-0.077540)
(3.185537,-0.117032)(3.232722,-0.156102)(3.280698,-0.194754)(3.329482,-0.232992)(3.379092,-0.270822)
(3.429544,-0.308246)(3.480858,-0.345269)(3.533051,-0.381895)(3.586142,-0.418128)(3.640152,-0.453972)
(3.695100,-0.489430)(3.751008,-0.524507)(3.807895,-0.559206)(3.865785,-0.593530)(3.924699,-0.627484)
(3.984661,-0.661070)(4.045695,-0.694293)(4.107824,-0.727155)(4.171073,-0.759660)(4.235470,-0.791811)
(4.301040,-0.823612)(4.367810,-0.855065)(4.435808,-0.886174)(4.505064,-0.916942)(4.575608,-0.947371)
(4.647469,-0.977466)(4.720680,-1.007228)(4.795273,-1.036662)(4.871282,-1.065768)(4.948740,-1.094552)
(5.027685,-1.123014)(5.108151,-1.151158)(5.190178,-1.178987)(5.273805,-1.206503)(5.359070,-1.233710)
(5.446017,-1.260608)(5.534688,-1.287202)(5.625127,-1.313493)(5.717379,-1.339485)(5.811493,-1.365178)
(5.907516,-1.390577)(6.005499,-1.415683)(6.105494,-1.440498)(6.207555,-1.465026)(6.311737,-1.489267)
(6.418098,-1.513225)(6.526697,-1.536902)(6.637596,-1.560300)(6.750857,-1.583420)(6.866549,-1.606266)
(6.984737,-1.628839)(7.105494,-1.651141)(7.228891,-1.673175)(7.355006,-1.694942)(7.483917,-1.716444)
(7.615704,-1.737684)(7.750454,-1.758663)(7.888253,-1.779384)(8.029193,-1.799847)(8.173368,-1.820056)
(8.320876,-1.840011)(8.471819,-1.859715)(8.626304,-1.879170)(8.784440,-1.898377)(8.946342,-1.917338)
(9.112130,-1.936055)(9.281926,-1.954529)(9.455862,-1.972763)(9.634070,-1.990758)(9.816691,-2.008515)
(10.003871,-2.026037)(10.195763,-2.043324)(10.392525,-2.060379)(10.594323,-2.077203)(10.801329,-2.093798)
(11.013726,-2.110165)(11.231700,-2.126306)(11.455451,-2.142222)(11.685184,-2.157915)(11.921114,-2.173386)
(12.163470,-2.188637)(12.412486,-2.203670)(12.668411,-2.218485)(12.931506,-2.233084)(13.202043,-2.247469)
(13.480309,-2.261641)(13.766604,-2.275601)(14.061244,-2.289350)(14.364563,-2.302891)(14.676908,-2.316225)
(14.998649,-2.329352)(15.330173,-2.342274)(15.671888,-2.354993)(16.024225,-2.367509)(16.387638,-2.379824)
(16.762609,-2.391940)(17.149645,-2.403857)(17.549281,-2.415577)(17.962088,-2.427101)(18.388666,-2.438430)
(18.829655,-2.449565)(19.285730,-2.460508)(19.757612,-2.471260)(20.246066,-2.481822)(20.751904,-2.492196)
(21.275992,-2.502381)(21.819255,-2.512380)(22.382677,-2.522194)(22.967311,-2.531823)(23.574282,-2.541270)
(24.204795,-2.550534)(24.860141,-2.559618)(25.541706,-2.568522)(26.250977,-2.577247)(26.989556,-2.585794)
(27.759167,-2.594165)(28.561671,-2.602360)(29.399075,-2.610381)(30.273554,-2.618229)(31.187462,-2.625904)
(32.143355,-2.633408)(33.144009,-2.640741)(34.192451,-2.647906)(35.291978,-2.654902)(36.446198,-2.661730)
(37.659060,-2.668392)(38.934898,-2.674889)(40.278480,-2.681222)(41.695061,-2.687391)

\newrgbcolor{color164.0443}{0  0  1}
\psline[plotstyle=line,linejoin=1,linestyle=dashed,dash=2pt 3pt,linewidth=\LineWidth,linecolor=color164.0443]
(42.000000,0.688969)(41.926082,0.715313)
\psline[plotstyle=line,linejoin=1,linestyle=dashed,dash=2pt 3pt,linewidth=\LineWidth,linecolor=color164.0443]
(39.593822,17.164009)(39.590241,20.000000)
\psline[plotstyle=line,linejoin=1,linestyle=dashed,dash=2pt 3pt,linewidth=\LineWidth,linecolor=color164.0443]
(41.926082,0.715313)(41.817182,0.757370)(41.709005,0.802816)(41.601546,0.852131)(41.494798,0.905888)
(41.388754,0.964781)(41.283409,1.029654)(41.178757,1.101550)(41.074791,1.181767)(40.971505,1.271943)
(40.868895,1.374179)(40.766953,1.491209)(40.665675,1.626661)(40.565055,1.785458)(40.465086,1.974443)
(40.365764,2.203428)(40.267083,2.486991)(40.169038,2.847771)(40.071623,3.322942)(39.974833,3.978096)
(39.878663,4.940820)(39.783108,6.496609)(39.688163,9.442597)(39.593822,17.164009)

\newrgbcolor{color165.0443}{1  0  0}
\psline[plotstyle=line,linejoin=1,linestyle=dashed,dash=3pt 2pt 1pt 2pt,linewidth=\LineWidth,linecolor=color165.0443]
(-2.000000,0.000000)(5.000000,0.000000)
\psline[plotstyle=line,linejoin=1,linestyle=dashed,dash=3pt 2pt 1pt 2pt,linewidth=\LineWidth,linecolor=color165.0443]
(35.000000,0.000000)(42.000000,0.000000)
\psline[plotstyle=line,linejoin=1,linestyle=dashed,dash=3pt 2pt 1pt 2pt,linewidth=\LineWidth,linecolor=color165.0443]
(5.000000,0.000000)(15.000000,0.000000)(25.000000,0.000000)(35.000000,0.000000)

\newrgbcolor{color166.0443}{1  0  0}
\psline[plotstyle=line,linejoin=1,linestyle=dashed,dash=3pt 2pt 1pt 2pt,linewidth=\LineWidth,linecolor=color166.0443]
(0.000000,-5.000000)(0.000000,20.000000)

{ \small 
\newrgbcolor{color234.0145}{0  0  0}
\psline[linestyle=solid,linewidth=0.5pt,linecolor=color234.0145,arrowsize=1.5pt 3,arrowlength=2,arrowinset=0.3]{->}(9,5)(9,7)
\uput{0pt}[33.769825](6, 3){%
\psframebox*[framesep=1pt]{\begin{tabular}{@{}c@{}}
$\beta\in (2(n-1)\pi,\,\tilde{\beta}_{n,\, 1})$\\[-0.3ex]
\end{tabular}}}
}

{ \small 
\newrgbcolor{color234.0145}{0  0  0}
\psline[linestyle=solid,linewidth=0.5pt,linecolor=color234.0145,arrowsize=1.5pt 3,arrowlength=2,arrowinset=0.3]{->}(6,1)(2.5,1)
\uput{0pt}[33.769825](6,0.8){%
\psframebox*[framesep=1pt]{\begin{tabular}{@{}c@{}}
$\beta\in (\beta_n^*,\,\tilde{\beta}_{n,\,2})$\\[-0.3ex]
\end{tabular}}}
}

{ \small 
\newrgbcolor{color234.0145}{0  0  0}
\psline[linestyle=solid,linewidth=0.5pt,linecolor=color234.0145,arrowsize=1.5pt 3,arrowlength=2,arrowinset=0.3]{->}(35,1)(41,1)
\uput{0pt}[33.769825](27,0.8){%
\psframebox*[framesep=1pt]{\begin{tabular}{@{}c@{}}
$\beta\in (\tilde{\beta}_{n,\,3},\,2n\pi)$\\[-0.3ex]
\end{tabular}}}
}
\end{pspicture}%
\caption{The curves of  $(\delta,\,h_1)$, $\delta>0$ where  $\xi=0.2$, $q=12$, $n=1$, $\beta_1\frac{\sqrt{2\sqrt{5}-2}}{4\sqrt{5}-8}=8.955929<q$,   $\beta_n^*=3.26398905$, $\tilde{\beta}_{n,\,1}=1.634732310091$, $\tilde{\beta}_{n,\,2}=4.99223679$ and $\tilde{\beta}_{n,\,3}=5.73783731$.   The shaded region near the origin $(0,\,0)$ is the stability region.}
\label{fig-01}
\end{center}
\end{figure}

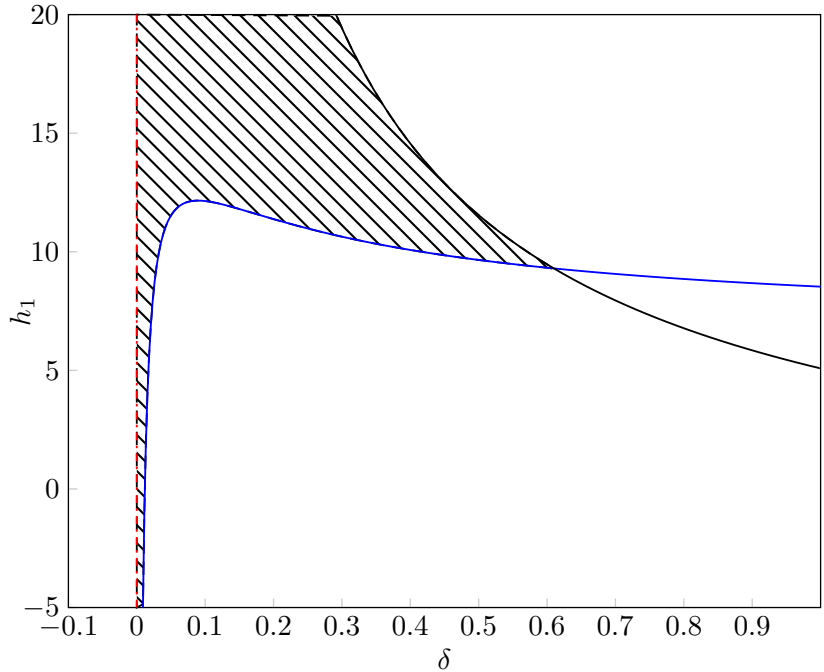
\begin{figure}[H]
\begin{center}
\psset{xunit=0.909091\plotwidth,yunit=0.031548\plotwidth}%
\begin{pspicture}(-0.211521,-7.777778)(1.005069,20.584795)%


\psline[linewidth=\AxesLineWidth,linecolor=GridColor](-0.100000,-5.000000)(-0.100000,-4.619632)
\psline[linewidth=\AxesLineWidth,linecolor=GridColor](0.000000,-5.000000)(0.000000,-4.619632)
\psline[linewidth=\AxesLineWidth,linecolor=GridColor](0.100000,-5.000000)(0.100000,-4.619632)
\psline[linewidth=\AxesLineWidth,linecolor=GridColor](0.200000,-5.000000)(0.200000,-4.619632)
\psline[linewidth=\AxesLineWidth,linecolor=GridColor](0.300000,-5.000000)(0.300000,-4.619632)
\psline[linewidth=\AxesLineWidth,linecolor=GridColor](0.400000,-5.000000)(0.400000,-4.619632)
\psline[linewidth=\AxesLineWidth,linecolor=GridColor](0.500000,-5.000000)(0.500000,-4.619632)
\psline[linewidth=\AxesLineWidth,linecolor=GridColor](0.600000,-5.000000)(0.600000,-4.619632)
\psline[linewidth=\AxesLineWidth,linecolor=GridColor](0.700000,-5.000000)(0.700000,-4.619632)
\psline[linewidth=\AxesLineWidth,linecolor=GridColor](0.800000,-5.000000)(0.800000,-4.619632)
\psline[linewidth=\AxesLineWidth,linecolor=GridColor](0.900000,-5.000000)(0.900000,-4.619632)
\psline[linewidth=\AxesLineWidth,linecolor=GridColor](-0.100000,-5.000000)(-0.086800,-5.000000)
\psline[linewidth=\AxesLineWidth,linecolor=GridColor](-0.100000,0.000000)(-0.086800,0.000000)
\psline[linewidth=\AxesLineWidth,linecolor=GridColor](-0.100000,5.000000)(-0.086800,5.000000)
\psline[linewidth=\AxesLineWidth,linecolor=GridColor](-0.100000,10.000000)(-0.086800,10.000000)
\psline[linewidth=\AxesLineWidth,linecolor=GridColor](-0.100000,15.000000)(-0.086800,15.000000)
\psline[linewidth=\AxesLineWidth,linecolor=GridColor](-0.100000,20.000000)(-0.086800,20.000000)

{ \footnotesize 
\rput[t](-0.100000,-5.380368){$-0.1$}
\rput[t](0.000000,-5.380368){$0$}
\rput[t](0.100000,-5.380368){$0.1$}
\rput[t](0.200000,-5.380368){$0.2$}
\rput[t](0.300000,-5.380368){$0.3$}
\rput[t](0.400000,-5.380368){$0.4$}
\rput[t](0.500000,-5.380368){$0.5$}
\rput[t](0.600000,-5.380368){$0.6$}
\rput[t](0.700000,-5.380368){$0.7$}
\rput[t](0.800000,-5.380368){$0.8$}
\rput[t](0.900000,-5.380368){$0.9$}
\rput[r](-0.113200,-5.000000){$-5$}
\rput[r](-0.113200,0.000000){$0$}
\rput[r](-0.113200,5.000000){$5$}
\rput[r](-0.113200,10.000000){$10$}
\rput[r](-0.113200,15.000000){$15$}
\rput[r](-0.113200,20.000000){$20$}
} 

\psframe[linewidth=\AxesLineWidth,dimen=middle](-0.100000,-5.000000)(1.000000,20.000000)

{ \small 
\rput[b](0.450000,-8.277778){
\begin{tabular}{c}
$\delta$\\
\end{tabular}
}

\rput[t]{90}(-0.211521,7.500000){
\begin{tabular}{c}
$h_1$\\
\end{tabular}
}
} 

\pspolygon[fillstyle=vlines, fillcolor=red,linecolor=black,linestyle=dashed](0,-5)(0.009085,-4.858943)(0.011051,-1.035312)(0.013130,1.735165)(0.015423,3.895502)(0.017808,5.526941)
(0.020406,6.844819)(0.023076,7.866122)(0.025961,8.710353)(0.028893,9.374698)(0.032041,9.931788)
(0.035213,10.373383)(0.038600,10.746437)(0.041983,11.042249)(0.045334,11.277041)(0.048881,11.476454)
(0.052363,11.633225)(0.055747,11.755973)(0.059000,11.851544)(0.062405,11.932161)(0.065640,11.993397)
(0.068671,12.039329)(0.071819,12.077077)(0.074721,12.104127)(0.077344,12.123016)(0.079657,12.135803)
(0.082030,12.145564)(0.084057,12.151461)(0.085710,12.154756)(0.087393,12.156817)(0.088675,12.157569)
(0.089975,12.157653)(0.091292,12.157081)(0.093076,12.155313)(0.094892,12.152421)(0.097210,12.147259)
(0.100062,12.138856)(0.103490,12.126098)(0.107548,12.107757)(0.112299,12.082512)(0.118390,12.045317)
(0.125417,11.997172)(0.134810,11.926495)(0.147070,11.827251)(0.171136,11.622480)(0.198773,11.387406)
(0.217894,11.230493)(0.235571,11.091253)(0.252365,10.964633)(0.269248,10.843047)(0.286039,10.727759)
(0.302528,10.619865)(0.318480,10.520314)(0.333637,10.429924)(0.349808,10.337784)(0.364871,10.255753)
(0.380866,10.172431)(0.397882,10.087839)(0.413355,10.014337)(0.429719,9.939930)(0.447054,9.864630)
(0.462305,9.801206)(0.478346,9.737177)(0.495240,9.672550)(0.513057,9.607330)(0.528026,9.554732)
(0.543678,9.501761)(0.560059,9.448422)(0.577223,9.394716)(0.595225,9.340647)(0.61,9.286217)(0.606942,9.349013)(0.590793,9.633277)(0.574953,9.926313)(0.559415,10.228552)(0.544172,10.540455)(0.529218,10.862515)(0.514546,11.195261)(0.500149,11.539262)(0.486022,11.895126)(0.472158,12.263510)(0.458551,12.645121)(0.445196,13.040722)(0.432087,13.451137)
(0.419218,13.877257)(0.406584,14.320049)(0.394181,14.780562)(0.382002,15.259937)(0.370043,15.759417)(0.358300,16.280359)(0.346767,16.824247)(0.335440,17.392706)(0.324314,17.987523)
(0.313385,18.610662)(0.302649,19.264294)(0.292102,19.950816) (0, 20)

\newrgbcolor{color162.0455}{0  0  1}
\psline[plotstyle=line,linejoin=1,linestyle=solid,linewidth=\LineWidth,linecolor=color162.0455]
(0.009026,-5.000000)(0.009085,-4.858943)
\psline[plotstyle=line,linejoin=1,linestyle=solid,linewidth=\LineWidth,linecolor=color162.0455]
(0.997707,8.532239)(1.000000,8.529129)
\psline[plotstyle=line,linejoin=1,linestyle=solid,linewidth=\LineWidth,linecolor=color162.0455]
(0.009085,-4.858943)(0.011051,-1.035312)(0.013130,1.735165)(0.015423,3.895502)(0.017808,5.526941)
(0.020406,6.844819)(0.023076,7.866122)(0.025961,8.710353)(0.028893,9.374698)(0.032041,9.931788)
(0.035213,10.373383)(0.038600,10.746437)(0.041983,11.042249)(0.045334,11.277041)(0.048881,11.476454)
(0.052363,11.633225)(0.055747,11.755973)(0.059000,11.851544)(0.062405,11.932161)(0.065640,11.993397)
(0.068671,12.039329)(0.071819,12.077077)(0.074721,12.104127)(0.077344,12.123016)(0.079657,12.135803)
(0.082030,12.145564)(0.084057,12.151461)(0.085710,12.154756)(0.087393,12.156817)(0.088675,12.157569)
(0.089975,12.157653)(0.091292,12.157081)(0.093076,12.155313)(0.094892,12.152421)(0.097210,12.147259)
(0.100062,12.138856)(0.103490,12.126098)(0.107548,12.107757)(0.112299,12.082512)(0.118390,12.045317)
(0.125417,11.997172)(0.134810,11.926495)(0.147070,11.827251)(0.171136,11.622480)(0.198773,11.387406)
(0.217894,11.230493)(0.235571,11.091253)(0.252365,10.964633)(0.269248,10.843047)(0.286039,10.727759)
(0.302528,10.619865)(0.318480,10.520314)(0.333637,10.429924)(0.349808,10.337784)(0.364871,10.255753)
(0.380866,10.172431)(0.397882,10.087839)(0.413355,10.014337)(0.429719,9.939930)(0.447054,9.864630)
(0.462305,9.801206)(0.478346,9.737177)(0.495240,9.672550)(0.513057,9.607330)(0.528026,9.554732)
(0.543678,9.501761)(0.560059,9.448422)(0.577223,9.394716)(0.595225,9.340647)(0.614128,9.286217)
(0.628937,9.245160)(0.644324,9.203903)(0.660322,9.162447)(0.676968,9.120793)(0.694302,9.078942)
(0.712368,9.036896)(0.731212,8.994655)(0.750886,8.952221)(0.771445,8.909593)(0.792950,8.866775)
(0.815469,8.823766)(0.839072,8.780567)(0.855451,8.751663)(0.872373,8.722675)(0.889866,8.693605)
(0.907959,8.664451)(0.926685,8.635214)(0.946076,8.605895)(0.966168,8.576494)(0.987000,8.547011)
(0.997707,8.532239)

\newrgbcolor{color163.045}{0  0  0}
\psline[plotstyle=line,linejoin=1,linestyle=solid,linewidth=\LineWidth,linecolor=color163.045]
(0.291376,20.000000)(0.292102,19.950816)
\psline[plotstyle=line,linejoin=1,linestyle=solid,linewidth=\LineWidth,linecolor=color163.045]
(0.998110,5.095745)(1.000000,5.082671)
\psline[plotstyle=line,linejoin=1,linestyle=solid,linewidth=\LineWidth,linecolor=color163.045]
(0.292102,19.950816)(0.302649,19.264294)(0.313385,18.610662)(0.324314,17.987523)(0.335440,17.392706)
(0.346767,16.824247)(0.358300,16.280359)(0.370043,15.759417)(0.382002,15.259937)(0.394181,14.780562)
(0.406584,14.320049)(0.419218,13.877257)(0.432087,13.451137)(0.445196,13.040722)(0.458551,12.645121)
(0.472158,12.263510)(0.486022,11.895126)(0.500149,11.539262)(0.514546,11.195261)(0.529218,10.862515)
(0.544172,10.540455)(0.559415,10.228552)(0.574953,9.926313)(0.590793,9.633277)(0.606942,9.349013)
(0.623408,9.073117)(0.640198,8.805211)(0.657321,8.544939)(0.671263,8.341995)(0.685426,8.143561)
(0.699816,7.949485)(0.714436,7.759620)(0.729291,7.573827)(0.744385,7.391972)(0.759723,7.213927)
(0.775310,7.039573)(0.791150,6.868792)(0.807250,6.701474)(0.823613,6.537512)(0.840245,6.376806)
(0.857151,6.219257)(0.874338,6.064773)(0.891809,5.913265)(0.909572,5.764646)(0.927633,5.618835)
(0.945996,5.475754)(0.964669,5.335326)(0.983657,5.197479)(0.998110,5.095745)

\newrgbcolor{color164.045}{0  0  1}

\newrgbcolor{color165.045}{1  0  0}

\newrgbcolor{color166.045}{1  0  0}
\psline[plotstyle=line,linejoin=1,linestyle=dashed,dash=3pt 2pt 1pt 2pt,linewidth=\LineWidth,linecolor=color166.045]
(0.000000,-5.000000)(0.000000,20.000000)
\end{pspicture}%
\caption{ The shaded region shows the details of the stability region  of Figure~\ref{fig-01} near the origin $(0,\,0)$. }
\label{fig-02}
\end{center}
\end{figure}

\begin{figure}[H]
\begin{center}
\psset{xunit=0.023256\plotwidth,yunit=0.037558\plotwidth}%
\scalebox{0.9}{
\begin{pspicture}(-6.359447,-3.333333)(41.198157,20.491228)%


\psline[linewidth=\AxesLineWidth,linecolor=GridColor](0.000000,-1.000000)(0.000000,-0.680491)
\psline[linewidth=\AxesLineWidth,linecolor=GridColor](5.000000,-1.000000)(5.000000,-0.680491)
\psline[linewidth=\AxesLineWidth,linecolor=GridColor](10.000000,-1.000000)(10.000000,-0.680491)
\psline[linewidth=\AxesLineWidth,linecolor=GridColor](15.000000,-1.000000)(15.000000,-0.680491)
\psline[linewidth=\AxesLineWidth,linecolor=GridColor](20.000000,-1.000000)(20.000000,-0.680491)
\psline[linewidth=\AxesLineWidth,linecolor=GridColor](25.000000,-1.000000)(25.000000,-0.680491)
\psline[linewidth=\AxesLineWidth,linecolor=GridColor](30.000000,-1.000000)(30.000000,-0.680491)
\psline[linewidth=\AxesLineWidth,linecolor=GridColor](35.000000,-1.000000)(35.000000,-0.680491)
\psline[linewidth=\AxesLineWidth,linecolor=GridColor](40.000000,-1.000000)(40.000000,-0.680491)
\psline[linewidth=\AxesLineWidth,linecolor=GridColor](-2.000000,0.000000)(-1.484000,0.000000)
\psline[linewidth=\AxesLineWidth,linecolor=GridColor](-2.000000,2.000000)(-1.484000,2.000000)
\psline[linewidth=\AxesLineWidth,linecolor=GridColor](-2.000000,4.000000)(-1.484000,4.000000)
\psline[linewidth=\AxesLineWidth,linecolor=GridColor](-2.000000,6.000000)(-1.484000,6.000000)
\psline[linewidth=\AxesLineWidth,linecolor=GridColor](-2.000000,8.000000)(-1.484000,8.000000)
\psline[linewidth=\AxesLineWidth,linecolor=GridColor](-2.000000,10.000000)(-1.484000,10.000000)
\psline[linewidth=\AxesLineWidth,linecolor=GridColor](-2.000000,12.000000)(-1.484000,12.000000)
\psline[linewidth=\AxesLineWidth,linecolor=GridColor](-2.000000,14.000000)(-1.484000,14.000000)
\psline[linewidth=\AxesLineWidth,linecolor=GridColor](-2.000000,16.000000)(-1.484000,16.000000)
\psline[linewidth=\AxesLineWidth,linecolor=GridColor](-2.000000,18.000000)(-1.484000,18.000000)
\psline[linewidth=\AxesLineWidth,linecolor=GridColor](-2.000000,20.000000)(-1.484000,20.000000)

{ \footnotesize 
\rput[t](0.000000,-1.319509){$0$}
\rput[t](5.000000,-1.319509){$5$}
\rput[t](10.000000,-1.319509){$10$}
\rput[t](15.000000,-1.319509){$15$}
\rput[t](20.000000,-1.319509){$20$}
\rput[t](25.000000,-1.319509){$25$}
\rput[t](30.000000,-1.319509){$30$}
\rput[t](35.000000,-1.319509){$35$}
\rput[t](40.000000,-1.319509){$40$}
\rput[r](-2.516000,0.000000){$0$}
\rput[r](-2.516000,2.000000){$2$}
\rput[r](-2.516000,4.000000){$4$}
\rput[r](-2.516000,6.000000){$6$}
\rput[r](-2.516000,8.000000){$8$}
\rput[r](-2.516000,10.000000){$10$}
\rput[r](-2.516000,12.000000){$12$}
\rput[r](-2.516000,14.000000){$14$}
\rput[r](-2.516000,16.000000){$16$}
\rput[r](-2.516000,18.000000){$18$}
\rput[r](-2.516000,20.000000){$20$}
} 

\psframe[linewidth=\AxesLineWidth,dimen=middle](-2.000000,-1.000000)(41.000000,20.000000)

{ \small 
\rput[b](19.500000,-3.833333){
\begin{tabular}{c}
$\delta$\\
\end{tabular}
}

\rput[t]{90}(-7.359447,9.500000){
\begin{tabular}{c}
$h_1$\\
\end{tabular}
}
} 

\pspolygon[fillstyle=vlines, fillcolor=black,linecolor=red,linestyle=dashed](0,0)
(0.238811,0.201732)(0.248281,0.229814)(0.257975,0.256157)(0.267896,0.280875)(0.278049,0.304075)
(0.288436,0.325851)(0.298085,0.344489)(0.307933,0.362086)(0.317983,0.378699)(0.328239,0.394381)
(0.339760,0.410616)(0.351538,0.425845)(0.363576,0.440125)(0.375878,0.453508)(0.388449,0.466043)
(0.401294,0.477775)(0.414418,0.488745)(0.427825,0.498992)(0.441520,0.508554)(0.455509,0.517463)
(0.469797,0.525752)(0.485732,0.534121)(0.502037,0.541822)(0.518719,0.548890)(0.535788,0.555354)
(0.554723,0.561710)(0.574132,0.567427)(0.594026,0.572536)(0.616005,0.577392)(0.638575,0.581613)
(0.663431,0.585466)(0.689005,0.588663)(0.717097,0.591386)(0.746057,0.593439)(0.777809,0.594924)
(0.812571,0.595756)(0.850587,0.595854)(0.892133,0.595142)(0.937520,0.593544)(0.987097,0.590992)
(1.043674,0.587244)(1.105530,0.582336)(1.178573,0.575701)(1.264412,0.567044)(1.371201,0.555386)
(1.516542,0.538580)(1.968983,0.484684)(2.173460,0.461046)(2.357986,0.440473)(2.536848,0.421302)
(2.712536,0.403239)(2.889121,0.385853)(3.065100,0.369280)(3.246592,0.352951)(3.433335,0.336925)
(3.624987,0.321257)(3.821112,0.305998)(4.021168,0.291192)(4.235585,0.276124)(4.454043,0.261568)
(4.675729,0.247563)(4.913324,0.233350)(5.153940,0.219736)(5.396379,0.206755)(5.656015,0.193611)
(5.934772,0.180303)(6.215448,0.167679)(6.517008,0.154911)(6.819475,0.142864)(7.144599,0.130690)
(7.469134,0.119272)(7.818026,0.107743)(8.194176,0.096104)(8.568505,0.085261)(8.972048,0.074321)
(9.408430,0.063284)(9.840917,0.053081)(10.308408,0.042794)
(10.017515,0.186551)(9.300219,0.234073)(8.065004,0.334663)(6.637259,0.494488)(5.650407,0.648928)(4.459578,0.920324)
(3.550766,1.243066)(2.968643,1.549207)(2.392867,1.993828)(2.061999,2.359027)(1.801159,2.739878)(1.162526,4.383958)(1.005551,5.105258)(0.849523,6.085163)(0.488747,10.738753)
(0.223280,19.751151)(0.223280,19.98)(0.091212,19.98)(0,19.98)

\newrgbcolor{color161.0573}{0  0  1}
\psline[plotstyle=line,linejoin=1,linestyle=solid,linewidth=\LineWidth,linecolor=color161.0573]
(0.095196,-1.000000)(0.095243,-0.998980)
\psline[plotstyle=line,linejoin=1,linestyle=solid,linewidth=\LineWidth,linecolor=color161.0573]
(40.924894,-0.152520)(41.000000,-0.152660)
\psline[plotstyle=line,linejoin=1,linestyle=solid,linewidth=\LineWidth,linecolor=color161.0573]
(0.095243,-0.998980)(0.101399,-0.874262)(0.107766,-0.760578)(0.114347,-0.656710)(0.120569,-0.569217)
(0.126974,-0.488376)(0.133564,-0.413560)(0.140342,-0.344219)(0.147309,-0.279861)(0.154467,-0.220052)
(0.161819,-0.164403)(0.169367,-0.112567)(0.177114,-0.064231)(0.185061,-0.019118)(0.193212,0.023025)
(0.201569,0.062426)(0.210135,0.099289)(0.218912,0.133800)(0.227903,0.166129)(0.237113,0.196430)
(0.246542,0.224841)(0.256195,0.251491)(0.266075,0.276497)(0.276186,0.299965)(0.286530,0.321994)
(0.296139,0.340847)(0.305947,0.358647)(0.315957,0.375452)(0.326171,0.391317)(0.336593,0.406291)
(0.348300,0.421788)(0.360267,0.436322)(0.372496,0.449945)(0.384994,0.462706)(0.397764,0.474652)
(0.410811,0.485826)(0.424140,0.496267)(0.437756,0.506012)(0.451664,0.515096)(0.465870,0.523551)
(0.481714,0.532092)(0.497925,0.539958)(0.514513,0.547181)(0.531484,0.553793)(0.550312,0.560301)
(0.569610,0.566163)(0.589391,0.571409)(0.611246,0.576406)(0.633688,0.580760)(0.658403,0.584749)
(0.683832,0.588074)(0.711765,0.590928)(0.740560,0.593103)(0.772130,0.594712)(0.806693,0.595668)
(0.844492,0.595890)(0.885799,0.595300)(0.930922,0.593824)(0.980211,0.591391)(1.034058,0.587936)
(1.095425,0.583188)(1.165226,0.576970)(1.247372,0.568822)(1.349807,0.557782)(1.486350,0.542130)
(1.977818,0.483647)(2.178366,0.460489)(2.363355,0.439886)(2.542683,0.420690)(2.718844,0.402605)
(2.895921,0.385199)(3.072408,0.368608)(3.254441,0.352262)(3.441760,0.336220)(3.634024,0.320537)
(3.830796,0.305264)(4.031535,0.290445)(4.235585,0.276124)(4.454043,0.261568)(4.675729,0.247563)
(4.913324,0.233350)(5.153940,0.219736)(5.396379,0.206755)(5.656015,0.193611)(5.934772,0.180303)
(6.215448,0.167679)(6.517008,0.154911)(6.819475,0.142864)(7.144599,0.130690)(7.469134,0.119272)
(7.818026,0.107743)(8.194176,0.096104)(8.568505,0.085261)(8.972048,0.074321)(9.408430,0.063284)
(9.840917,0.053081)(10.308408,0.042794)(10.815412,0.032424)(11.315070,0.022922)(11.856477,0.013349)
(12.445182,0.003705)(13.087765,-0.006013)(13.718625,-0.014821)(14.406317,-0.023691)(15.158989,-0.032621)
(15.890453,-0.040611)(16.689199,-0.048650)(17.565070,-0.056740)(18.529928,-0.064881)(19.458445,-0.072046)
(20.477800,-0.079251)(21.602134,-0.086497)(22.848682,-0.093783)(24.238681,-0.101111)(25.564166,-0.107425)
(27.035001,-0.113771)(28.676632,-0.120148)(30.520816,-0.126557)(32.607710,-0.132998)(34.568932,-0.138391)
(36.772101,-0.143806)(39.265111,-0.149245)(40.924894,-0.152520)

\newrgbcolor{color162.0568}{0  0  0}
\psline[plotstyle=line,linejoin=1,linestyle=solid,linewidth=\LineWidth,linecolor=color162.0568]
(0.264768,20.000000)(0.274186,19.304109)
\psline[plotstyle=line,linejoin=1,linestyle=solid,linewidth=\LineWidth,linecolor=color162.0568]
(39.460482,14.494429)(39.464193,20.000000)
\psline[plotstyle=line,linejoin=1,linestyle=solid,linewidth=\LineWidth,linecolor=color162.0568]
(0.274186,19.304109)(0.284375,18.605197)(0.294567,17.954456)(0.304761,17.347070)(0.314958,16.778843)
(0.325158,16.246104)(0.335361,15.745626)(0.345566,15.274564)(0.355774,14.830395)(0.365985,14.410879)
(0.376199,14.014017)(0.386416,13.638020)(0.396635,13.281283)(0.406858,12.942362)(0.417083,12.619954)
(0.427312,12.312879)(0.437543,12.020067)(0.447778,11.740545)(0.458015,11.473428)(0.468256,11.217907)
(0.478500,10.973241)(0.488747,10.738753)(0.498997,10.513819)(0.509250,10.297865)(0.519506,10.090364)
(0.529766,9.890828)(0.540028,9.698806)(0.550294,9.513881)(0.560564,9.335665)(0.570836,9.163800)
(0.581112,8.997950)(0.591391,8.837806)(0.601674,8.683076)(0.611960,8.533490)(0.622249,8.388795)
(0.632542,8.248755)(0.642839,8.113147)(0.653138,7.981765)(0.663442,7.854412)(0.673748,7.730907)
(0.684059,7.611076)(0.694373,7.494757)(0.704690,7.381799)(0.715011,7.272056)(0.725336,7.165393)
(0.735664,7.061682)(0.745996,6.960802)(0.756332,6.862637)(0.766671,6.767079)(0.777015,6.674025)
(0.787362,6.583378)(0.797712,6.495045)(0.808067,6.408938)(0.818425,6.324973)(0.828787,6.243073)
(0.839153,6.163160)(0.849523,6.085163)(0.859897,6.009013)(0.870275,5.934647)(0.880656,5.862000)
(0.891042,5.791014)(0.901432,5.721633)(0.911825,5.653802)(0.922223,5.587470)(0.932625,5.522587)
(0.943030,5.459106)(0.953440,5.396981)(0.963854,5.336170)(0.974272,5.276631)(0.984694,5.218325)
(0.995121,5.161212)(1.005551,5.105258)(1.015986,5.050425)(1.026425,4.996682)(1.036869,4.943995)
(1.047316,4.892334)(1.057768,4.841668)(1.078684,4.743210)(1.099618,4.648403)(1.120570,4.557048)
(1.141539,4.468957)(1.162526,4.383958)(1.183531,4.301888)(1.204554,4.222598)(1.225596,4.145948)
(1.246655,4.071807)(1.267734,4.000053)(1.288831,3.930571)(1.309947,3.863253)(1.331082,3.798000)
(1.352237,3.734717)(1.373410,3.673314)(1.394603,3.613709)(1.415816,3.555823)(1.437049,3.499581)
(1.458301,3.444914)(1.479574,3.391756)(1.500866,3.340044)(1.522179,3.289719)(1.543512,3.240726)
(1.564866,3.193012)(1.586241,3.146526)(1.607637,3.101221)(1.629053,3.057053)(1.650491,3.013977)
(1.671950,2.971954)(1.693431,2.930945)(1.714933,2.890913)(1.736456,2.851823)(1.758002,2.813641)
(1.779569,2.776336)(1.801159,2.739878)(1.822770,2.704237)(1.844404,2.669386)(1.866061,2.635298)
(1.887740,2.601948)(1.909441,2.569312)(1.931166,2.537367)(1.963796,2.490696)(1.996477,2.445460)
(2.029212,2.401591)(2.061999,2.359027)(2.094839,2.317710)(2.127734,2.277585)(2.160682,2.238600)
(2.193685,2.200705)(2.226742,2.163855)(2.259855,2.128006)(2.293024,2.093117)(2.326248,2.059148)
(2.359529,2.026064)(2.392867,1.993828)(2.426262,1.962409)(2.459714,1.931774)(2.493224,1.901894)
(2.526793,1.872741)(2.560420,1.844287)(2.594106,1.816508)(2.627852,1.789377)(2.661657,1.762874)
(2.695522,1.736975)(2.729447,1.711659)(2.763433,1.686907)(2.797480,1.662699)(2.831589,1.639016)
(2.877163,1.608228)(2.922847,1.578304)(2.968643,1.549207)(3.014551,1.520903)(3.060571,1.493359)
(3.106705,1.466543)(3.152952,1.440426)(3.199315,1.414980)(3.245793,1.390178)(3.292388,1.365996)
(3.339099,1.342410)(3.385928,1.319397)(3.432876,1.296935)(3.491727,1.269604)(3.550766,1.243066)
(3.609993,1.217285)(3.669409,1.192229)(3.729016,1.167866)(3.788815,1.144167)(3.848807,1.121104)
(3.908993,1.098651)(3.969375,1.076782)(4.029952,1.055475)(4.090728,1.034708)(4.163919,1.010469)
(4.237399,0.986942)(4.311167,0.964096)(4.385226,0.941899)(4.459578,0.920324)(4.534223,0.899343)
(4.609163,0.878932)(4.696969,0.855807)(4.785180,0.833390)(4.873800,0.811647)(4.962830,0.790547)
(5.052272,0.770061)(5.142128,0.750162)(5.232399,0.730824)(5.336078,0.709380)(5.440304,0.688603)
(5.545080,0.668462)(5.650407,0.648928)(5.756287,0.629972)(5.876067,0.609306)(5.996551,0.589305)
(6.117741,0.569938)(6.239640,0.551173)(6.362249,0.532982)(6.499315,0.513413)(6.637259,0.494488)
(6.776083,0.476174)(6.915787,0.458444)(7.070477,0.439581)(7.226232,0.421357)(7.383049,0.403741)
(7.540927,0.386703)(7.714365,0.368744)(7.889059,0.351407)(8.065004,0.334663)(8.257015,0.317159)
(8.450480,0.300284)(8.645388,0.284006)(8.856890,0.267112)(9.070038,0.250844)(9.300219,0.234073)
(9.532246,0.217951)(9.766097,0.202445)(10.017515,0.186551)(10.270939,0.171290)(10.542359,0.155734)
(10.815950,0.140823)(11.107944,0.125703)(11.402246,0.111236)(11.715322,0.096639)(12.030801,0.082698)
(12.365375,0.068699)(12.702389,0.055359)(13.058745,0.042026)(13.434648,0.028762)(13.813013,0.016185)
(14.211003,0.003733)(14.628707,-0.008536)(15.048612,-0.020100)(15.488061,-0.031434)(15.946968,-0.042486)
(16.425170,-0.053208)(16.922408,-0.063552)(17.438326,-0.073473)(17.954650,-0.082627)(18.488628,-0.091326)
(19.039613,-0.099530)(19.606831,-0.107199)(20.189375,-0.114298)(20.786206,-0.120791)(21.396151,-0.126644)
(22.017903,-0.131826)(22.650027,-0.136305)(23.290972,-0.140055)(23.939082,-0.143047)(24.576420,-0.145212)
(25.218009,-0.146612)(25.846538,-0.147217)(26.476487,-0.147051)(27.091336,-0.146123)(27.705000,-0.144415)
(28.301894,-0.141974)(28.881762,-0.138829)(29.444490,-0.135008)(29.990097,-0.130534)(30.518731,-0.125430)
(31.030658,-0.119712)(31.526257,-0.113396)(31.994766,-0.106662)(32.448705,-0.099369)(32.878109,-0.091716)
(33.294844,-0.083520)(33.689609,-0.074992)(34.063923,-0.066146)(34.419212,-0.056989)(34.756808,-0.047519)
(35.068858,-0.038012)(35.356990,-0.028501)(35.631389,-0.018690)(35.884226,-0.008895)(36.116567,0.000852)
(36.329304,0.010515)(36.523174,0.020046)(36.698783,0.029384)(36.864904,0.038947)(37.013563,0.048215)
(37.153298,0.057651)(37.284283,0.067248)(37.398517,0.076316)(37.504310,0.085396)(37.601750,0.094435)
(37.690909,0.103363)(37.779932,0.113005)(37.860759,0.122489)(37.933430,0.131701)(38.006042,0.141653)
(38.070545,0.151211)(38.126958,0.160205)(38.183351,0.169866)(38.239728,0.180287)(38.288042,0.189906)
(38.336350,0.200244)(38.376605,0.209477)(38.416860,0.219340)(38.457116,0.229909)(38.497375,0.241270)
(38.529585,0.250998)(38.561799,0.261359)(38.594018,0.272423)(38.626242,0.284270)(38.650415,0.293724)
(38.674591,0.303715)(38.698771,0.314293)(38.722956,0.325516)(38.747146,0.337446)(38.771341,0.350158)
(38.795541,0.363735)(38.819748,0.378273)(38.835889,0.388553)(38.852033,0.399345)(38.868180,0.410692)
(38.884330,0.422637)(38.900483,0.435232)(38.916640,0.448533)(38.932800,0.462604)(38.948963,0.477514)
(38.965130,0.493343)(38.981301,0.510183)(38.997476,0.528134)(39.005564,0.537563)(39.013654,0.547315)
(39.021745,0.557407)(39.029837,0.567858)(39.037929,0.578687)(39.046023,0.589917)(39.054118,0.601569)
(39.062214,0.613670)(39.070311,0.626245)(39.078409,0.639323)(39.086508,0.652936)(39.094609,0.667118)
(39.102710,0.681906)(39.110813,0.697340)(39.118916,0.713464)(39.127021,0.730325)(39.135128,0.747977)
(39.143235,0.766477)(39.151343,0.785887)(39.159453,0.806278)(39.167564,0.827728)(39.175676,0.850320)
(39.183790,0.874150)(39.191904,0.899323)(39.200020,0.925958)(39.208138,0.954186)(39.216256,0.984156)
(39.224376,1.016036)(39.232497,1.050014)(39.240620,1.086307)(39.248744,1.125160)(39.256869,1.166858)
(39.264996,1.211724)(39.273124,1.260137)(39.281253,1.312536)(39.289384,1.369436)(39.297516,1.431447)
(39.305650,1.499291)(39.313785,1.573835)(39.321922,1.656126)(39.330060,1.747441)(39.338200,1.849354)
(39.346341,1.963829)(39.354484,2.093346)(39.362628,2.241083)(39.370774,2.411183)(39.378921,2.609145)
(39.387070,2.842437)(39.395220,3.121452)(39.403372,3.461106)(39.411526,3.883596)(39.419681,4.423452)
(39.427838,5.137501)(39.435997,6.126311)(39.444157,7.586313)(39.452319,9.959771)(39.460482,14.494429)

\newrgbcolor{color163.0568}{1  0  0}
\psline[plotstyle=line,linejoin=1,linestyle=dashed,dash=3pt 2pt 1pt 2pt,linewidth=\LineWidth,linecolor=color163.0568]
(-2.000000,0.000000)(2.000000,0.000000)
\psline[plotstyle=line,linejoin=1,linestyle=dashed,dash=3pt 2pt 1pt 2pt,linewidth=\LineWidth,linecolor=color163.0568]
(37.000000,0.000000)(41.000000,0.000000)
\psline[plotstyle=line,linejoin=1,linestyle=dashed,dash=3pt 2pt 1pt 2pt,linewidth=\LineWidth,linecolor=color163.0568]
(2.000000,0.000000)(37.000000,0.000000)

\newrgbcolor{color164.0568}{1  0  0}
\psline[plotstyle=line,linejoin=1,linestyle=dashed,dash=3pt 2pt 1pt 2pt,linewidth=\LineWidth,linecolor=color164.0568]
(0.000000,19.000000)(0.000000,20.000000)
\psline[plotstyle=line,linejoin=1,linestyle=dashed,dash=3pt 2pt 1pt 2pt,linewidth=\LineWidth,linecolor=color164.0568]
(0.000000,-1.000000)(0.000000,4.000000)(0.000000,9.000000)(0.000000,14.000000)(0.000000,19.000000)

{ \small 
\newrgbcolor{color234.0145}{0  0  0}
\psline[linestyle=solid,linewidth=0.5pt,linecolor=color234.0145,arrowsize=1.5pt 3,arrowlength=2,arrowinset=0.3]{->}(37,2)(40,-0.25)
\uput{0pt}[33.769825](22,2){%
\psframebox*[framesep=1pt]{\begin{tabular}{@{}c@{}}
$\beta\in (2(n-1)\pi,\,\tilde{\beta}_{n,\,1})$\\[-0.3ex]
\end{tabular}}}
}

{ \small 
\newrgbcolor{color234.0145}{0  0  0}
\psline[linestyle=solid,linewidth=0.5pt,linecolor=color234.0145,arrowsize=1.5pt 3,arrowlength=2,arrowinset=0.3]{->}(32,11)(39,11)
\uput{0pt}[33.769825](22,10.5){%
\psframebox*[framesep=1pt]{\begin{tabular}{@{}c@{}}
$\beta\in (\beta_n^*,\,2n\pi)$\\[-0.3ex]
\end{tabular}}}
}
\end{pspicture}
}%
\caption{The curves of  $(\delta,\,h_1)$, $\delta>0$ where   $\xi=0.2$, $q=0.8$,  $n=1$, $\beta_1\frac{\sqrt{2\sqrt{5}-2}}{4\sqrt{5}-8}=8.955929>q$,  $\beta_n^*=3.26398905$, $\tilde{\beta}_{n,\,1}=2.28275758$. The upper branch is the curve determined by (\ref{lemma-2-3-1}) with $\beta\in (\beta_n^*,\,2n\pi)$ and the lower one with $\beta\in (2(n-1)\pi,\,\tilde{\beta}_{n,\,1})$. The shaded region near the origin $(0,\,0)$ is the stability region.}
\label{fig-4}
\end{center}
\end{figure}
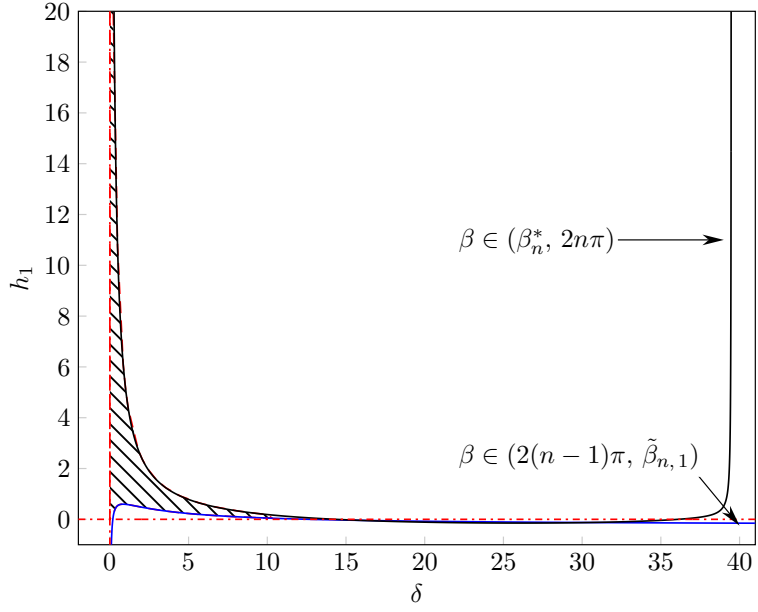

\begin{figure}[H]
\psset{xunit=0.022727\plotwidth,yunit=0.071701\plotwidth}
\begin{center}
\begin{pspicture}(-6.460829,-2.222222)(42.202765,10.257310)%


\psline[linewidth=\AxesLineWidth,linecolor=GridColor](0.000000,-1.000000)(0.000000,-0.832638)
\psline[linewidth=\AxesLineWidth,linecolor=GridColor](5.000000,-1.000000)(5.000000,-0.832638)
\psline[linewidth=\AxesLineWidth,linecolor=GridColor](10.000000,-1.000000)(10.000000,-0.832638)
\psline[linewidth=\AxesLineWidth,linecolor=GridColor](15.000000,-1.000000)(15.000000,-0.832638)
\psline[linewidth=\AxesLineWidth,linecolor=GridColor](20.000000,-1.000000)(20.000000,-0.832638)
\psline[linewidth=\AxesLineWidth,linecolor=GridColor](25.000000,-1.000000)(25.000000,-0.832638)
\psline[linewidth=\AxesLineWidth,linecolor=GridColor](30.000000,-1.000000)(30.000000,-0.832638)
\psline[linewidth=\AxesLineWidth,linecolor=GridColor](35.000000,-1.000000)(35.000000,-0.832638)
\psline[linewidth=\AxesLineWidth,linecolor=GridColor](40.000000,-1.000000)(40.000000,-0.832638)
\psline[linewidth=\AxesLineWidth,linecolor=GridColor](-2.000000,-1.000000)(-1.472000,-1.000000)
\psline[linewidth=\AxesLineWidth,linecolor=GridColor](-2.000000,0.000000)(-1.472000,0.000000)
\psline[linewidth=\AxesLineWidth,linecolor=GridColor](-2.000000,1.000000)(-1.472000,1.000000)
\psline[linewidth=\AxesLineWidth,linecolor=GridColor](-2.000000,2.000000)(-1.472000,2.000000)
\psline[linewidth=\AxesLineWidth,linecolor=GridColor](-2.000000,3.000000)(-1.472000,3.000000)
\psline[linewidth=\AxesLineWidth,linecolor=GridColor](-2.000000,4.000000)(-1.472000,4.000000)
\psline[linewidth=\AxesLineWidth,linecolor=GridColor](-2.000000,5.000000)(-1.472000,5.000000)
\psline[linewidth=\AxesLineWidth,linecolor=GridColor](-2.000000,6.000000)(-1.472000,6.000000)
\psline[linewidth=\AxesLineWidth,linecolor=GridColor](-2.000000,7.000000)(-1.472000,7.000000)
\psline[linewidth=\AxesLineWidth,linecolor=GridColor](-2.000000,8.000000)(-1.472000,8.000000)
\psline[linewidth=\AxesLineWidth,linecolor=GridColor](-2.000000,9.000000)(-1.472000,9.000000)
\psline[linewidth=\AxesLineWidth,linecolor=GridColor](-2.000000,10.000000)(-1.472000,10.000000)

{ \footnotesize 
\rput[t](0.000000,-1.167362){$0$}
\rput[t](5.000000,-1.167362){$5$}
\rput[t](10.000000,-1.167362){$10$}
\rput[t](15.000000,-1.167362){$15$}
\rput[t](20.000000,-1.167362){$20$}
\rput[t](25.000000,-1.167362){$25$}
\rput[t](30.000000,-1.167362){$30$}
\rput[t](35.000000,-1.167362){$35$}
\rput[t](40.000000,-1.167362){$40$}
\rput[r](-2.528000,-1.000000){$-1$}
\rput[r](-2.528000,0.000000){$0$}
\rput[r](-2.528000,1.000000){$1$}
\rput[r](-2.528000,2.000000){$2$}
\rput[r](-2.528000,3.000000){$3$}
\rput[r](-2.528000,4.000000){$4$}
\rput[r](-2.528000,5.000000){$5$}
\rput[r](-2.528000,6.000000){$6$}
\rput[r](-2.528000,7.000000){$7$}
\rput[r](-2.528000,8.000000){$8$}
\rput[r](-2.528000,9.000000){$9$}
\rput[r](-2.528000,10.000000){$10$}
} 

\psframe[linewidth=\AxesLineWidth,dimen=middle](-2.000000,-1.000000)(42.000000,10.000000)

{ \small 
\rput[b](20.000000,-2.522222){
\begin{tabular}{c}
$\delta$\\
\end{tabular}
}

\rput[t]{90}(-6.460829,4.500000){
\begin{tabular}{c}
$h_2$\\
\end{tabular}
}
} 


\pspolygon[fillstyle=vlines, fillcolor=black,linecolor=red,linestyle=dashed](0,0)
(0.539656,0.000000)(0.559562,0.014234)(0.579788,0.027780)(0.602054,0.041736)(0.624687,0.055003)
(0.647681,0.067635)(0.672843,0.080584)(0.698412,0.092908)(0.726252,0.105478)(0.756447,0.118224)
(0.787147,0.130347)(0.820306,0.142607)(0.856010,0.154959)(0.894342,0.167364)(0.935384,0.179794)
(0.979217,0.192230)(1.025917,0.204660)(1.077730,0.217613)(1.134826,0.231028)(1.197360,0.244872)
(1.267753,0.259602)(1.348530,0.275643)(1.446999,0.294305)(1.592208,0.320801)(1.789200,0.356489)
(1.894459,0.376231)(1.979666,0.392950)(2.051992,0.407891)(2.113921,0.421411)(2.170157,0.434426)
(2.220769,0.446885)(2.265886,0.458721)(2.307745,0.470450)(2.346385,0.482051)(2.381866,0.493496)
(2.414268,0.504749)(2.443690,0.515771)(2.470246,0.526516)(2.494066,0.536934)(2.515290,0.546969)
(2.535596,0.557377)(2.553475,0.567340)(2.570472,0.577668)(2.585236,0.587475)(2.599177,0.597630)
(2.611100,0.607182)(2.622273,0.617058)(2.631660,0.626238)(2.640385,0.635712)(2.648421,0.645493)
(2.655739,0.655601)(2.661619,0.664874)(2.666886,0.674434)(2.671520,0.684294)(2.675496,0.694470)
(2.678792,0.704980)(2.681381,0.715841)(2.683240,0.727073)(2.684245,0.737222)(2.684652,0.747685)
(2.684440,0.758478)(2.683592,0.769618)(2.682086,0.781124)(2.679532,0.794744)(2.676059,0.808898)
(2.671634,0.823618)(2.666222,0.838941)(2.659785,0.854905)(2.652286,0.871554)(2.643687,0.888934)
(2.633945,0.907096)(2.621567,0.928533)(2.607625,0.951118)(2.592054,0.974947)(2.572757,1.003018)
(2.551261,1.032927)(2.524949,1.068181)(2.492889,1.109785)(2.447580,1.167079)(2.312530,1.335239)
(2.258357,1.403740)(2.214425,1.460336)(2.172877,1.514951)(2.134515,1.566485)(2.100085,1.613757)
(2.064178,1.664194)(2.033093,1.708881)(2.000920,1.756212)(1.967631,1.806428)(1.933195,1.859805)
(1.904798,1.904988)(1.875631,1.952567)(1.845675,2.002739)(1.814914,2.055722)(1.791303,2.097450)
(1.767221,2.141009)(1.742660,2.186523)(1.717611,2.234127)(1.692067,2.283968)(1.666019,2.336208)
(1.639458,2.391026)(1.612375,2.448616)(1.584761,2.509196)(1.556608,2.573004)(1.537534,2.617466)
(1.518213,2.663567)(1.498642,2.711398)(1.478818,2.761060)(1.458739,2.812658)(1.438400,2.866310)
(1.417800,2.922140)(1.396934,2.980284)(1.375800,3.040888)(1.354395,3.104113)(1.332715,3.170133)
(1.310756,3.239136)(1.288517,3.311331)(1.265992,3.386944)(1.243179,3.466224)(1.231664,3.507324)
(1.220075,3.549445)(1.208412,3.592626)(1.196675,3.636907)(1.184863,3.682332)(1.172975,3.728945)
(1.161013,3.776793)(1.148973,3.825927)(1.136858,3.876398)(1.124665,3.928262)(1.112395,3.981577)
(1.100046,4.036407)(1.087620,4.092815)(1.075114,4.150871)(1.062528,4.210648)(1.049863,4.272225)
(1.037117,4.335683)(1.024290,4.401111)(1.011381,4.468601)(0.998390,4.538252)(0.985317,4.610171)
(0.972161,4.684470)(0.958921,4.761268)(0.945597,4.840695)(0.932188,4.922887)(0.918694,5.007991)
(0.905115,5.096166)(0.891449,5.187581)(0.877696,5.282417)(0.863856,5.380870)(0.849927,5.483151)
(0.835910,5.589489)(0.821805,5.700129)(0.807609,5.815337)(0.793323,5.935404)(0.778946,6.060642)
(0.764478,6.191394)(0.749918,6.328031)(0.735265,6.470961)(0.720519,6.620629)(0.705679,6.777521)
(0.690745,6.942176)(0.675715,7.115182)(0.660590,7.297194)(0.645369,7.488931)(0.630050,7.691196)
(0.614634,7.904880)(0.599120,8.130979)(0.583507,8.370606)(0.567794,8.625014)(0.551981,8.895614)
(0.536067,9.184002)(0.520051,9.491993)(0.503933,9.821656)
(0.503933,9.99)(0,9.99)

\newrgbcolor{color161.0573}{0  0  1}
\psline[plotstyle=line,linejoin=1,linestyle=solid,linewidth=\LineWidth,linecolor=color161.0573]
(0.503933,9.821656)(0.495754,10.000000)
\psline[plotstyle=line,linejoin=1,linestyle=solid,linewidth=\LineWidth,linecolor=color161.0573]
(0.539656,0.000000)(0.559562,0.014234)(0.579788,0.027780)(0.602054,0.041736)(0.624687,0.055003)
(0.647681,0.067635)(0.672843,0.080584)(0.698412,0.092908)(0.726252,0.105478)(0.756447,0.118224)
(0.787147,0.130347)(0.820306,0.142607)(0.856010,0.154959)(0.894342,0.167364)(0.935384,0.179794)
(0.979217,0.192230)(1.025917,0.204660)(1.077730,0.217613)(1.134826,0.231028)(1.197360,0.244872)
(1.267753,0.259602)(1.348530,0.275643)(1.446999,0.294305)(1.592208,0.320801)(1.789200,0.356489)
(1.894459,0.376231)(1.979666,0.392950)(2.051992,0.407891)(2.113921,0.421411)(2.170157,0.434426)
(2.220769,0.446885)(2.265886,0.458721)(2.307745,0.470450)(2.346385,0.482051)(2.381866,0.493496)
(2.414268,0.504749)(2.443690,0.515771)(2.470246,0.526516)(2.494066,0.536934)(2.515290,0.546969)
(2.535596,0.557377)(2.553475,0.567340)(2.570472,0.577668)(2.585236,0.587475)(2.599177,0.597630)
(2.611100,0.607182)(2.622273,0.617058)(2.631660,0.626238)(2.640385,0.635712)(2.648421,0.645493)
(2.655739,0.655601)(2.661619,0.664874)(2.666886,0.674434)(2.671520,0.684294)(2.675496,0.694470)
(2.678792,0.704980)(2.681381,0.715841)(2.683240,0.727073)(2.684245,0.737222)(2.684652,0.747685)
(2.684440,0.758478)(2.683592,0.769618)(2.682086,0.781124)(2.679532,0.794744)(2.676059,0.808898)
(2.671634,0.823618)(2.666222,0.838941)(2.659785,0.854905)(2.652286,0.871554)(2.643687,0.888934)
(2.633945,0.907096)(2.621567,0.928533)(2.607625,0.951118)(2.592054,0.974947)(2.572757,1.003018)
(2.551261,1.032927)(2.524949,1.068181)(2.492889,1.109785)(2.447580,1.167079)(2.312530,1.335239)
(2.258357,1.403740)(2.214425,1.460336)(2.172877,1.514951)(2.134515,1.566485)(2.100085,1.613757)
(2.064178,1.664194)(2.033093,1.708881)(2.000920,1.756212)(1.967631,1.806428)(1.933195,1.859805)
(1.904798,1.904988)(1.875631,1.952567)(1.845675,2.002739)(1.814914,2.055722)(1.791303,2.097450)
(1.767221,2.141009)(1.742660,2.186523)(1.717611,2.234127)(1.692067,2.283968)(1.666019,2.336208)
(1.639458,2.391026)(1.612375,2.448616)(1.584761,2.509196)(1.556608,2.573004)(1.537534,2.617466)
(1.518213,2.663567)(1.498642,2.711398)(1.478818,2.761060)(1.458739,2.812658)(1.438400,2.866310)
(1.417800,2.922140)(1.396934,2.980284)(1.375800,3.040888)(1.354395,3.104113)(1.332715,3.170133)
(1.310756,3.239136)(1.288517,3.311331)(1.265992,3.386944)(1.243179,3.466224)(1.231664,3.507324)
(1.220075,3.549445)(1.208412,3.592626)(1.196675,3.636907)(1.184863,3.682332)(1.172975,3.728945)
(1.161013,3.776793)(1.148973,3.825927)(1.136858,3.876398)(1.124665,3.928262)(1.112395,3.981577)
(1.100046,4.036407)(1.087620,4.092815)(1.075114,4.150871)(1.062528,4.210648)(1.049863,4.272225)
(1.037117,4.335683)(1.024290,4.401111)(1.011381,4.468601)(0.998390,4.538252)(0.985317,4.610171)
(0.972161,4.684470)(0.958921,4.761268)(0.945597,4.840695)(0.932188,4.922887)(0.918694,5.007991)
(0.905115,5.096166)(0.891449,5.187581)(0.877696,5.282417)(0.863856,5.380870)(0.849927,5.483151)
(0.835910,5.589489)(0.821805,5.700129)(0.807609,5.815337)(0.793323,5.935404)(0.778946,6.060642)
(0.764478,6.191394)(0.749918,6.328031)(0.735265,6.470961)(0.720519,6.620629)(0.705679,6.777521)
(0.690745,6.942176)(0.675715,7.115182)(0.660590,7.297194)(0.645369,7.488931)(0.630050,7.691196)
(0.614634,7.904880)(0.599120,8.130979)(0.583507,8.370606)(0.567794,8.625014)(0.551981,8.895614)
(0.536067,9.184002)(0.520051,9.491993)(0.503933,9.821656)

\newrgbcolor{color162.0568}{0  0  0}
\psline[plotstyle=line,linejoin=1,linestyle=solid,linewidth=\LineWidth,linecolor=color162.0568]
(39.447967,8.536224)(39.451088,10.000000)
\psline[plotstyle=line,linejoin=1,linestyle=solid,linewidth=\LineWidth,linecolor=color162.0568]
(34.434283,-0.000000)(35.049038,0.015379)(35.451622,0.026117)(35.782933,0.035666)(36.057781,0.044293)
(36.304922,0.052784)(36.522747,0.061017)(36.709950,0.068802)(36.876643,0.076418)(37.033466,0.084303)
(37.168814,0.091786)(37.293619,0.099353)(37.407676,0.106933)(37.510803,0.114430)(37.614375,0.122682)
(37.706804,0.130762)(37.787958,0.138507)(37.869367,0.146985)(37.939348,0.154926)(38.009512,0.163586)
(38.068120,0.171433)(38.126852,0.179934)(38.185707,0.189179)(38.232879,0.197186)(38.280129,0.205804)
(38.327455,0.215113)(38.374857,0.225202)(38.410459,0.233344)(38.446102,0.242039)(38.481787,0.251344)
(38.517515,0.261331)(38.553283,0.272079)(38.589093,0.283681)(38.612988,0.291944)(38.636902,0.300671)
(38.660834,0.309905)(38.684784,0.319691)(38.708752,0.330082)(38.732737,0.341138)(38.756741,0.352923)
(38.780762,0.365517)(38.792779,0.372143)(38.804800,0.379004)(38.816826,0.386114)(38.828856,0.393487)
(38.840890,0.401137)(38.852929,0.409080)(38.864973,0.417334)(38.877020,0.425919)(38.889072,0.434853)
(38.901128,0.444161)(38.913189,0.453864)(38.925253,0.463990)(38.937322,0.474568)(38.949396,0.485627)
(38.961473,0.497203)(38.973555,0.509333)(38.985641,0.522057)(38.997731,0.535421)(39.009826,0.549475)
(39.021925,0.564274)(39.034028,0.579879)(39.046135,0.596358)(39.058246,0.613788)(39.070362,0.632254)
(39.082481,0.651850)(39.094605,0.672686)(39.106733,0.694883)(39.118865,0.718579)(39.131002,0.743932)
(39.143142,0.771124)(39.155287,0.800362)(39.167435,0.831888)(39.179588,0.865981)(39.191745,0.902970)
(39.203906,0.943240)(39.216071,0.987251)(39.228240,1.035550)(39.240413,1.088796)(39.252590,1.147794)
(39.264771,1.213529)(39.276956,1.287228)(39.289145,1.370433)(39.301339,1.465113)(39.313536,1.573821)
(39.325737,1.699927)(39.337942,1.847975)(39.350151,2.024243)(39.362364,2.237654)(39.374582,2.501339)
(39.386803,2.835446)(39.399027,3.272543)(39.411256,3.868949)(39.423489,4.731192)(39.435726,6.088026)
(39.447967,8.536224)

\newrgbcolor{color163.0568}{1  0  0}
\psline[plotstyle=line,linejoin=1,linestyle=dashed,dash=3pt 2pt 1pt 2pt,linewidth=\LineWidth,linecolor=color163.0568]
(39.000000,0.000000)(42.000000,0.000000)
\psline[plotstyle=line,linejoin=1,linestyle=dashed,dash=3pt 2pt 1pt 2pt,linewidth=\LineWidth,linecolor=color163.0568]
(-1.000000,0.000000)(39.000000,0.000000)

\newrgbcolor{color164.0568}{1  0  0}
\psline[plotstyle=line,linejoin=1,linestyle=dashed,dash=3pt 2pt 1pt 2pt,linewidth=\LineWidth,linecolor=color164.0568]
(0.000000,8.000000)(0.000000,10.000000)
\psline[plotstyle=line,linejoin=1,linestyle=dashed,dash=3pt 2pt 1pt 2pt,linewidth=\LineWidth,linecolor=color164.0568]
(0.000000,-1.000000)(0.000000,2.000000)(0.000000,5.000000)(0.000000,8.000000)

{ \small 
\newrgbcolor{color234.0145}{0  0  0}
\psline[linestyle=solid,linewidth=0.5pt,linecolor=color234.0145,arrowsize=1.5pt 3,arrowlength=2,arrowinset=0.3]{->}(6,1)(3,1)
\uput{0pt}[33.769825](6,0.8){%
\psframebox*[framesep=1pt]{\begin{tabular}{@{}c@{}}
$\beta\in (\gamma_n^*,\,\beta_n^*)$\\[-0.3ex]
\end{tabular}}}
}

{ \small 
\newrgbcolor{color234.0145}{0  0  0}
\psline[linestyle=solid,linewidth=0.5pt,linecolor=color234.0145,arrowsize=1.5pt 3,arrowlength=2,arrowinset=0.3]{->}(35,1)(39,1)
\uput{0pt}[33.769825](25,0.8){%
\psframebox*[framesep=1pt]{\begin{tabular}{@{}c@{}}
$\beta\in (\tilde{\gamma}_n,\,2n\pi)$\\[-0.3ex]
\end{tabular}}}
}
\end{pspicture}
\caption{The curves of  $(\delta,\,h_2)$, $\delta>0$, $h_2>0$ where   $\xi=0.2$, $n=1$,  $q=0.8$ with $\xi<2q$, $\beta_n^*=3.26398905$, $\tilde{\beta}_{n}=3.91714943$, $\gamma_n^*=0.95335728$ and $\tilde{\gamma}_n=5.59545581$. The left-hand side branch is the curve determined by (\ref{lemma-3-3-1}) with $\beta\in (\gamma_n^*,\,\beta_n^*)$ and the right-hand side one with $\beta\in (\tilde{\gamma}_n,\,2n\pi)$. The shaded region near the origin $(0,\,0)$ is the stability region.}
\label{fig-5}
\end{center}
\end{figure}
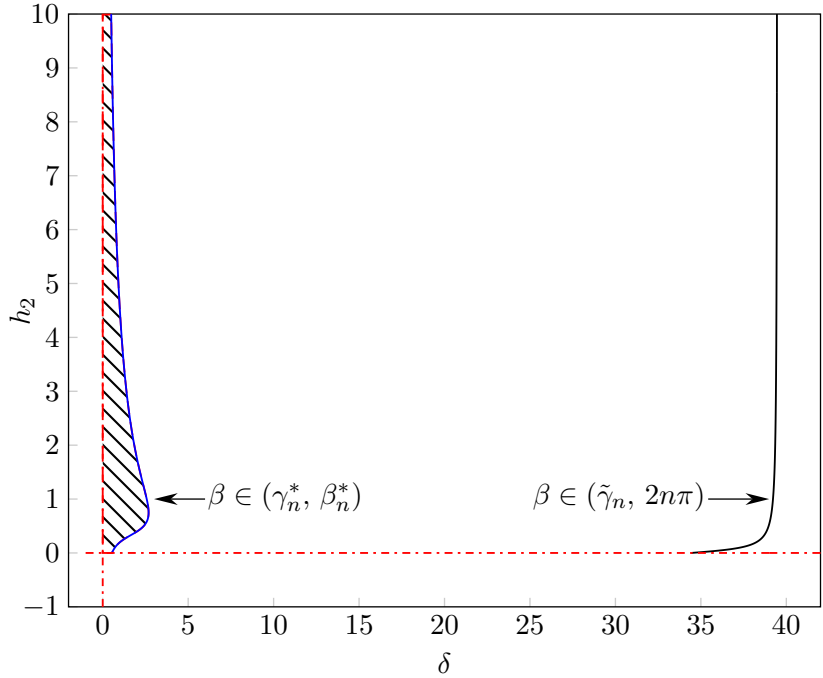

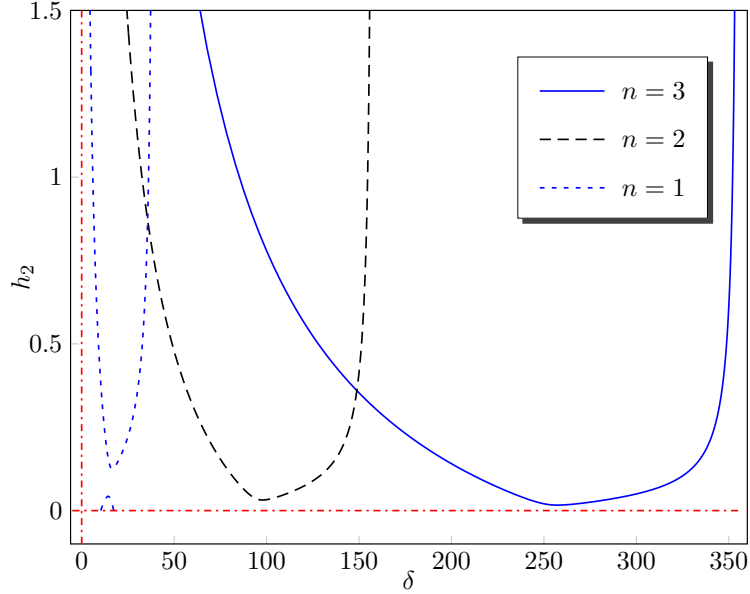
\begin{figure}[H]
\psset{xunit=0.002732\plotwidth,yunit=0.492944\plotwidth}%
\begin{center}
\scalebox{0.9}{
\begin{pspicture}(-46.479263,-0.277778)(361.686636,1.537427)%


\psline[linewidth=\AxesLineWidth,linecolor=GridColor](0.000000,-0.100000)(0.000000,-0.075656)
\psline[linewidth=\AxesLineWidth,linecolor=GridColor](50.000000,-0.100000)(50.000000,-0.075656)
\psline[linewidth=\AxesLineWidth,linecolor=GridColor](100.000000,-0.100000)(100.000000,-0.075656)
\psline[linewidth=\AxesLineWidth,linecolor=GridColor](150.000000,-0.100000)(150.000000,-0.075656)
\psline[linewidth=\AxesLineWidth,linecolor=GridColor](200.000000,-0.100000)(200.000000,-0.075656)
\psline[linewidth=\AxesLineWidth,linecolor=GridColor](250.000000,-0.100000)(250.000000,-0.075656)
\psline[linewidth=\AxesLineWidth,linecolor=GridColor](300.000000,-0.100000)(300.000000,-0.075656)
\psline[linewidth=\AxesLineWidth,linecolor=GridColor](350.000000,-0.100000)(350.000000,-0.075656)
\psline[linewidth=\AxesLineWidth,linecolor=GridColor](-6.000000,0.000000)(-1.608000,0.000000)
\psline[linewidth=\AxesLineWidth,linecolor=GridColor](-6.000000,0.500000)(-1.608000,0.500000)
\psline[linewidth=\AxesLineWidth,linecolor=GridColor](-6.000000,1.000000)(-1.608000,1.000000)
\psline[linewidth=\AxesLineWidth,linecolor=GridColor](-6.000000,1.500000)(-1.608000,1.500000)

{ \footnotesize 
\rput[t](0.000000,-0.124344){$0$}
\rput[t](50.000000,-0.124344){$50$}
\rput[t](100.000000,-0.124344){$100$}
\rput[t](150.000000,-0.124344){$150$}
\rput[t](200.000000,-0.124344){$200$}
\rput[t](250.000000,-0.124344){$250$}
\rput[t](300.000000,-0.124344){$300$}
\rput[t](350.000000,-0.124344){$350$}
\rput[r](-10.392000,0.000000){$0$}
\rput[r](-10.392000,0.500000){$0.5$}
\rput[r](-10.392000,1.000000){$1$}
\rput[r](-10.392000,1.500000){$1.5$}
} 

\psframe[linewidth=\AxesLineWidth,dimen=middle](-6.000000,-0.100000)(360.000000,1.500000)

{ \small 
\rput[b](177.000000,-0.277778){
\begin{tabular}{c}
$\delta$\\
\end{tabular}
}

\rput[t]{90}(-46.479263,0.700000){
\begin{tabular}{c}
$h_2$\\
\end{tabular}
}
} 

\newrgbcolor{color212.1261}{0  0  1}
\psline[plotstyle=line,linejoin=1,linestyle=solid,linewidth=\LineWidth,linecolor=color212.1261]
(64.062939,1.500000)(67.881476,1.384325)
\psline[plotstyle=line,linejoin=1,linestyle=solid,linewidth=\LineWidth,linecolor=color212.1261]
(353.132119,1.499988)(353.132136,1.500000)
\psline[plotstyle=line,linejoin=1,linestyle=solid,linewidth=\LineWidth,linecolor=color212.1261]
(67.881476,1.384325)(73.243491,1.246408)(78.304259,1.133568)(83.088486,1.039534)(87.618252,0.959967)
(91.913344,0.891768)(95.991550,0.832662)(99.868905,0.780945)(103.559902,0.735312)(107.077671,0.694750)
(110.434138,0.658458)(113.640158,0.625795)(116.705634,0.596244)(119.639618,0.569379)(122.450402,0.544850)
(125.145593,0.522366)(127.732187,0.501680)(130.216624,0.482586)(132.604848,0.464907)(134.902347,0.448490)
(137.114201,0.433206)(139.245119,0.418941)(141.299469,0.405597)(143.281311,0.393086)(145.194423,0.381334)
(147.042325,0.370274)(148.828301,0.359845)(150.555416,0.349996)(153.844357,0.331854)(156.929972,0.315527)
(159.830585,0.300755)(162.562383,0.287326)(166.374508,0.269327)(169.880704,0.253488)(173.116465,0.239443)
(177.061702,0.223018)(180.639321,0.208749)(184.668579,0.193345)(188.955159,0.177688)(193.331296,0.162430)
(197.663674,0.148002)(202.282089,0.133318)(207.267014,0.118228)(212.506318,0.103166)(217.778955,0.088790)
(223.156840,0.074904)(228.517788,0.061831)(233.753510,0.049836)(238.473961,0.039767)(242.513860,0.031876)
(245.830232,0.026119)(248.566867,0.022098)(250.944638,0.019355)(253.161779,0.017576)(255.411025,0.016584)
(257.904119,0.016335)(260.928020,0.016925)(264.902587,0.018628)(270.329128,0.021892)(277.126746,0.026860)
(284.120356,0.032763)(290.457862,0.038868)(296.008085,0.044962)(300.858351,0.051036)(305.118300,0.057120)
(308.885164,0.063253)(312.233549,0.069460)(315.225649,0.075764)(317.910824,0.082179)(320.332389,0.088724)
(322.523731,0.095407)(324.511661,0.102232)(326.323495,0.109213)(327.979702,0.116359)(329.497415,0.123671)
(330.890352,0.131147)(332.172418,0.138792)(333.357682,0.146626)(334.452999,0.154632)(335.468992,0.162827)
(336.412688,0.171208)(337.291175,0.179782)(338.107903,0.188522)(338.870057,0.197447)(339.581162,0.206541)
(340.248479,0.215845)(340.875594,0.225361)(341.462398,0.235036)(342.016227,0.244942)(342.537000,0.255033)
(343.024645,0.265249)(343.486544,0.275697)(343.922646,0.286333)(344.336634,0.297205)(344.728468,0.308275)
(345.098111,0.319491)(345.449264,0.330920)(345.781900,0.342520)(346.095992,0.354237)(346.395258,0.366167)
(346.679677,0.378268)(346.952977,0.390669)(347.211399,0.403162)(347.458674,0.415884)(347.694789,0.428800)
(347.919734,0.441871)(348.137249,0.455287)(348.343575,0.468786)(348.542454,0.482575)(348.730128,0.496350)
(348.910341,0.510336)(349.083089,0.524502)(349.248365,0.538810)(349.409922,0.553569)(349.563999,0.568417)
(349.710591,0.583302)(349.853453,0.598576)(349.988823,0.613806)(350.120457,0.629376)(350.248353,0.645280)
(350.372510,0.661506)(350.489162,0.677513)(350.602070,0.693761)(350.711230,0.710228)(350.816643,0.726889)
(350.918305,0.743714)(351.016216,0.760672)(351.110374,0.777726)(351.204545,0.795564)(351.294960,0.813478)
(351.381619,0.831422)(351.464519,0.849345)(351.547430,0.868061)(351.626581,0.886714)(351.701971,0.905242)
(351.777368,0.924563)(351.849004,0.943701)(351.916876,0.962579)(351.984754,0.982230)(352.048867,1.001542)
(352.112986,1.021632)(352.173338,1.041293)(352.233696,1.061727)(352.294059,1.082982)(352.350653,1.103698)
(352.407252,1.125225)(352.460082,1.146090)(352.512916,1.167747)(352.565754,1.190239)(352.614821,1.211917)
(352.663892,1.234402)(352.712966,1.257739)(352.758267,1.280080)(352.803572,1.303231)(352.848880,1.327238)
(352.890415,1.350036)(352.931952,1.373633)(352.973491,1.398072)(353.015033,1.423398)(353.052800,1.447233)
(353.090570,1.471883)(353.128341,1.497388)(353.132119,1.499988)

\newrgbcolor{color213.1256}{0  0  0}
\psline[plotstyle=line,linejoin=1,linestyle=dashed,linewidth=\LineWidth,linecolor=color213.1256]
(24.467496,1.500000)(25.320757,1.431280)
\psline[plotstyle=line,linejoin=1,linestyle=dashed,linewidth=\LineWidth,linecolor=color213.1256]
(155.735703,1.499470)(155.736473,1.500000)
\psline[plotstyle=line,linejoin=1,linestyle=dashed,linewidth=\LineWidth,linecolor=color213.1256]
(25.320757,1.431280)(26.624301,1.336775)(27.880215,1.254085)(29.091067,1.181123)(30.259243,1.116268)
(31.386964,1.058241)(32.476299,1.006016)(33.529178,0.958766)(34.547405,0.915812)(35.532665,0.876593)
(36.486537,0.840643)(37.410500,0.807569)(38.305942,0.777040)(39.174167,0.748772)(40.016398,0.722524)
(40.833790,0.698086)(41.627427,0.675278)(42.398333,0.653941)(43.147472,0.633939)(43.875756,0.615148)
(44.584046,0.597464)(45.273156,0.580790)(45.943858,0.565043)(46.596880,0.550147)(47.232916,0.536035)
(47.852621,0.522647)(48.456618,0.509929)(49.045499,0.497831)(49.619826,0.486310)(50.180133,0.475325)
(50.726930,0.464839)(51.260701,0.454820)(52.290991,0.436060)(53.274443,0.418833)(54.214192,0.402957)
(55.113100,0.388281)(55.973782,0.374672)(57.198324,0.356019)(58.349491,0.339203)(59.433719,0.323964)
(60.456707,0.310091)(61.734119,0.293419)(62.922167,0.278527)(64.295327,0.262009)(65.558230,0.247436)
(66.946135,0.232066)(68.412624,0.216516)(69.916444,0.201265)(71.579713,0.185165)(73.315257,0.169174)
(75.052593,0.153945)(76.836567,0.139071)(78.664463,0.124595)(80.565745,0.110324)(82.452255,0.096937)
(84.326061,0.084404)(86.166378,0.072867)(87.895932,0.062793)(89.455107,0.054461)(90.845086,0.047778)
(92.085559,0.042562)(93.206538,0.038606)(94.255371,0.035676)(95.280197,0.033602)(96.333747,0.032284)
(97.471161,0.031697)(98.770322,0.031886)(100.336106,0.032998)(102.315960,0.035310)(104.878146,0.039213)
(108.055593,0.044937)(111.525526,0.052009)(114.877023,0.059616)(117.932596,0.067312)(120.675857,0.074980)
(123.133949,0.082608)(125.347873,0.090237)(127.348643,0.097893)(129.163892,0.105601)(130.817698,0.113388)
(132.328047,0.121263)(133.713562,0.129252)(134.988752,0.137372)(136.163845,0.145622)(137.249365,0.154011)
(138.256105,0.162561)(139.190339,0.171267)(140.058490,0.180127)(140.867117,0.189150)(141.620498,0.198327)
(142.325399,0.207683)(142.986273,0.217228)(143.605225,0.226940)(144.186819,0.236839)(144.733249,0.246913)
(145.246740,0.257152)(145.731975,0.267603)(146.188813,0.278215)(146.619557,0.288992)(147.028970,0.300011)
(147.414521,0.311161)(147.781006,0.322534)(148.128354,0.334088)(148.458949,0.345862)(148.772735,0.357816)
(149.069658,0.369900)(149.352127,0.382165)(149.622559,0.394683)(149.880922,0.407421)(150.127186,0.420344)
(150.361319,0.433404)(150.585763,0.446702)(150.800494,0.460204)(151.005493,0.473873)(151.200739,0.487662)
(151.388687,0.501710)(151.569325,0.515992)(151.742638,0.530477)(151.908615,0.545129)(152.067244,0.559906)
(152.220993,0.575010)(152.367374,0.590166)(152.508858,0.605591)(152.645439,0.621265)(152.777110,0.637162)
(152.903866,0.653252)(153.025699,0.669502)(153.142605,0.685871)(153.254577,0.702319)(153.364100,0.719187)
(153.468681,0.736068)(153.570805,0.753335)(153.667979,0.770534)(153.762690,0.788071)(153.854937,0.805936)
(153.944716,0.824120)(154.029532,0.842068)(154.111875,0.860258)(154.191743,0.878669)(154.269135,0.897278)
(154.344049,0.916058)(154.416483,0.934980)(154.488933,0.954706)(154.558901,0.974563)(154.626386,0.994515)
(154.691384,1.014522)(154.753895,1.034538)(154.816419,1.055366)(154.876453,1.076170)(154.933997,1.096897)
(154.991552,1.118444)(155.046614,1.139866)(155.101685,1.162130)(155.154262,1.184215)(155.204344,1.206047)
(155.254434,1.228705)(155.302027,1.251038)(155.349627,1.274203)(155.394728,1.296959)(155.439836,1.320547)
(155.484950,1.345015)(155.527564,1.368975)(155.570184,1.393810)(155.610301,1.418026)(155.650424,1.443103)
(155.690553,1.469089)(155.728177,1.494319)(155.735703,1.499470)

\newrgbcolor{color214.1256}{0  0  1}
\psline[plotstyle=line,linejoin=1,linestyle=dashed,dash=2pt 3pt,linewidth=\LineWidth,linecolor=color214.1256]
(4.521445,1.500000)(4.967420,1.318789)
\psline[plotstyle=line,linejoin=1,linestyle=dashed,dash=2pt 3pt,linewidth=\LineWidth,linecolor=color214.1256]
(37.278912,1.498609)(37.280969,1.500000)
\psline[plotstyle=line,linejoin=1,linestyle=dashed,dash=2pt 3pt,linewidth=\LineWidth,linecolor=color214.1256]
(4.967420,1.318789)(5.429505,1.164563)(5.853102,1.044620)(6.242863,0.948676)(6.602723,0.870185)
(6.936028,0.804784)(7.245643,0.749453)(7.534036,0.702033)(7.803342,0.660942)(8.055420,0.624994)
(8.291896,0.593280)(8.514196,0.565096)(8.723578,0.539884)(8.921155,0.517198)(9.107918,0.496678)
(9.284746,0.478027)(9.452429,0.461002)(9.611674,0.445401)(9.763116,0.431051)(9.907330,0.417809)
(10.044834,0.405551)(10.176098,0.394173)(10.301550,0.383583)(10.421578,0.373702)(10.536539,0.364461)
(10.752529,0.347670)(10.951809,0.332809)(11.136313,0.319565)(11.388907,0.302200)(11.616517,0.287280)
(11.822843,0.274325)(12.069945,0.259497)(12.290164,0.246896)(12.534295,0.233584)(12.790166,0.220364)
(13.048153,0.207788)(13.301397,0.196184)(13.545529,0.185708)(13.799604,0.175574)(14.035234,0.166911)
(14.269769,0.159032)(14.498867,0.152098)(14.720628,0.146145)(14.934791,0.141136)(15.150946,0.136838)
(15.367555,0.133297)(15.591583,0.130419)(15.821900,0.128241)(16.070595,0.126697)(16.341796,0.125843)
(16.639873,0.125737)(16.979704,0.126463)(17.370433,0.128151)(17.831399,0.131001)(18.382349,0.135276)
(19.039789,0.141246)(19.798996,0.148996)(20.623723,0.158240)(21.466801,0.168487)(22.287307,0.179240)
(23.065532,0.190210)(23.795023,0.201259)(24.478881,0.212386)(25.115310,0.223506)(25.712818,0.234715)
(26.270027,0.245936)(26.790858,0.257185)(27.284766,0.268625)(27.745654,0.280067)(28.183380,0.291705)
(28.591972,0.303332)(28.981635,0.315190)(29.346498,0.327055)(29.697062,0.339228)(30.027529,0.351475)
(30.343101,0.363947)(30.637954,0.376360)(30.922980,0.389132)(31.192405,0.401975)(31.446019,0.414816)
(31.689291,0.427889)(31.922096,0.441158)(32.144313,0.454583)(32.355828,0.468114)(32.562274,0.482096)
(32.757839,0.496111)(32.942421,0.510089)(33.121718,0.524425)(33.295684,0.539111)(33.458453,0.553605)
(33.615778,0.568363)(33.767616,0.583364)(33.913931,0.598580)(34.054684,0.613980)(34.189841,0.629527)
(34.319367,0.645179)(34.443229,0.660890)(34.561397,0.676607)(34.679763,0.693117)(34.792395,0.709592)
(34.899265,0.725966)(35.006295,0.743140)(35.107527,0.760150)(35.202936,0.776913)(35.298472,0.794457)
(35.388152,0.811665)(35.477944,0.829662)(35.561852,0.847219)(35.645856,0.865562)(35.723948,0.883344)
(35.802124,0.901897)(35.874361,0.919751)(35.946670,0.938350)(36.019050,0.957742)(36.085462,0.976258)
(36.151934,0.995528)(36.218467,1.015598)(36.279003,1.034581)(36.339589,1.054309)(36.400225,1.074827)
(36.460911,1.096185)(36.515570,1.116167)(36.570271,1.136914)(36.625011,1.158469)(36.679792,1.180880)
(36.728520,1.201564)(36.777280,1.223005)(36.826072,1.245247)(36.874896,1.268335)(36.917643,1.289271)
(36.960415,1.310928)(37.003211,1.333344)(37.046031,1.356560)(37.088876,1.380620)(37.131745,1.405571)
(37.168510,1.427705)(37.205293,1.450565)(37.242093,1.474186)(37.278912,1.498609)

\newrgbcolor{color215.1256}{0  0  1}
\psline[plotstyle=line,linejoin=1,linestyle=dashed,dash=2pt 3pt,linewidth=\LineWidth,linecolor=color215.1256]
(10.333291,-0.000000)(10.839010,0.008482)(11.357385,0.016369)(11.874874,0.023448)(12.374474,0.029500)
(12.840463,0.034375)(13.264461,0.038051)(13.645261,0.040595)(13.987915,0.042125)(14.298938,0.042751)
(14.586252,0.042561)(14.857572,0.041603)(15.120536,0.039888)(15.382875,0.037377)(15.653480,0.033973)
(15.943264,0.029497)(16.269312,0.023606)(16.661272,0.015639)(17.208742,0.003540)(17.361641,0.000051)

\newrgbcolor{color216.1256}{1  0  0}
\psline[plotstyle=line,linejoin=1,linestyle=dashed,dash=3pt 2pt 1pt 2pt,linewidth=\LineWidth,linecolor=color216.1256]
(-5.000000,0.000000)(355.000000,0.000000)

\newrgbcolor{color217.1256}{1  0  0}
\psline[plotstyle=line,linejoin=1,linestyle=dashed,dash=3pt 2pt 1pt 2pt,linewidth=\LineWidth,linecolor=color217.1256]
(0.000000,-0.100000)(0.000000,1.500000)

{ \small 
\rput(288.500000,1.11111){%
\psshadowbox[framesep=0pt,linewidth=\AxesLineWidth]{\psframebox*{\begin{tabular}{l}
\Rnode{a1}{\hspace*{0.0ex}} \hspace*{0.7cm} \Rnode{a2}{~~$n=3$} \\
\Rnode{a3}{\hspace*{0.0ex}} \hspace*{0.7cm} \Rnode{a4}{~~$n=2$} \\
\Rnode{a5}{\hspace*{0.0ex}} \hspace*{0.7cm} \Rnode{a6}{~~$n=1$} \\
\end{tabular}}
\ncline[linestyle=solid,linewidth=\LineWidth,linecolor=color212.1261]{a1}{a2}
\ncline[linestyle=dashed,linewidth=\LineWidth,linecolor=color213.1256]{a3}{a4}
\ncline[linestyle=dashed,dash=2pt 3pt,linewidth=\LineWidth,linecolor=color214.1256]{a5}{a6}
}%
}%
} 
\end{pspicture}
}%
\caption{The curves of  $(\delta,\,h_2)$, $\delta>0,\,h_2>0$ where    $\xi=1.62$,   $q=0.8$ with $\xi>2q$.  The upper lobes are determined by (\ref{lemma-2-3-1}) with $\beta\in (\beta_n^*,\,2n\pi)$, $n=1,\,2,\,3$, respectively and the lower one with $\beta\in (\gamma_n^*,\,\tilde{\gamma}_n)$. The   connected region near the origin $(0,\,0)$ without crossing the lobes, or the vertical line $\delta=0$ or the horizontal line $h_2=0$ is the stability region.}
\label{fig-6}
\end{center}
\end{figure}
\section{Concluding Remarks}\label{Concluding-remarks}
Starting from the model of turning processes with state-dependent delay, we investigated two spindle control strategies one of which contains a state-dependent delay and the other one an instantaneous PD (proportional-derivative) control. Using a time domain transformation, we show that the controlled model can be reduced into models with constant and distributed delays. Since whose characteristic equations are transcendental equations containing both of the terms $\lambda$ and $e^\lambda$, the analytical description of the stability region in the space of the   parameters $\delta$ and $h$ composite of the intrinsic parameters of the turning processes becomes complicated.
The analysis shows that   feedback spindle control with state-dependent delay still can stabilize the equilibrium state of the turning processes, even though the simulated stability region does not show  improvement of the stability region over the PD control.  This means that the  practically realizable  delayed feedback control is feasible and is able to achieve stability. When the parameters are beyond the stability region, we refer to \cite{HU-JDE-1} for analysis of global Hopf bifurcation of turning processes with threshold-type state-dependent delay.

 We remark that with the delayed feedback control, $h_1=K_1p^{q-1}\left(1-\frac{p}{k_r}\right)$  may not   be positive as in in Figure~\ref{fig-4}. However, it remains unclear whether or not a positive $h_1$ will remove the bound of the stability region in the $\delta$-direction and the lead to a stability region unbounded in both of the $\delta$ and $h_1$ directions. We also remark that we may improve the stability region by investigating nonlinear type   spindle speed control with state-dependent delay, namely how to find $\Omega(t)=r(x(t-\tau(t)))$, where $r: \mathbb{R}\rightarrow\mathbb{R}$ is a nonlinear function instead of the linear one investigated in the current work, in order to improve the stability region. We leave these problems for future work.


\begin{thebibliography}{10}

\bibitem{Altintas}
{\sc Altintas, Y., and Budak, E.}
\newblock Analytical prediction of stability lobes in milling.
\newblock {\em CIRP Annals - Manufacturing Technology 44}, 1 (1995), 357--362.



\bibitem{BST}
{\sc Bachrathy, D., St\'ep\'an, G. and Turi, J.}
\newblock State dependent regenerative effect in milling processes.
\newblock {\em J. Comput. Nonlinear Dynam. 6}, 4 (2011), 
\newblock Article Number: 041002.  doi:10.1115/1.4003624.

\bibitem{Balach}
{\sc Balachandran, B., and Zhao, M.~X.}
\newblock A mechanics based model for study of dynamics of milling operations.
\newblock {\em Meccanica 2\/} (2000), 89--109.


\bibitem{HU-JDE-1}
{\sc Balanov, Z., Hu, Q. and Krawcewicz, W.} 
\newblock Global Hopf bifurcation of differential equations with threshold-type state-dependent delay, 
\newblock{\em J. Differential Equations}, 
257 (2014),  2622--2670.  
  






%
%




\bibitem{Gilsinn}
{\sc Gilsinn, D.~E.}
\newblock Estimating critical hopf bifurcation parameters for a second-order
  delay differential equation with application to machine tool chatter.
\newblock {\em Nonlinear Dynam. 30}, 2 (2002), 103--154.



\bibitem{HKWW}
{\sc Hartung, F., Krisztin, T., Walther, H.-O., and Wu, J.}
\newblock Chapter 5: {F}unctional {D}ifferential {E}quations with
  {S}tate-{D}ependent {D}elays: {T}heory and {A}pplications.
\newblock In {\em Handbook of Differential Equations: Ordinary Differential
  Equations}, P.~D. {A. Ca{\~N}ada} and A.~Fonda, Eds., vol.~3. North-Holland,
  2006, pp.~435--545.
  
 


\bibitem{HKT}
{\sc Hu, Q., Krawcewicz, W., and Turi, J.}
\newblock Stabilization in a state-dependent model of turning processes.
\newblock {\em  SIAM J. Appl. Math.}, 72 (2011), 1--24.




\bibitem{HKT-1}
{\sc Hu, Q., Krawcewicz, W., and Turi, J.}
\newblock Global stability lobes of a state-dependent model of turning
  processes.
\newblock {\em SIAM Journal on Applied Mathematics 72\/} (2012), 1383--1405.




\bibitem{Insperger-2}
{\sc Insperger, T., and St{\'e}p{\'a}n, G.}
\newblock Stability analysis of turning with periodic spindle speed maching.
\newblock {\em J. Manuf. Sci. Eng. 122}, 3 (2000), 391--397.

\bibitem{Turi-1}
{\sc Insperger, T., St{\'e}p{\'a}n, G., and Turi, J.}
\newblock State-dependent delay in regenerative turning processes.
\newblock {\em Nonlinear Dyn. 47\/} (2007), 275--283.






\bibitem{Ismail}
{\sc Ismail, F., and Soliman, E.}
\newblock A new method for the identification of stability lobes in machining.
\newblock {\em Int. J. Mach. Tools Manufacture 37}, 6 (1997), 763--774.




\bibitem{KT}
{\sc Koenigsberger, F., and Tlusty, J.}
\newblock {\em Machine Tool Structures}, vol.~1.
\newblock Pergamon Press, 1970.

\bibitem{kw}
{\sc Krawcewicz, W., and Wu, J.}
\newblock {\em Theory of Degrees with Applications to Bifurcations and
  Differential Equations}.
\newblock Canadian Mathematical Society Series of Monographs and Advanced
  Texts. Johns Wiley \& Sons, New York, 1997.
  
 


\bibitem{Long-1}
{\sc Long, X., and Balachandran, B.}
\newblock Stability of up-milling and down-milling operations with variable
  spindle speeds.
\newblock {\em J. Vibration and Control 16}, 7--8 (2010), 1151 -- 1168.

\bibitem{Pak}
{\sc Pakdemirli, M., and Ulsoy, A.~G.}
\newblock Perturbation analysis of spindle speed variation in machine tool
  chatter.
\newblock {\em J. Vibration and Control 3\/} (1996), 261--278.



\bibitem{Sexton}
{\sc Sexton, J.~S., Milne, R.~D., and Stone, B.~J.}
\newblock A stability analysis of single point machining with varying spindle
  control.
\newblock {\em Appl. Math. Modeling 1\/} (1977), 310 -- 318.



\bibitem{Smith-Tlusky}
{\sc Smith, S., and Tlusty, J.}
\newblock Update on high-speed milling dynamics.
\newblock {\em ASME Journal of Engineering for Industry 112\/} (1990), 142 --
  149.

  

\bibitem{Stepan}
{\sc St{\'e}p{\'a}n, G.}
\newblock {\em Retarded Dynamical Systems: Stability and Characteristic
  Functions}.
\newblock Longman Sci. tech., UK, 1989.



\bibitem{Stone-Sue}
{\sc Stone, E., and Campbell, S.}
\newblock Stability and bifurcation analysis of a nonlinear dde model for
  drilling.
\newblock {\em J. Nonlinear Science 14}, 1 (2004), 27--57.

\bibitem{Taylor}
{\sc Taylor, F.~W.}
\newblock On the art of cutting metals.
\newblock {\em Oscillation and Dynamics in Delay Equations, Contemporary
  Mathematics\/} (1907).

  

\bibitem{Tobias}
{\sc Tobias, S.~A.}
\newblock {\em Machine tool vibration}.
\newblock Blackie, London, 1965.

\bibitem{Tobias-0}
{\sc Tobias, S.~A., and Fishwick, W.}
\newblock Theory of regenerative machine tool chatter.
\newblock {\em The Engineer, London 205}, 1 (1958), 199 -- 203.

\end{thebibliography}
\end{document}